\pdfoutput=1
\RequirePackage{ifpdf}
\ifpdf 
\documentclass[pdftex]{sigma}
\else
\documentclass{sigma}
\fi

\usepackage{pgf,tikz,tikz-cd}
\usetikzlibrary{shapes.geometric, arrows}
\usetikzlibrary{arrows}
\usetikzlibrary{intersections}
\usetikzlibrary{shapes.misc}
\usetikzlibrary{patterns}
\usetikzlibrary{calc,patterns,angles,quotes}
\usetikzlibrary{positioning}
\usetikzlibrary{decorations.markings,decorations.pathreplacing}

\tikzset{cross/.style={cross out, draw,
 minimum size=2*(#1-\pgflinewidth),
 inner sep=0pt, outer sep=0pt}}

\tikzset{%
 add/.style args={#1 and #2}{
 to path={%
 ($(\tikztostart)!-#1!(\tikztotarget)$)--($(\tikztotarget)!-#2!(\tikztostart)$)%
 \tikztonodes},add/.default={.2 and .2}}
}

\xdefinecolor{darkgreen}{RGB}{0, 100, 0}

\usepackage[all]{xy}

\numberwithin{equation}{section}

\newtheorem{Theorem}{Theorem}[section]
\newtheorem{Corollary}[Theorem]{Corollary}
\newtheorem{Lemma}[Theorem]{Lemma}
\newtheorem{Proposition}[Theorem]{Proposition}
 { \theoremstyle{definition}
\newtheorem{Definition}[Theorem]{Definition}
\newtheorem{Example}[Theorem]{Example}
\newtheorem{Construction}[Theorem]{Construction}
\newtheorem{Remark}[Theorem]{Remark} }

\begin{document}
\allowdisplaybreaks

\newcommand{\arXivNumber}{1901.04166}

\renewcommand{\thefootnote}{}

\renewcommand{\PaperNumber}{013}

\FirstPageHeading

\ShortArticleName{Cluster Structures and Subfans in Scattering Diagrams}

\ArticleName{Cluster Structures and Subfans\\ in Scattering Diagrams\footnote{This paper is a~contribution to the Special Issue on Cluster Algebras. The full collection is available at \href{https://www.emis.de/journals/SIGMA/cluster-algebras.html}{https://www.emis.de/journals/SIGMA/cluster-algebras.html}}}

\Author{Yan ZHOU}

\AuthorNameForHeading{Y.~Zhou}

\Address{Beijing International Center for Mathematical Research, Peking University, China}
\Email{\href{mailto:y-chou@pku.edu.cn}{y-chou@pku.edu.cn}}
\URLaddress{\url{https://sites.google.com/view/yan-zhou/home}}

\ArticleDates{Received July 24, 2019, in final form March 01, 2020; Published online March 11, 2020}

\Abstract{We give more precise statements of Fock--Goncharov duality conjecture for cluster varieties parametrizing ${\rm SL}_{2}/{\rm PGL}_{2}$-local systems on the once punctured torus. Then we prove these statements. Along the way, using distinct subfans in the scattering diagram, we produce an example of a cluster variety with two non-equivalent cluster structures. To overcome the technical difficulty of infinite non-cluster wall-crossing in the scattering diagram, we introduce quiver folding into the machinery of scattering diagrams and give a quotient construction of scattering diagrams.}


\Classification{13F60; 14J32; 14J33; 14N35}

\renewcommand{\thefootnote}{\arabic{footnote}}
\setcounter{footnote}{0}

\section{Introduction}

\subsection[Fock--Goncharov duality conjecture for $\mathcal{A}_{\mathbb{T}_{1}^{2}}$, $\mathcal{X}_{\mathbb{T}_{1}^{2}}$ and $\mathcal{A}_{{\rm prin},\mathbb{T}_{1}^{2}}$]{Fock--Goncharov duality conjecture for $\boldsymbol{\mathcal{A}_{\mathbb{T}_{1}^{2}}}$, $\boldsymbol{\mathcal{X}_{\mathbb{T}_{1}^{2}}}$ and $\boldsymbol{\mathcal{A}_{{\rm prin},\mathbb{T}_{1}^{2}}}$}

In \cite{FG06}, Fock and Goncharov introduced cluster coordinates
on moduli spaces of local systems on Riemann surfaces in Teichmüller
theory. In this paper, the main examples we investigate are cluster
varieties with seeds coming from ideal triangulations of the once
punctured torus $\mathbb{T}_{1}^{2}$. These cluster varieties parametrize
${\rm SL}_{2}/{\rm PGL}_{2}$-local systems with decorations or framing
on $\mathbb{T}_{1}^{2}$, as constructed in~\cite{FG06}. Any ideal
triangulation of $\mathbb{T}_{1}^{2}$ yields the same isomorphic
class of seeds that can be encoded by Markov quiver:

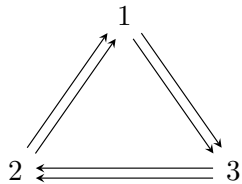
\begin{figure}[h!]\centering
\begin{tikzcd}[arrow style=tikz,>=stealth,row sep=4em]
& 1 \arrow[dr,shift left=.4ex] \arrow[dr,shift right=.4ex]
\\
2 \arrow[ur,shift left=.4ex] \arrow[ur,shift right=.4ex]
&& 3 \arrow[ll, shift left=.4ex] \arrow[ll, shift right=.4ex]
\end{tikzcd}
\caption{Markov quiver.}\label{Markov}
\end{figure}

Let $\mathcal{A}_{\mathbb{T}_{1}^{2}}$, $\mathcal{A}_{{\rm prin,\mathbb{T}_{1}^{2}}}$
and $\mathcal{X}_{\mathbb{T}_{1}^{2}}$ be cluster varieties associated
to $\mathbb{T}_{1}^{2}$ of type $\mathcal{A}$, $\mathcal{A}_{{\rm prin}}$
and $\mathcal{X}$ respectively. Using techniques of \emph{scattering
diagrams }(cf.~\cite{GS11,KS06}) and \emph{broken lines} (cf.~\mbox{\cite{CPS, G09}})
coming from mirror symmetry, Gross, Hacking, Keel and Kontsevich in~\cite{GHKK} construct canonical bases for cluster varieties whose elements, \emph{theta functions}, are generalization of \emph{cluster
variables}. Specifically, given a cluster variety $V$ over $\mathbb{C}$,
let $V^{\vee}$ be its \emph{Fock--Goncharov dual} variety and~$V^{\vee}\big(\mathbb{Z}^{T}\big)$
the set of integer tropical points of~$V^{\vee}$. For each~$q$ in
$V^{\vee}\big(\mathbb{Z}^{T}\big)$, there is a~theta function $\vartheta_{q}$
corresponding to~$q$. Write ${\rm can}(V)$ for the $\mathbb{C}$-vector
space with basis the set of theta functions, i.e.,
\[
{\rm can}(V)=\bigoplus_{q\in V^{\vee}(\mathbb{Z}^{T})}\mathbb{C}\cdot\vartheta_{q}.
\]
Theta functions can be formal power series. However, there is a canonically
defined subset $\Theta(V)$ in $V^{\vee}\big(\mathbb{Z}^{T}\big)$ called \emph{theta set} that indexes a vector subspace ${\rm mid}(V)$ of ${\rm can}(V)$
and ${\rm mid}(V)$ has a~$\mathbb{C}$-algebra structure. Moreover,
${\rm mid}(V)$ can always be identified as a subring of $\Gamma(V,\mathcal{O}_{V})$
if $V$ is of type $\mathcal{X}$ or $\mathcal{A}_{{\rm prin}}$.
In the case of cluster varieties associated to $\mathbb{T}_{1}^{2}$,
it is known that scattering diagrams on $\mathcal{A}_{\mathbb{T}_{1}^{2}}^{\vee}\big(\mathbb{R}^{T}\big)$
and $\mathcal{A}_{{\rm prin},\mathbb{T}_{1}^{2}}^{\vee}\big(\mathbb{R}^{T}\big)$
have two distinct subfan structures, one called \emph{cluster complex}
consisting of positive chambers $\{\mathcal{C}_{\overline{{\bf s}}_{v}}^{+}\}_{v\in\mathfrak{T}_{\overline{{\bf s}}}}$
and the other consisting of negative chambers~$\{\mathcal{C}_{{\bf \overline{{\bf s}}}_{v}}^{-}\}_{v\in\mathfrak{T}_{\overline{{\bf s}}}}$.
Here $\mathfrak{T}_{\overline{{\bf s}}}$ is the infinite tree whose
vertices parametrizing all the seeds mutationally equivalent to an
initial seed $\overline{{\bf s}}$ coming from an ideal triangulation
of $\mathbb{T}_{1}^{2}$. Because of the infinite wall-crossings from
$\mathcal{C}_{\overline{{\bf s}}}^{-}$ to $\mathcal{C}_{\overline{{\bf s}}}^{+}$
in the scattering diagram, it is not clear whether theta functions
parametrized by integer tropical points in $\mathcal{C}_{\overline{{\bf s}}}^{-}$
are formal sums or polynomials. Given a cluster variety $V$, we say
that the \emph{full Fock--Goncharov conjecture} holds for $V$ if the
following holds:
\[
\Gamma(V,\mathcal{O}_{V})={\rm mid}(V)={\rm can}(V)=\bigoplus_{q\in V^{\vee}\big(\mathbb{Z}^{T}\big)}\mathbb{C}\cdot\vartheta_{q}.
\]
 In this paper, one of our main results is the following:
\begin{Theorem}[cf.\ Theorem \ref{thm:5.1}, Corollary~\ref{cor:5.2}] The full Fock--Goncharov conjecture holds for~$\mathcal{A}_{{\rm prin},\mathbb{T}_{1}^{2}}$ and~$\mathcal{X}_{\mathbb{T}_{1}^{2}}$.
\end{Theorem}

It is unnatural to build up $\mathcal{A}_{{\rm prin},\mathbb{T}_{1}^{2}}$
using only the atlas $\mathcal{T}^{+}$ of torus charts parametrized
by positive chambers and exclude the atlas $\mathcal{T}^{-}$ of those
parametrized by negative chambers. If we use torus charts in both
$\mathcal{T}^{+}$ and $\mathcal{T}^{-}$, we will get a new variety
$\tilde{\mathcal{A}}_{{\rm prin},\mathbb{T}_{1}^{2}}$ that contains
the original cluster variety $\mathcal{A}_{{\rm prin},\mathbb{T}_{1}^{2}}$.
Moreover, $\tilde{\mathcal{A}}_{{\rm prin},\mathbb{T}_{1}^{2}}$ and
$\mathcal{A}_{{\rm prin},\mathbb{T}_{1}^{2}}$ have the same ring
of regular functions. The atlas $\mathcal{T}^{-}$ also provides for
$\tilde{\mathcal{A}}_{{\rm prin},\mathbb{T}_{1}^{2}}$ \emph{an atlas
of cluster torus charts} (cf.\ Definition \ref{def:5.3}) that are
not mutationally equivalent to those in~$\mathcal{T}^{+}$. Therefore
$\mathcal{T}^{-}$ equips $\tilde{\mathcal{A}}_{{\rm prin},\mathbb{T}_{1}^{2}}$
with a cluster structure non-equivalent to the one provided by $\mathcal{T}^{+}$.
See Definition~\ref{def:5.4} for what we precisely mean by two cluster
structures are non-equivalent. Hence we have the following theorem:
\begin{Theorem}[cf.\ Theorem~\ref{thm:5.7}]\label{thm:1.2} The variety $\tilde{\mathcal{A}}_{{\rm prin},\mathbb{T}_{1}^{2}}$ is a log Calabi--Yau variety with $2$ non-equivalent cluster structures,
each of which corresponds to one of the two distinct simplicial fans in the scattering diagram.
\end{Theorem}

The full Fock--Goncharov conjecture says that $\mathcal{A}_{\mathbb{T}_{1}^{2}}^{\vee}\big(\mathbb{Z}^{T}\big)$
should parametrize a vector space basis for $\Gamma\big(\mathcal{A}_{\mathbb{T}_{1}^{2}},\mathcal{O}_{\mathcal{A}_{\mathbb{T}_{1}^{2}}}\big)$
and it turns out to be false.
\begin{Theorem}The full Fock--Goncharov conjecture does not hold for $\mathcal{A}_{\mathbb{T}_{1}^{2}}$.
We have
\[
\Theta\big(\mathcal{A}_{\mathbb{T}_{1}^{2}}\big)=\mathcal{A}_{\mathbb{T}_{1}^{2}}^{\vee}\big(\mathbb{Z}^{T}\big)
\]
and therefore,
\[
{\rm can}\big(\mathcal{A}_{\mathbb{T}_{1}^{2}}\big)={\rm mid}\big(\mathcal{A}_{\mathbb{T}_{1}^{2}}\big).
\]
However ${\rm can}\big(\mathcal{A}_{\mathbb{T}_{1}^{2}}\big)$ is strictly contained in $\Gamma\big(\mathcal{A}_{\mathbb{T}_{1}^{2}},\mathcal{O}_{\mathbb{T}_{1}^{2}}\big)$.
\end{Theorem}

It is not surprising that $\mathcal{A}_{\mathbb{T}_{1}^{2}}^{\vee}\big(\mathbb{Z}^{T}\big)$
does not parametrize a vector space basis for the ring of regular
functions on $\mathcal{A}_{\mathbb{T}_{1}^{2}}$, if we take into
consideration that when we build $\mathcal{A}_{\mathbb{T}_{1}^{2}}$,
we use only half of the scattering diagram on $\mathcal{A}_{\mathbb{T}_{1}^{2}}^{\vee}\big(\mathbb{Z}^{T}\big)$.
From the mirror symmetry perspective, it will be more natural to expect
that $\mathcal{A}_{\mathbb{T}_{1}^{2}}^{\vee}\big(\mathbb{Z}^{T}\big)$ should
parametrize a basis for the ring of regular functions of the space
that include torus charts of all chambers in the scattering diagram.
We can run a construction similar to what we have done for $\mathcal{A}_{{\rm prin},\mathbb{T}_{1}^{2}}$
and get a variety $\widetilde{\mathcal{A}}_{\mathbb{T}_{1}^{2}}\supset\mathcal{A}_{\mathbb{T}_{1}^{2}}$
by also including torus charts coming from negative chambers in the
scattering diagram. The complement of the original cluster variety
$\mathcal{A}_{\mathbb{T}_{1}^{2}}$ has codimension $1$ inside $\tilde{\mathcal{A}}_{\mathbb{T}_{1}^{2}}$,
so $\Gamma\big(\tilde{\mathcal{A}}_{\mathbb{T}_{1}^{2}}, \mathcal{O}_{\tilde{\mathcal{A}}_{\mathbb{T}_{1}^{2}}}\big)$ is strictly contained in $\Gamma\big(\mathcal{A}_{\mathbb{T}_{1}^{2}},\mathcal{O}_{\mathbb{T}_{1}^{2}}\big)$.
Moreover, we have the following theorem:
\begin{Theorem}
We have
\[
{\rm can}\big(\mathcal{A}_{\mathbb{T}_{1}^{2}}\big)= \Gamma\big(\tilde{\mathcal{A}}_{\mathbb{T}_{1}^{2}}, \mathcal{O}_{\tilde{\mathcal{A}}_{\mathbb{T}_{1}^{2}}}\big).
\]
So $\mathcal{A}_{\mathbb{T}_{1}^{2}}^{\vee}\big(\mathbb{Z}^{T}\big)$ parametrizes
a vector space basis for the ring of regular functions of the enlarged
variety $\tilde{\mathcal{A}}_{\mathbb{T}_{1}^{2}}$.
\end{Theorem}

The above theorem justifies that $\mathcal{\widetilde{A}}_{\mathbb{T}_{1}^{2}}$
should be the correct mirror dual variety to $\mathcal{A}_{\mathbb{T}_{1}^{2}}^{\vee}$
since it is the space where the duality conjecture becomes true. At
the end of the introduction, we will elaborate more on this point
against the backdrop of the mirror symmetry program.

\subsection[Technical framework: Quiver folding in the context of scattering
diagrams]{Technical framework: Quiver folding in the context\\ of scattering
diagrams}

In Section~\ref{sec:Full-Fock--Goncharov-conjecture}, instead of directly
studying canonical bases for $\mathcal{A}_{\mathbb{T}_{1}^{2}}$,
$\mathcal{A}_{{\rm prin,\mathbb{T}_{1}^{2}}}$ or $\mathcal{X}_{\mathbb{T}_{1}^{2}}$
and their scattering diagrams, we relate them to those of cluster
varieties associated to $4$-punctured sphere~$\mathbb{S}_{4}^{2}$.
In order to do this, we introduce the technique of quiver-folding
in cluster theory into scattering diagrams.

Given a seed ${\bf s}$, if a seed $\overline{{\bf s}}$ can obtained
from ${\bf s}$ through a procedure called \emph{folding},\footnote{In this paper, I only deal with a special class of folding, see Section~\ref{subsec:Application-of-folding}.} then we prove the following relationship between the corresponding scattering diagrams:
\begin{Theorem}[cf.\ Construction \ref{2.1}, Theorem~\ref{thm:2.21}] The scattering
diagram $\mathfrak{D}_{\overline{{\bf s}}}$ can be obtained from $\mathfrak{D}_{{\bf s}}$ via a quotient construction.
\end{Theorem}

The folding procedure can also be applied to theta functions. Given
a cluster variety $V$, let $\Theta(V)\subset V^{\vee}\big(\mathbb{Z}^{T}\big)$
be the subset that parametrize theta functions that are contained
in the ring of regular functions on~$V$. Let $\mathcal{A}_{{\rm prin},{\bf s}}$
and $\mathcal{A}_{{\rm prin},\overline{{\bf s}}}$ be the cluster
$\mathcal{A}_{{\rm prin}}$-variety with initial seed ${\bf s}$ and
$\overline{{\bf s}}$ respectively. Then we prove the following theorem:
\begin{Theorem}[Theorem \ref{thm:3.5}] If $\Theta(\mathcal{A}_{{\rm prin},{\bf s}})=\mathcal{A}_{{\rm prin},{\bf s}}^{\vee}\big(\mathbb{Z}^{T}\big)$ holds, then $\Theta(\mathcal{A}_{{\rm prin},\overline{{\bf s}}})=\allowbreak\mathcal{A}_{{\rm prin},\overline{{\bf s}}}^{\vee}\big(\mathbb{Z}^{T}\big)$.
\end{Theorem}

By the relationship between rank $2$ cluster varieties and log Calabi--Yau
surfaces as proved in~\cite{GHK15}, we could view $\mathcal{X}_{\mathbb{S}_{4}^{2}}$
as the universal family of cubic surfaces and $\mathcal{X}_{\mathbb{T}_{1}^{2}}\subset\mathcal{X}_{\mathbb{S}_{4}^{2}}$
the degenerate subfamily of singular cubic surfaces. For further applications
of the folding technique in cases where $\mathcal{X}$-cluster varieties
correspond to families of log Calabi--Yau surfaces, see~\cite{Z19}.

\subsection[Building $\mathcal{A}_{{\rm prin}}$ using all chambers in the scattering diagram]{Building $\boldsymbol{\mathcal{A}_{{\rm prin}}}$ using all chambers in the scattering diagram}

Let ${\bf s}_{0}$ be an initial seed. Recall that the scattering
diagram $\mathfrak{D}_{{\bf s}_{0}}$ always has two subfans, the
cluster complex $\Delta_{{\bf s}_{0}}^{+}$ that consists of positive
chambers and $\Delta_{{\bf s}_{0}}^{-}$ that consists of negative
chambers. The two complexes either have trivial intersection or coincide.
In~\cite{GHKK}, a positive space $\mathcal{A}_{{\rm prin},{\bf s}_{0}}^{{\rm scat}}$
is built by attaching a copy of algebraic torus to each chamber in
$\Delta_{{\bf s}_{0}}^{+}$ and gluing these tori together using the
wall-crossings between these chambers. The positive space $\mathcal{A}_{{\rm prin},{\bf s}_{0}}^{{\rm scat}}$ is isomorphic to $\mathcal{A}_{{\rm prin},{\bf s}_{0}}$. In Section~\ref{sec:Building--using}, we build a positive space $\mathcal{\tilde{A}}_{{\rm prin},{\bf s}_{0}}^{{\rm scat}}\supset\mathcal{A}_{{\rm prin},{\bf s}_{0}}^{{\rm scat}}$
using all chambers in $\Delta_{{\bf s}_{0}}^{+}\bigcup\Delta_{{\bf s}_{0}}^{-}$
if the path ordered product $\mathfrak{p}_{+,-}^{{\bf s}_{0}}$, or
equivalently the Donaldson--Thomas transformation of $\mathcal{A}_{{\rm prin},{\bf s}_{0}}$
as defined in \cite{GS16}, is rational. Let $\mathcal{A}_{{\rm prin},{\bf s}_{0}}^{{\rm scat},-}\subset\mathcal{\tilde{A}}_{{\rm prin},{\bf s}_{0}}^{{\rm scat}}$
be the subspace built by gluing all tori attached to negative chambers.
We prove that $\mathcal{\tilde{A}}_{{\rm prin},{\bf s}_{0}}^{{\rm scat}}$
is isomorphic to a positive space built by gluing two copies of $\mathcal{A}_{{\rm prin},{\bf s}_{0}}$
together:
\begin{Theorem}
We have an isomorphism of positive spaces:
\[
\tilde{\mathcal{A}}_{{\rm prin,{\bf s}_{0}}}=\mathcal{A}_{{\rm prin{\bf s}_{0}}}^{+}\bigcup\mathcal{A}_{{\rm prin,{\bf s}_{0}}}^{-}\stackrel{\sim}{\longrightarrow}\tilde{\mathcal{A}}_{{\rm prin},{\bf s}_{0}}^{{\rm scat}}.
\]
Here $\mathcal{A}_{{\rm prin},{\bf s}_{0}}^{\pm}$ are two copies
of $\mathcal{A}_{{\rm prin},{\bf s}_{0}}$, the image of $\mathcal{A}_{{\rm prin},{\bf s}_{0}}^{+}$
is $\mathcal{A}_{{\rm prin},{\bf s}_{0}}^{{\rm scat}}$ and the image
of $\mathcal{A}_{{\rm prin},{\bf s}_{0}}^{-}$ is $\mathcal{A}_{{\rm prin},{\bf s}_{0}}^{{\rm scat},-}$.
\end{Theorem}

The construction of $\tilde{\mathcal{A}}_{{\rm prin},\mathbb{T}_{1}^{2}}$ in the beginning is a specific instance of the above theorem. In general, $\mathfrak{D}_{{\bf s}_{0}}$ might have subfans other than $\Delta_{{\bf s}_{0}}^{+}$ and $\Delta_{{\bf s}_{0}}^{-}$. It will be interesting to produce such an example and see if these subfans can yield non equivalent cluster structures as in the case of~$\tilde{\mathcal{A}}_{{\rm prin},\mathbb{T}_{1}^{2}}$.

\subsection{Background in mirror symmetry}

In this subsection, we will recall some background in mirror symmetry and review our main results in the light of mirror symmetry. Given an affine log Calabi--Yau variety $V$ over a ground field~$\mathbb{K}$
of characteristic $0$, inspired by the homological mirror symmetry conjecture, Gross, Hacking and Keel conjectured (Abouzaid, Kontsevich and Soibelman also suggested) that $\Gamma(V,\mathcal{O}_{V})$, the
ring of regular functions on $V$, has a canonical vector space basis
called \emph{theta functions} parametrized by the integer tropical
points $V^{\vee}\big(\mathbb{Z}^{T}\big)$ of the mirror dual variety~$V^{\vee}$.
Moreover, the structure constants of $\Gamma(V,\mathcal{O}_{V})$
as a $\mathbb{K}$-algebra can be computed by counting holomorphic
discs on $V^{\vee}$. For the development of mirror symmetry for log
Calabi--Yau surfaces, see~\cite{A07,GHK11,Yu16,Yu17}.

The subject of cluster algebra is invented to study Lusztig's \emph{dual canonical basis} (cf.~\cite{Lu90,Lu94}) of the ring of regular functions on base affine space $G/N$ where $G$ is a reductive algebraic group and $N$ an unipotent subgroup. As an attempt to find an algorithm
to compute Lusztig's dual canonical basis, Fomin and Zelevinsky discovered
that $\Gamma(G/N,\mathcal{O}_{G/N})$ is a cluster algebra and~$G/N$
a~cluster variety.

Cluster varieties are particular examples of log Calabi--Yau varieties.
Therefore mirror symmetry predicts a geometric way to produce a canonical
basis for cluster varieties as mentioned in the beginning of this
subsection. The first attempt in this vein is~\cite{GHKK}. Let us
briefly review the significance of~\cite{GHKK} towards cluster theory.
First, though scattering diagrams and broken lines in~\cite{GHKK}
are still combinatorial gadgets, mirror symmetry predicts that broken
lines and theta functions should be intrinsic to the enumerative geometry
of the mirror dual varieties and therefore independent of the auxiliary
combinatorial data of cluster algebras. In fact, the technology has
developed after~\cite{GHKK}. In~\cite{KY}, for cluster varieties
whose rings of regular functions are finitely generated and smooth,
theta functions are constructed via counts of rigid analytic discs
on mirror/Langlands dual cluster varieties using Berkovich geometry.
In~\cite{M19}, it is shown that theta functions for general cluster
varieties are determined by certain descendant log Gromov--Witten invariants
of their mirror/Langlands dual.

Theta functions produced by combinatorial methods in \cite{FZ02,FZ07},
called \emph{cluster variables}, are only parametrized by integer
tropical points in the cluster complex in the scattering diagram.
Cluster complex is a subfan structure in scattering diagram. In some
cases, we expect distinct subfan structures in the scattering diagrams
to yield non-equivalent cluster structures for cluster varieties.
This paper produces a first example along this line, as mentioned
in Theorem~\ref{thm:1.2}. The insight of~\cite{GHKK} also gives
correction for the original Fock--Goncharov duality conjecture. The
mirror dual variety to a cluster variety should be the space built
by the scattering diagram which includes torus charts coming from
all chambers in the scattering diagram.

\section{A quotient construction of scattering diagrams} \label{sec:A-quotient-construction}

\subsection{Scattering diagrams}

In this subsection, we will briefly recall basics of scattering diagrams
and prove lemmas we need for the rest of the paper. See Section~1 of the first arXiv version of \cite{GHKK} for a more detailed exposition on scattering diagrams. We refer to the first arXiv version of~\cite{GHKK}
here since the first version contains a more general treatment of
scattering diagrams that fits the framework of this paper better.

Let $N$ be a lattice with dual lattice $M$ and a skew-symmetric bilinear form
\[
\{\,,\,\}\colon \ N\times N\rightarrow\mathbb{Q}.
\]
Let $I$ be the index set $\{1,2,\dots,\operatorname{rank}N\}$. Fix a basis
$(e_{i})_{i\in I}$ for~$N$. Define $N^{+}\subset N$ as follows:
\[
N^{+}=\bigg\{\sum_{i\in I}a_{i}e_{i}\,|\, a_{i}\in\mathbb{N},\,\sum_{i\in I}a_{i}>0\bigg\}.
\]
Let $\mathfrak{g}=\bigoplus_{n\in N^{+}}\mathfrak{g}_{n}$ be a Lie
algebra over a ground field $\mathbb{K}$ of characteristic $0$.
Assume that $\mathfrak{g}$ is graded by $N^{+}$, that is, $[\mathfrak{g}_{n_{1}},\mathfrak{g}_{n_{2}}]\subset\mathfrak{g}_{n_{1}+n_{2}}$,
and that $\mathfrak{g}$ is skew-symmetric for $\{\,,\,\}$, that is,
$[\mathfrak{g}_{n_{1}},\mathfrak{g}_{n_{2}}]=0$ if $\{n_{1},n_{2}\}=0$.
Define the following linear functional on $N$:
\begin{align}
d\colon \ N & \rightarrow\mathbb{Z},\nonumber \\
\sum_{i\in I}a_{i}e_{i} & \mapsto\sum_{i\in I}a_{i}.\label{eq:-11}
\end{align}
For each $k\in\mathbb{Z}_{\geq0}$, there is an Lie subalgebra $\mathfrak{g}^{>k}$
of $\mathfrak{g}$:
\[
\mathfrak{g}^{>k}=\bigoplus_{n\in N^{+}\,|\, d(n)>k}\mathfrak{g}_{n}.
\]
And a $N^{+}$-graded nilpotent Lie algebra $\mathfrak{g}^{\leq k}$:
\[
\mathfrak{g}^{\leq k}=\mathfrak{g}/\mathfrak{g}^{>k}.
\]
We denote the unipotent Lie group corresponding to $\mathfrak{g}^{\leq k}$
as $G^{\leq k}$. As a set, it is in bijection with~$\mathfrak{g}^{\leq k}$,
but the group multiplication is given by the Baker--Campbell--Hausdorff
formula. Given $i<j$, we have canonical homomorphisms
\[
\mathfrak{g}^{\leq j}\rightarrow\mathfrak{g}^{\leq i},\qquad G^{\leq j}\rightarrow G^{\leq i}.
\]
Define the pro-nilpotent Lie algebra and its corresponding pro-unipotent
Lie group by taking the projective limits:
\[
\hat{\mathfrak{g}}=\lim_{\longleftarrow}\mathfrak{g}^{\leq k},\qquad G=\lim_{\longleftarrow}G^{\leq k}.
\]
We have canonical bijections:
\[
\exp\colon \ \mathfrak{g}^{\leq k}\rightarrow G^{\leq k},\qquad\exp\colon \ \hat{\mathfrak{g}}\rightarrow G.
\]
Moreover, given any Lie subalgebra $\mathfrak{h}\subset\mathfrak{g}$,
we can define a Lie subalgebra $\hat{\mathfrak{h}}$ of the completion~$\hat{\mathfrak{g}}$ and its corresponding Lie subgroup $\exp(\hat{\mathfrak{h}})\subset G$:
\[
\hat{\mathfrak{h}}=\lim_{\longleftarrow}\mathfrak{h}^{\leq k},\qquad\mathfrak{h}^{\leq k}=\mathfrak{h}/\big(\mathfrak{h}\cap\mathfrak{g}^{>k}\big).
\]
Given an element $n_{0}$ in $N^{+}$, there is an Lie subalgebra
$\mathfrak{g}_{n_{0}}^{\parallel}$ of $\mathfrak{g}$:
\[
\mathfrak{g}_{n_{0}}^{\parallel}=\bigoplus_{k\in\mathbb{Z}_{>0}}\mathfrak{g}_{k\cdot n_{0}},
\]
and its corresponding Lie subgroup of $G$
\[
G_{n_{0}}^{\parallel}=\exp\big(\widehat{\mathfrak{g}_{n_{0}}^{\parallel}}\big).
\]
By our assumption that $\mathfrak{g}$ is skew-symmetric for $\{\,,\,\}$,
$\mathfrak{g}_{n_{0}}^{\parallel}$ and therefore $G_{n_{0}}^{\parallel}$
are abelian.
\begin{Definition}
A \emph{wall} in $M_{\mathbb{R}}$ for $N^{+}$ and $\mathfrak{g}$
is a~pair $(\mathfrak{d},g_{\mathfrak{d}})$ such that
\begin{enumerate}\itemsep=0pt
\item[1)] $g_{\mathfrak{d}}$ is contained in $G_{n_{0}}^{\parallel}$ for some primitive element $n_{0}$ in $N^{+}$,
\item[2)] $\mathfrak{d}$ is contained in $n_{0}^{\perp}\subset M_{\mathbb{R}}$
and is a convex rational polyhedral cone of dimension \linebreak \mbox{$(\operatorname{rank}N-1)$}.
\end{enumerate}

For a wall $(\mathfrak{d},g_{\mathfrak{d}})$, we say $\mathfrak{d}$ is its \emph{support}.
\end{Definition}

\begin{Definition}
A \emph{scattering diagram} $\mathfrak{D}$ for $N^{+}$ and $\mathfrak{g}$
is a collection of walls such that for every order~$k$ in $\mathbb{Z}_{\geq0}$,
there are only finitely many walls $(\mathfrak{d},g_{\mathfrak{d}})$
in $\mathfrak{D}$ with non-trivial image of~$g_{\mathfrak{d}}$ under
the projection map $G\rightarrow G^{\leq k}$. Given a scattering
diagram $\mathfrak{D}$ and an order $k>0$, by passing the attached
group elements $g_{\mathfrak{d}}$ through the projection $G\rightarrow G^{\leq k}$,
we get a scattering diagram~$\mathfrak{D}^{\leq k}$ for $N^{+}$
and $\mathfrak{g}^{\leq k}$.
\end{Definition}

\begin{Remark}The finite condition in the definition of the scattering
diagram enables us, up to order $k$, reduce to the case where the
scattering diagram has only finitely many walls, a trick repeatedly
used in this paper.
\end{Remark}
\begin{Definition}
We define the\emph{ essential support, }or simply \emph{support},
of a scattering digram $\mathfrak{D}$ to be following set:
\[
{\rm supp}_{{\rm ess}}(\mathfrak{D})=\bigcup_{(\mathfrak{d},g_{\mathfrak{d}})\in\mathfrak{D}\,|\, g_{\mathfrak{d}}\neq0}\mathfrak{d}.
\]
We define the \emph{singular locus} of a scattering diagram $\mathfrak{D}$
to be the following set
\[
{\rm sing}(\mathfrak{D})=\bigcup_{(\mathfrak{d},g_{\mathfrak{d}})\in\mathfrak{D}}\partial\mathfrak{d}\cup\bigcup_{\substack{(\mathfrak{d}_{1},g_{\mathfrak{d}_{1}}),(\mathfrak{d}_{2},g_{\mathfrak{d}})\in\mathfrak{D}\\
\dim(\mathfrak{d}_{1}\cap\mathfrak{d}_{2})=n-2
}
}\mathfrak{d}_{1}\cap\mathfrak{d}_{2}.
\]
If $\mathfrak{D}$ has only finitely many walls, then its support
is a finite polyhedral cone complex. A \emph{joint} is an $(n-2)$-dimensional
cell of this complex.
\end{Definition}

\begin{Definition}
A point $x$ in $M_{\mathbb{R}}$ is \emph{general} if there exists
at most one rational hyperplane $n_{0}^{\perp}$ such that $x$ is
contained in $n_{0}^{\perp}$. Given a general point $x$ in $M_{\mathbb{R}}$,
define
\[
g_{x}(\mathfrak{D})=\prod_{(\mathfrak{d},g_{\mathfrak{d}})\in\mathfrak{D}\,|\,\mathfrak{d}\ni x}g_{\mathfrak{d}}.
\]
\end{Definition}

\begin{Definition}\label{def:Two-scattering-diagrams}Two scattering diagrams $\mathfrak{D}$, $\mathfrak{D}'$ are \emph{equivalent} if and only if $g_{x}(\mathfrak{D})=g_{x}(\mathfrak{D}')$
for all general points~$x$ in~$M_{\mathbb{R}}$.
\end{Definition}

Let $\gamma\colon [0,1]\rightarrow\mathfrak{D}$ be a smooth immersion. We say that $\gamma$ is \emph{generic} for $\mathfrak{D}$ if it satisfies the following conditions:
\begin{enumerate}\itemsep=0pt
\item[i)] The end points $\gamma(0)$ and $\gamma(1)$ does not lie in the
essential support of $\mathfrak{D}$.
\item[ii)] $\gamma$ does not pass through the singular locus of $\mathfrak{D}$.
\item[iii)] Whenever $\gamma$ crosses a wall, it crosses it transversally.
\end{enumerate}
Then given a path $\gamma$ generic for $\mathfrak{D}$, let $\mathfrak{p}_{\gamma,\mathfrak{D}}$
be the \emph{path-ordered product} as defined in~\cite{GHKK}.
\begin{Lemma}
Two scattering diagram $\mathfrak{D}$ and $\mathfrak{D}'$ are
equivalent if and only if $\mathfrak{p}_{\gamma,\mathfrak{D}}=\mathfrak{p}_{\gamma,\mathfrak{D}'}$
for any path $\gamma$ that is generic for both $\mathfrak{D}$ and~$\mathfrak{D}'$.
\end{Lemma}

\begin{Definition}A scattering diagram $\mathfrak{D}$ is \emph{consistent }if for any
path $\gamma$ that is generic for it, the path-ordered product $\mathfrak{p}_{\gamma,\mathfrak{D}}$
depends only on the end points of~$\gamma$.
\end{Definition}
\begin{Definition}Define the following chambers in $M_{\mathbb{R}}$:
\begin{gather*}
\mathcal{C}^{+} =\{m\in M_{\mathbb{R}}\,|\, m|_{N^{+}}\geq0\},\qquad
\mathcal{C}^{-} =\{m\in M_{\mathbb{R}}\,|\, m|_{N^{+}}\leq0\}.
\end{gather*}
We call $\mathcal{C}^{+}$ and $\mathcal{C}^{-}$ the \emph{positive
chamber }and the \emph{negative chamber} respectively.
\end{Definition}

\begin{Theorem}[Kontsevich--Soibelman \cite{KS13}] \label{thm:2.12}Given an element
$\mathfrak{p}_{+,-}$ in $G$, there is an unique equivalence class
of consistent scattering diagrams corresponding to $\mathfrak{p}_{+,-}$.
Choose a representative $\mathfrak{D}_{\mathfrak{p_{+,-}}}$ in this
equivalence class. The element $\mathfrak{p}_{+,-}$ is just the path-ordered
product along any path with the initial point in the interior of $\mathcal{C}^{+}$
and the end point in the interior of $\mathcal{C}^{-}$.
\end{Theorem}

\begin{Definition}
Given a wall $(\mathfrak{d},g_{\mathfrak{d}})$ such that $\mathfrak{d}$
is contained in $n_{0}^{\perp}$ for a primitive element $n_{0}\in N^{+}$,
we say that $(\mathfrak{d},g_{\mathfrak{d}})$ is \emph{incoming}
if $\{n_{0},\cdot\}$ is contained in $\mathfrak{d}$. Otherwise,
we say that $(\mathfrak{d},g_{\mathfrak{d}})$ is \emph{outgoing}.
\end{Definition}

\begin{Theorem}[{\cite[Theorems 1.12 and~1.21]{GHKK}}] \label{thm: uniqueness-incoming walls} The equivalence class of a consistent scattering diagram is determined by its set of incoming walls. Vice versa, let $\mathfrak{D}_{{\rm in}}$ be a scattering diagram whose only walls
are full hyperplanes, i.e., are of the form $\big(n_{0}^{\perp},g_{n_{0}}\big)$
where $n_{0}$ is a primitive element in $N^{+}$. Then there is a
scattering diagram $\mathfrak{D}$ satisfying:
\begin{enumerate}\itemsep=0pt
\item[$(1)$] $\mathfrak{D}$ is consistent,
\item[$(2)$] $\mathfrak{D}\supset\mathfrak{D}_{{\rm in}}$,
\item[$(3)$] $\mathfrak{D}\setminus\mathfrak{D}_{{\rm in}}$ consists only of outgoing walls.
\end{enumerate}

Moreover, $\mathfrak{D}$ satisfying these three properties is unique up to equivalence.
\end{Theorem}

Given a primitive element $n_{0}$ in $N^{+}$, consider the following splitting of $\mathfrak{g}$ with respect to $n_{0}$:
\begin{gather}
\mathfrak{g}=\mathfrak{g}_{+}^{n_{0}}\bigoplus \left(\mathfrak{g}_{n_{0}}^{\parallel}\bigoplus\mathfrak{g}_{n_{0}}^{\perp}\right) \bigoplus\mathfrak{g}_{-}^{n_{0}},\label{eq:-1}\\
\mathfrak{g}_{+}^{n_{0}}:=\bigoplus_{\{n_{0},n\}>0}\mathfrak{g}_{n},\qquad \mathfrak{g}_{-}^{n_{0}}:=\bigoplus_{\{n_{0},n\}<0}\mathfrak{g}_{n},\qquad \mathfrak{g}_{n_{0}}^{\parallel}:=\bigoplus_{k\in\mathbb{N}^{+}}\mathfrak{g}_{kn_{0}},\qquad \mathfrak{g}_{n_{0}}^{\perp}:=\bigoplus_{\substack{\{n_{0},n\}=0\\
n\notin\mathbb{N}^{+}\cdot n_{0}}}\mathfrak{g}_{n}.\nonumber
\end{gather}
The above splitting of $\mathfrak{g}$ induces unique factorization
$g=g_{+}^{n_{0}}\circ \big(g_{n_{0}}^{\parallel}\circ g_{n_{0}}^{\perp}\big)\circ g_{-}^{n_{0}}$.
Thus we can define the following map
\begin{align*}
\Psi\colon \ G & \rightarrow\prod_{n_{0}\in N^{+}\text{primitive}}G_{n_{0}}^{\parallel},\\
g & \mapsto\prod_{n_{0}\in N^{+}\text{primitive}}g_{n_{0}}^{\parallel}.
\end{align*}

\begin{Proposition}[{\cite[Proposition 1.20]{GHKK}}]\label{prop:()-The-map} The map~$\Psi$ is a bijection.
\end{Proposition}

\subsection{A quotient construction of scattering diagrams}

Let $\Pi$ be a finite group that acts on $I$ and satisfies the following
condition: For any indices $i,j\in I$, for any $\pi$, $\pi'$ in~$\Pi$, we have
\begin{gather}
\{e_{\pi\cdot i},e_{\pi'\cdot j}\}=\{e_{i},e_{j}\}.\label{eq:-2}
\end{gather}
The group $\Pi$ acts on the basis $(e_{i})_{i\in I}$ via permutation
of indices in $I$. Moreover, its action on $(e_{i})_{i\in I}$ can
be extended $\mathbb{Z}$-linearly to an action on the lattice $N$
and $\mathbb{R}$-linearly to an action on~$N_{\mathbb{R}}$. Given~$\pi$ in~$\Pi$, let $\pi^{*}\colon M_{\mathbb{R}}\rightarrow M_{\mathbb{R}}$ be the automorphism dual to $\pi\colon N_{\mathbb{R}}\rightarrow N_{\mathbb{R}}$. Assume that $\Pi$ also acts as Lie algebra automorphisms on~$\mathfrak{g}$
compatibly with its action on~$N^{+}$, that is, for any~$\pi$ in~$\Pi$, for any~$n$ in $N^{+}$,
\begin{gather}
\pi\cdot\mathfrak{g}_{n}=\mathfrak{g}_{\pi(n)}.\label{eq:-12}
\end{gather}

\begin{Proposition}\label{prop:2.14}Given $\pi$ in $\Pi$, for any general point~$x$ in $M_{\mathbb{R}}$, we have
\[
\pi\cdot g_{x}(\mathfrak{D}_{\mathfrak{p_{+,-}}})=g_{(\pi^{-1})^{*}\cdot x}(\mathfrak{D}_{\pi\cdot \mathfrak{p_{+,-}}}).
\]
\end{Proposition}

\begin{proof}It is clear that the action of $\Pi$ on $M_{\mathbb{R}}$ maps general
points to general points. Given a~general point $x$ in $M_{\mathbb{R}}$,
we have the splitting of Lie algebra $\mathfrak{g}$ with respect to~$x$: $\mathfrak{g}=\mathfrak{g}_{+}^{x}\oplus\mathfrak{g}_{0}^{x}\oplus\mathfrak{g}_{-}^{x}$. This splitting induces an unique factorization of $\mathfrak{p}_{+,-}$:
\[
\mathfrak{p}_{+,-}=\left(\mathfrak{p}_{+,-}\right)_{+}^{x}\circ \left(\mathfrak{p}_{+,-}\right)_{0}^{x}\circ\left(\mathfrak{p}_{+,-}\right)_{-}^{x},
\]
where $\left(\mathfrak{p}_{+,-}\right)_{\pm}^{x}$ is contained in
$\exp(\mathfrak{\hat{g}}_{\pm}^{x})$ and $\left(\mathfrak{p}_{+,-}\right)_{0}^{x}$
is contained in $\exp(\hat{\mathfrak{g}}_{0}^{x})$. Notice that $\pi\cdot(\mathfrak{g}_{+}^{x})=\mathfrak{g}_{+}^{(\pi^{-1})^{*}\cdot x}$.
Indeed, by condition~\eqref{eq:-12},
\begin{align*}
\mathfrak{g}_{n}\subseteq\mathfrak{g}_{+}^{(\pi^{-1})^{*}\cdot x} & \Longleftrightarrow\big\langle n,\big(\pi^{-1}\big)^{*}\cdot x\big\rangle >0
 \Longleftrightarrow\big\langle \pi^{-1}\cdot n,x\big\rangle >0\\
 & \Longleftrightarrow\mathfrak{g}_{\pi^{-1}\cdot n}\subseteq\mathfrak{g}_{+}^{x}
 \Longleftrightarrow\pi^{-1}\cdot\mathfrak{g}_{n}\subseteq\mathfrak{g}_{+}^{x}
 \Longleftrightarrow\mathfrak{g}_{n}\subseteq\pi\cdot\big(\mathfrak{g}_{+}^{x}\big).
\end{align*}
Similarly, we have $\pi\cdot\big(\mathfrak{g}_{0}^{x}\big)=\mathfrak{g}_{0}^{(\pi^{-1})^{*}\cdot x}$
and $\pi\cdot(\mathfrak{g}_{-}^{x})=\mathfrak{g}_{-}^{(\pi^{-1})^{*}\cdot x}$.
Hence
\[
\pi\cdot\mathfrak{p}_{+,-}=\pi\cdot (\mathfrak{p}_{+,-} )_{+}^{x}\circ\pi\cdot (\mathfrak{p}_{+,-} )_{0}^{x}\circ\pi\cdot (\mathfrak{p}_{+,-} )_{-}^{x}
\]
 is the unique factorization of $\pi\cdot\mathfrak{p}_{+,-}$ with respect to $\big(\pi^{-1}\big)^{*}\cdot x$. Therefore,
\[
 (\pi\cdot\mathfrak{p_{+,-}} )_{0}^{(\pi^{-1})^{*}(x)}=\pi\cdot (\mathfrak{p}_{+,-} )_{0}^{x}.
\]
By Theorem \ref{thm:2.12},
\begin{gather*}
g_{x}(\mathfrak{D}_{\mathfrak{p_{+,-}}}) = (\mathfrak{p}_{+,-} )_{0}^{x},\qquad
g_{(\pi^{-1})^{*}(x)}(\mathfrak{D}_{\pi\cdot\mathfrak{p_{+,-}}}) = (\pi\cdot\mathfrak{p}_{+,-} )_{0}^{(\pi^{-1})^{*}(x)}.
\end{gather*}
Hence, we get the equality:
\begin{gather*}
g_{(\pi^{-1})^{*}(x)}(\mathfrak{D}_{\pi\cdot\mathfrak{p_{+,-}}}) =\pi\cdot (\mathfrak{p}_{+,-} )_{0}^{x}=\pi\cdot g_{x}(\mathfrak{D}_{\mathfrak{p}_{+,-}}). \tag*{\qed}
\end{gather*}\renewcommand{\qed}{}
\end{proof}

\begin{Corollary}
Let $G^{\Pi}$ be the subgroup of $G$ invariant under the action
of~$\Pi$. Then given an element $\mathfrak{p}_{+,-}$ in~$G^{\Pi}$,
the set of walls in $\mathfrak{D}_{\mathfrak{p_{+,-}}}$ is invariant
under the action of $\Pi$ in the following sense:
\[
\pi\cdot g_{x}(\mathfrak{D}_{\mathfrak{p}_{+,-}})=g_{(\pi^{-1})^{*}\cdot x}(\mathfrak{D}_{\mathfrak{p_{+,-}}}).
\]
\end{Corollary}

Inspired by the above corollary, we can define an action of $\Pi$
on the set of walls for~$N^{+}$ and~$\mathfrak{g}$ as follows:
\begin{gather}
\pi\cdot(\mathfrak{d},g_{\mathfrak{d}}):=\big(\big(\pi^{-1}\big)^{*}(\mathfrak{d}),\pi\cdot g_{\mathfrak{d}}\big).\label{eq:-15}
\end{gather}
Observe that $(\pi^{-1})^{*}(n_{0}^{\perp})= (\pi\cdot n_{0} )^{\perp}$ for any primitive element $n_{0}$ in $N^{+}$. Indeed,
\begin{gather*}
 \langle n_{0},x \rangle =0 \Longleftrightarrow\big\langle n_{0},\pi^{*}\big(\pi^{-1}\big)^{*}x\big\rangle =0 \Longleftrightarrow\big\langle \pi\cdot n_{0},\big(\pi^{-1}\big)^{*}x\big\rangle =0.
\end{gather*}
So if $\mathfrak{d}$ is contained in $n_{0}^{\perp}$ for a primitive
element $n_{0}$ in $N^{+}$, $\big(\pi^{-1}\big)^{*}(\mathfrak{d})$ will
be contained in $(\pi\cdot n_{0})^{\perp}$. Since $g_{\mathfrak{d}}$
is contained in $G_{n_{0}}^{\parallel}$ and $\pi\cdot G_{n_{0}}^{\parallel}=G_{\pi\cdot n_{0}}^{\parallel}$,
we conclude that $\pi\cdot g_{\mathfrak{d}}$ is contained in $G_{\pi\cdot n_{0}}^{\parallel}$.
So the action of $\Pi$ on the set of walls is well-defined.

A natural question arises: How to characterize the elements in $G^{\Pi}$?
The following lemma gives a characterization of elements in $G^{\Pi}$
in terms of incoming walls of the corresponding scattering diagrams.
\begin{Lemma}
\label{lem:2.18}Given $\mathfrak{p}_{+,-}$ in $G$, set
\[
\mathfrak{D}_{{\rm in}}:=\big\{ \big(n_{0}^{\perp},\Psi(\mathfrak{p}_{+,-})_{n_{0}}\big)\,|\, n_{0}\in N^{+}\text{ primitive}\big\},
\]
where
\[
\Psi\colon \ G\rightarrow\prod_{n_{0}\in N^{+}\text{ primitive}}G_{n_{0}}^{\parallel}
\]
is the canonical projection map in Proposition~{\rm \ref{prop:()-The-map}}.
Then $\mathfrak{p}_{+,-}$ is in $G^{\Pi}$ if and only if $\mathfrak{D}_{{\rm in}}$
is invariant subset under the action of $\Pi$ on the set of walls $($cf.~\eqref{eq:-15}$)$.
\end{Lemma}

\begin{proof}
By Proposition \ref{prop:()-The-map}, $\Psi$ is a bijection of sets.
We claim that the projection map $\Psi$ is compatible with the action
of $\Pi$: For any $\pi$ in $\Pi$, for any $\mathfrak{p}_{+,-}$
in $G$,
\begin{gather*}
\Psi(\pi\cdot\mathfrak{p}_{+,-})_{\pi\cdot n_{0}}=\pi\cdot\Psi(\mathfrak{p}_{+,-})_{n_{0}}.
\end{gather*}
Let
\[
\mathfrak{p}_{+,-}=\left(\mathfrak{p}_{+,-}\right)_{n_{0}}^{+}\circ\left(\mathfrak{p}_{+,-}\right)_{n_{0}}^{\parallel}\circ\left(\mathfrak{p}_{+,-}\right)_{n_{0}}^{\perp}\circ\left(\mathfrak{p}_{+,-}\right)_{n_{0}}^{-}
\]
be the unique factorization of $\mathfrak{p}_{+,-}$ induced by splitting
of $\mathfrak{g}$ in~\eqref{eq:-1}. Observe that
\begin{gather*}
\pi\cdot\big(\mathfrak{g}_{+}^{n_{0}}\big) =\bigoplus_{\substack{n\in N^{+}
\{n_{0},\pi^{-1}(n)\}>0 }}\mathfrak{g}_{n}
 =\bigoplus_{\substack{n\in N^{+}\\
\{\pi(n_{0}),\pi (\pi^{-1}(n) )\}>0
}
}\mathfrak{g}_{n}
 =\bigoplus_{\substack{n\in N^{+}\\
\{\pi(n_{0}),n\}>0
}
}\mathfrak{g}_{n} =\mathfrak{g}_{+}^{\pi(n_{0})}.
\end{gather*}
Similarly, we have $\pi\cdot\big(\mathfrak{g}_{-}^{n_{0}}\big)=\mathfrak{g}_{-}^{\pi(n_{0})}$,
$\pi\cdot\big(\mathfrak{g}_{n_{0}}^{\parallel}\big)=\mathfrak{g}_{\pi(n_{0})}^{\parallel}$
and $\pi\cdot\big(\mathfrak{g}_{n_{0}}^{\perp}\big)=\mathfrak{g}_{\pi(n_{0})}^{\perp}$.
Therefore,
\[
\pi\cdot\mathfrak{p}_{+,-}=\pi\cdot\left(\mathfrak{p}_{+,-}\right)_{+}^{n_{0}}\circ\pi\cdot\left(\mathfrak{p}_{+,-}\right)_{n_{0}}^{\parallel}\circ\pi\cdot\left(\mathfrak{p}_{+,-}\right)_{n_{0}}^{\perp}\circ\pi\cdot\left(\mathfrak{p}_{+,-}\right)_{-}^{n_{0}}
\]
is the unique factorization of $\pi\cdot\mathfrak{p}_{+,-}$ with
respect to $\pi\cdot n_{0}$. By definition of $\Psi$, we have
\[
\Psi(\pi\cdot\mathfrak{p}_{+,-})_{\pi\cdot n_{0}}=\pi\cdot\left(\mathfrak{p}_{+,-}\right)_{n_{0}}^{\parallel}=\pi\cdot\Psi(\mathfrak{p}_{+,-})_{n_{0}}.
\]
Therefore, $\mathfrak{p}_{+,-}$ is contained in $G^{\Pi}$ if and
only if for all $\pi$ in $\Pi,$ for all primitive $n_{0}$ in $N^{+}$,
the following is satisfied
\begin{gather}
\Psi(\pi\cdot\mathfrak{p}_{+,-})_{\pi\cdot n_{0}}=\pi\cdot\Psi(\mathfrak{p}_{+,-})_{n_{0}}=\Psi(\mathfrak{p}_{+,-})_{\pi\cdot n_{0}}.\label{eq:-14}
\end{gather}
By definition,
\begin{gather*}
\pi\cdot\big(n_{0}^{\perp},\Psi(\mathfrak{p}_{+,-})_{n_{0}}\big) =\big(\big(\pi^{-1}\big)^{*}\big(n_{0}^{\perp}\big),\pi\cdot (\Psi(\mathfrak{p}_{+,-})_{n_{0}})\big)
 =\big((\pi\cdot n_{0})^{\perp},\pi\cdot (\Psi(\mathfrak{p}_{+,-})_{n_{0}} )\big).
\end{gather*}
Hence the equality \eqref{eq:-14} is equivalent to that for any $\pi$
in $\Pi$ and any primitive $n_{0}$ in $N^{+}$,
\[
\pi\cdot\big(n_{0}^{\perp},\Psi(\mathfrak{p}_{+,-})_{n_{0}}\big)=\big((\pi\cdot n_{0})^{\perp},\Psi(\mathfrak{p}_{+,-})_{\pi\cdot n_{0}}\big).
\]
which is equivalent to that $(\mathfrak{D}_{\mathfrak{p_{+,-}}})_{{\rm in}}$
is invariant under the action of $\Pi$. By Theorem~\ref{thm:2.12},
$(\mathfrak{D}_{\mathfrak{p}_{+,-}})_{{\rm in}}$ determines
the equivalence class of the scattering diagrams corresponding to
$\mathfrak{p_{+,-}}$. Hence we have the bijection of sets as stated
in the lemma.
\end{proof}

The action of $\Pi$ on $N$ enables us to construct a quotient lattice
$\overline{N}$ with a basis $(e_{\Pi i})_{\Pi i\in\overline{I}}$
where~$\overline{I}$ is the index set of orbits of $I$ under the
action of~$I$. Let $q\colon N\twoheadrightarrow\overline{N}$ be the natural
quotient homomorphism of lattices that sends $e_{i}$ to $e_{\Pi i}$.
We define the following skew-symmetric bilinear form $\{\,,\,\}$ on $\overline{N}$:
\[
\{e_{\Pi i},e_{\Pi j}\}=\{e_{i},e_{j}\}.
\]
The condition~\eqref{eq:-2} implies that the skew symmetric form on
$\overline{N}$ is well-defined and the skew symmetric form on~$N$
is the pull-back of that on~$\overline{N}$. The surjective homomorphism
between lattices~$q$ induces an injective homomorphism between dual lattices
$q^{*}\colon \overline{M}\hookrightarrow M$, which extends to an injective
linear map $q^{*}\colon \overline{M}_{\mathbb{R}}\hookrightarrow M_{\mathbb{R}}$.
Define
\[
\overline{N}^{+}=\bigg\{\sum_{\Pi i\in\overline{I}}a_{\Pi i}e_{\Pi i}\,|\, a_{\Pi i}\in\mathbb{Z}_{\geq0},\sum_{\Pi i\in\overline{I}}a_{\Pi i}>0\bigg\}.
\]
Notice that $q(N^{+})=\overline{N}^{+}$. Define the following linear
function
\begin{align}
\overline{d}\colon \ \overline{N} & \rightarrow\mathbb{Z},\nonumber \\
\sum_{\Pi i\in\overline{I}}a_{\Pi i}e_{\Pi i} & \mapsto\sum_{\Pi i\in\overline{I}}a_{\Pi i}.\label{eq:-10}
\end{align}
Given an element $n_{0}$ in $N^{+}$, let $\overline{n_{0}}$ be
its image in $\overline{N}^{+}$ under~$q$. Let $\overline{\mathfrak{g}}=\bigoplus_{\overline{n}\in\overline{N}^{+}}\overline{\mathfrak{g}}_{\overline{n}}$
be a~$\overline{N}^{+}$-graded Lie algebra skew-symmetric for $\{\,,\,\}$
on $\overline{N}$. Similar to $\mathfrak{g}$, for every $k$ in
$\mathbb{Z}_{>0}$, we define
\begin{gather*}
\overline{\mathfrak{g}}^{>k}=\bigoplus_{\overline{n}\in\overline{N}^{+}\,|\,\bar{d} (\overline{n})>k}\overline{\mathfrak{g}}_{\overline{n}},\qquad
\overline{\mathfrak{g}}^{\leq k}=\overline{\mathfrak{g}}/\overline{\mathfrak{g}}^{>k},\qquad\overline{G}^{\leq k}=\exp\big(\overline{\mathfrak{g}}^{\leq k}\big),\\
\widehat{\overline{\mathfrak{g}}}=\lim_{\longleftarrow}\overline{\mathfrak{g}}^{\leq k},\qquad G=\exp\big(\widehat{\overline{\mathfrak{g}}}\big)=\lim_{\longleftarrow}\overline{G}^{\leq k}.
\end{gather*}
Suppose that we are given an order-preserving graded Lie algebra homomorphism
$\tilde{q}\colon \mathfrak{g}\rightarrow\overline{\mathfrak{g}}$, then
for each order $k\in\mathbb{Z}_{>0}$, we have an induced Lie algebra
homomorphism: $\tilde{q}\colon \mathfrak{g}^{\leq k}\rightarrow\overline{\mathfrak{g}}^{\leq k}$.
Let $\tilde{q}\colon G^{\leq k}\rightarrow\overline{G}^{\leq k}$ be the
corresponding homomorphism of Lie groups. Taking the limit of compositions
of Lie group homomorphisms $G\rightarrow G^{\leq k}\rightarrow\overline{G}^{\leq k}$
for all $k\in\mathbb{Z}_{>0}$, we get a Lie group homomorphism $\tilde{q}\colon G\rightarrow\overline{G}$.

Let $\bar{n}_{0}$ be a primitive element in $\overline{N}^{+}$ and
$x$ a general element in $\bar{n}_{0}^{\perp}$. Notice that $q^{*}(x)$
would no longer necessarily be general in $M_{\mathbb{R}}$. However,
we have the following lemma.
\begin{Lemma}
\label{lem:2.19}The direct sum
\[
\mathfrak{g}_{_{0}}^{q^{*}(x)}=\bigoplus_{\substack{n\in N_{{\bf }}^{+}\\
\langle n,q^{*}(x)\rangle=0}
}\mathfrak{g}_{n}
\]
is an abelian Lie subalgebra.
\end{Lemma}

\begin{proof}
Since $x\in\bar{n}_{0}^{\perp}$ is general, we have
\[
\mathfrak{g}_{0}^{q^{*}(x)}=\bigoplus_{\substack{q(n)=k\overline{n}_{0}\\
k\in\mathbb{N}\\
n\in N^{+}\text{ is primitive}
}
}\mathfrak{g}_{n}^{\parallel}.
\]
For any primitive elements $n_{1}$, $n_{2}$ in $N^{+}$ such that $q(n_{i})=k_{i}\overline{n}_{0}$,
$k_{i}\in\mathbb{N}$,
\begin{gather*}
\{n_{1},n_{2}\} =\{q(n_{1}),q(n_{2})\} =\{k_{1}\overline{n}_{0},k_{2}\overline{n}_{0}\} =0.
\end{gather*}
Since $\mathfrak{g}$ is skew symmetric for $\{\,,\,\}$ on $N$, $\{n_{1},n_{2}\}=0$
implies that $[\mathfrak{g}_{n_{1}},\mathfrak{g}_{n_{2}}]=0$. Therefore~$\mathfrak{g}_{0}^{q^{*}(x)}$ is an abelian Lie subalgebra.
\end{proof}
\begin{Lemma}\label{lem: general-interior}There exists a scattering diagram $\mathfrak{D}'$
in the equivalence class of $\mathfrak{D}_{\mathfrak{p}_{+,-}}$ such
that for each wall $(\mathfrak{d},g_{\mathfrak{d}})$ in $\mathfrak{D}'$
whose support contains $q^{*}(x)$ for a general point $x$ in $\overline{M}_{\mathbb{R}}$,
$q^{*}(x)$ is contained in the interior of $\mathfrak{d}$.
\end{Lemma}

\begin{proof}\looseness=-1 We work modulo $\mathfrak{g}^{>k}$ for any $k$, so that we could
assume that $\mathfrak{D}_{\mathfrak{p}_{+,-}}$ has finitely many
walls. By the inductive algorithm to construct a consistent scattering
diagram as given in~\cite{GS11}, $\mathfrak{D}_{\mathfrak{p}_{+,-}}$
is equivalent to a scattering diagram~$\mathfrak{D}'$ such that
if~$\partial\mathfrak{d}$ is a boundary of a wall in~$\mathfrak{D}'$,
then~$\partial\mathfrak{d}$ must be contained in a joint that is
the intersection of two walls whose attached Lie group elements do
not commute with each other and whose boundaries are not this joint
(cf.\ Definition-Lemma~C.2, Lemma~C.6 and Lemma~C.7 of~\cite{GHKK}).
However, given a general point $x$ in $\overline{M}_{\mathbb{R}}$,
by Lemma~\ref{lem:2.19}, any two walls that contain~$q^{*}(x)$ must
have attached Lie group elements commuting with each other. Therefore,
$q^{*}(x)$ cannot be contained in the boundary of a wall in~$\mathfrak{D}'$.
\end{proof}

\begin{Construction} \label{2.1} By Lemma \ref{lem: general-interior}, given any general point $x$ in $\overline{M}_{\mathbb{R}}$, if $q^{*}(x)$ is contained in the support of a wall $(\mathfrak{d},g_{\mathfrak{d}})$ in $\mathfrak{D}_{\mathfrak{p}_{+,-}}$, we could assume that $q^{*}(x)$ is contained in the interior of~$\mathfrak{d}$. Under this assumption, we build a scattering diagram $\overline{\mathfrak{D}_{\mathfrak{p_{+,-}}}}$ in $\overline{M}_{\mathbb{R}}$. To do so, it suffices to build $\overline{\mathfrak{D}_{\mathfrak{p_{+,-}}}}$ up to each order~$k$. Modulo $G^{>k}$, there are only finitely many nontrivial walls in~$\mathfrak{D}_{\mathfrak{p}_{+,-}}$. For each such nontrivial wall $(\mathfrak{d},g_{\mathfrak{d}})$, if $\mathfrak{d}\cap q^{*}\big(\overline{M}_{\mathbb{R}}\big)$ has codimension $1$ in $q^{*}\big(\overline{M}_{\mathbb{R}}\big)$, define $\overline{\mathfrak{d}}=q^{*-1}\big(\mathfrak{d}\cap q^{*}\big(\overline{M}_{\mathbb{R}}\big)\big)$ and $\overline{g}_{\mathfrak{d}}= \tilde{q}({g_\mathfrak{d}})$. Then $(\overline{\mathfrak{d}},\overline{g}_{\mathfrak{d}})$ will be a wall for $\overline{N}^{+}$ and $\overline{\mathfrak{g}}$. Indeed, if~$\mathfrak{d}$ is contained in $n^{\perp}$ for a primitive element~$n$ in~$N^{+}$, then $\overline{n}$ is an integer multiple of $\overline{n}_{0}$ for a primitive element $\overline{n}_{0}$ in $\overline{N}^{+}$. Then $\overline{\mathfrak{d}}$ is contained in $\overline{n}_{0}^{\perp}$ and $\overline{g}_{\mathfrak{d}}$ is contained in $\overline{G}_{\overline{n}_{0}}$. \end{Construction}

Let $\overline{\mathfrak{p}_{+,-}}$ be the image of $\mathfrak{p}_{+,-}$
under $\tilde{q}\colon G\rightarrow\overline{G}$. Let $\mathfrak{D}_{\overline{\mathfrak{p_{+,-}}}}$
be a consistent scattering diagram for~$\overline{N}^{+}$ and $\overline{\mathfrak{g}}$
corresponding to $\overline{\mathfrak{p}_{+,-}}$. In the next proposition,
we show that the quotient construction is compatible with the Lie
group homomorphism $\tilde{q}\colon G\rightarrow\overline{G}$, that is,
$\overline{\mathfrak{D}_{\mathfrak{p_{+,-}}}}$ is equivalent to $\mathfrak{D}_{\overline{\mathfrak{p_{+,-}}}}$:
\begin{Theorem}\label{thm:2.21}The scattering diagram $\overline{\mathfrak{D_{\mathfrak{p_{+,-}}}}}$
is consistent and it belongs to the equivalence class of consistent
scattering diagrams uniquely determined by $\overline{\mathfrak{p}_{+,-}}$.
\end{Theorem}

\begin{proof}\looseness=-1 It suffices to show that for each order $k$, $\overline{\mathfrak{D_{\mathfrak{p_{+,-}}}}}$
is consistent and is equivalent to $\mathfrak{D}_{\overline{\mathfrak{p_{+,-}}}}$
modulo~$\overline{G}^{>k}$. Modulo~$G^{>k}$, by the proof of Lemma~\ref{lem: general-interior}, we could assume that if $\partial\mathfrak{d}$ is a boundary of a wall in $\mathfrak{D}_{\mathfrak{p}_{+,-}}$, then
$\partial\mathfrak{d}$ must be contained in a joint that is the intersection
of two walls whose attached Lie group elements do not commute with
each other and whose boundaries are not this joint. To prove the consistency
of $\overline{\mathfrak{D_{\mathfrak{p_{+,-}}}}}$, given a generic
path $\overline{\gamma}$ in $\overline{M}_{\mathbb{R}}$ (cf.\ Definition~\ref{def:Two-scattering-diagrams}), we want to show that the path-ordered
product $\mathfrak{p}_{\overline{\gamma},\overline{\mathfrak{D}_{\mathfrak{p}_{+,-}}}}$
depends only on end points of $\overline{\gamma}$. We show that we
could always perturb $q^{*}(\overline{\gamma})$ to a generic path
$\gamma$ in $M_{\mathbb{R}}$ with the same end points such that
\[
\mathfrak{p}_{\overline{\gamma},\overline{\mathfrak{D}_{\mathfrak{p}_{+,-}}}} =\overline{\mathfrak{p}_{\gamma,\mathfrak{D}_{\mathfrak{p}_{+,-}}}}.
\]
Let $0<t_{1}<\cdots<t_{l}<1$ be time numbers when $\overline{\gamma}$
crosses a wall in $\overline{\mathfrak{D}_{\mathfrak{p_{+,-}}}}$.
For any $t\notin\{t_{1},\dots,t_{l}\}$, we claim that $q^{*}\left(\overline{\gamma}(t)\right)$
is not contained in any nontrivial wall in $\mathfrak{D}_{\mathfrak{p}_{+,-}}$.
Indeed, assume that $q^{*}\left(\overline{\gamma}(t)\right)$ is contained
in a non-trivial wall $(\mathfrak{d},g_{\mathfrak{d}})$ in $\mathfrak{D}_{\mathfrak{p}_{+,-}}$.
Since $\overline{\gamma}(t)$ is not contained in any wall in $\overline{\mathfrak{D}_{\mathfrak{p}_{+,-}}}$,
$\mathfrak{d}\cap q^{*}\left(\overline{M}_{\mathbb{R}}\right)$ has
codimension $2$ in $q^{*}(\overline{M}_{\mathbb{R}})$, which implies
that $\mathfrak{d}\cap q^{*}\left(\overline{M}_{\mathbb{R}}\right)$
is contained in $\partial\mathfrak{d}$. By our assumption about $\mathfrak{D}_{\mathfrak{p}_{+,-}}$,
$\partial\mathfrak{d}$ is contained in a joint that is the intersection
of two non-trivial walls in $\mathfrak{D}_{\mathfrak{p}_{+,-}}$ whose
boundaries are not this joint. These two walls will contribute to
two non-trivial walls in $\overline{\mathfrak{D_{\mathfrak{p_{+,-}}}}}$
whose intersection contains $\overline{\gamma}(t)$, which is in contradiction
with that $\overline{\gamma}$ is a generic path and does not pass
through a joint in $\overline{\mathfrak{D_{\mathfrak{p_{+,-}}}}}$.

Given $i\in\{1,\dots, l\}$, let $D=\{(\mathfrak{d}_{1},g_{\mathfrak{d}_{1}}),\dots,(\mathfrak{d}_{j},g_{\mathfrak{d}_{j}})\}$
be the collection of all nontrivial walls in~$\mathfrak{D}_{\mathfrak{p_{+,-}}}$
whose support contains $q^{*}(\overline{\gamma}(t_{i}))$ and whose
intersection with $q^{*}(\overline{M}_{\mathbb{R}})$ has codimension~$1$ in $q^{*}(\overline{M}_{\mathbb{R}})$. Let $\delta=\min\{t_{i}-t_{i-1},t_{i+1}-t_{i}\}$.
There exists $\epsilon\in(0,\delta)$ such that we can find a~semi-circle
$C_{i}$ that has end points $q^{*}(\overline{\gamma}(t_{i}-\epsilon))$
and $q^{*}(\overline{\gamma}(t_{i}+\epsilon))$, crosses walls in
$D$ transversally and does not cross any non-trivial wall in $\mathfrak{D}_{\mathfrak{p}_{+,-}}\setminus D$.
It is possible to find such a semi-circle since $q^{*}(\bar{\gamma}(t_{i}))$
is contained in the interior of each wall in~$D$ and $\overline{\gamma}$
crosses no wall on $[t_{i}-\epsilon,t_{i})\cup(t_{i},t_{i}+\epsilon]$.
Replace $q^{*}(\overline{\gamma})|_{[t_{i}-\epsilon,t_{i}+\epsilon]}$
with~$C$. Then
\[
\overline{g}_{\gamma(t_{i})}(\overline{\mathfrak{D}_{\mathfrak{p}_{+,-}}}) =\overline{\mathfrak{p}_{C_{i},\mathfrak{D}_{\mathfrak{p}_{+,-}}}}.
\]
Repeating the perturbation for each time~$t_{i}$ and smoothing the
singularities at the end points of~$C_{i}$, we get a generic path~$\gamma$ with the same end points as~$\overline{\gamma}$ for $\mathfrak{D}_{\mathfrak{p}_{+,-}}$
such that $\mathfrak{p}_{\overline{\gamma}, \overline{\mathfrak{D}_{\mathfrak{p}_{+,-}}}}= \overline{\mathfrak{p}_{\gamma,\mathfrak{D}_{\mathfrak{p}_{+,-}}}}$.
The consistency of $\overline{\mathfrak{D}_{\mathfrak{p_{+,-}}}}$
thus follows from that of~$\mathfrak{D}_{\mathfrak{p}_{+,-}}$. If
we let $\overline{\gamma}$ be the path from~$\overline{\mathcal{C}}^{+}$
to~$\overline{\mathcal{C}}^{-}$ where~$\overline{\mathcal{C}}^{\pm}$
are the positive chamber and the negative chamber in $\overline{M}_{\mathbb{R}}$,
we see that~$\overline{\mathfrak{D_{\mathfrak{p_{+,-}}}}}$ is equivalent
to~$\mathfrak{D}_{\overline{\mathfrak{p}_{+,-}}}$.
\end{proof}
\begin{Corollary}\label{cor:2.22}Let $x$ be a general point in $\overline{M}_{\mathbb{R}}$, then
\[
\overline{g}_{x}(\mathfrak{D}_{\overline{\mathfrak{p_{+,-}}}}) =\prod_{\substack{(\mathfrak{d},g_{\mathfrak{d}})\in\mathfrak{D}_{\mathfrak{p}_{+,-}}\\
\mathfrak{d}\ni q^{*}(x)
}
}\overline{g_{\mathfrak{d}}},
\]
where $\overline{g_{\mathfrak{d}}}$ is the image of $g_{\mathfrak{d}}$ under the composition $\tilde{q}\colon G\rightarrow\overline{G}$.
\end{Corollary}

\section[Applying the quotient construction to cluster scattering diagrams]{Applying the quotient construction\\ to cluster scattering diagrams} \label{sec:Apply-quotient-constructions}

In this section, we will apply techniques developed in the previous
section to scattering diagrams arising from cluster theory. Given
a skew-symmetric seed~${\bf s}$, via exploiting the symmetry of~${\bf s}$,
we will construct a new seed~$\overline{{\bf s}}$. Readers familiar
with folding in cluster theory will recognize our construction as
a special class of the folding procedure. The new contribution of
this paper to the folding technique is the application to scattering
diagrams. For previous work on folding, see~\cite{Du08,FWZ}.

\subsection{Application of folding to cluster scattering diagrams}\label{subsec:Application-of-folding}

Fix a lattice $N$ of rank $r$ with a skew-symmetric $\mathbb{Z}$-bilinear
form $\{\,,\,\}\colon N\times N\rightarrow\mathbb{Z}$ and an index set $I=\{1,2,\dots,r\}$.
A \emph{seed} for $(N,\{\cdot,\cdot\})$ is a labelled
collection of elements of~$N$
\[
{\bf s}:=(e_{i}\,|\, i\in I)
\]
such that $\{e_{i}\,|\, i\in I\}$ is a basis of $N$. A choice of
seed~${\bf s}$ defines an exchange matrix~$\epsilon$ such that $\epsilon_{ij}=\{e_{i},e_{j}\}$.
Given a seed ${\bf s}$, define
\[
N^{+}=\bigg\{\sum_{i\in I}a_{i}e_{i}\,|\, a_{i}\in\mathbb{N},\sum_{i\in I}a_{i}>0\bigg\}.
\]

\begin{Remark}
For simplicity, we assume that we do not have frozen variables for
the entire paper. In general when we fix the lattice $N$, we also
fix a sublattice $N^{\circ}\subset N$ and positive integers $d_{i}$
for $i\in I$ such that $\{d_{i}e_{i}\,|\, i\in I\}$ is a basis for
$N^{\circ}$. In this paper, for simplicity, we assume that the original
seed ${\bf s}$ is \emph{simply laced}, that is, $d_{i}=1$ for all
$i\in I$. Later in this section, given a group $\Pi$ acting on $I$,
we will apply the folding technique to construct a new seed $\overline{{\bf s}}$
that will no longer be simply laced.
\end{Remark}

From now on, we always assume that we are working over the ground
field $\mathbb{C}$. Let us recall how cluster theory fits into the
framework in the previous section. Let $M={\rm Hom}(N,\mathbb{Z})$
be the dual lattice of $N$. Define the following algebraic torus
over $\mathbb{C}$:
\[
T_{M}=M\otimes_{\mathbb{Z}}\mathbb{G}_{m}={\rm Hom}(N,\mathbb{G}_{m})={\rm Spec}(\mathbb{C}[N]).
\]
The skew-symmetric form $\{\,,\,\}$ on $N$ gives the algebraic torus
$T_{M}$ a Poisson structure
\[
\big\{z^{n},z^{n'}\big\}=\{n,n'\}z^{n+n'}.
\]
Define $\mathfrak{g}=\bigoplus_{n\in N^{+}}\mathbb{C}\cdot z^{n}$.
The Poisson bracket $\{\,,\,\}$ on $\mathbb{C}[N]$ endows $\mathfrak{g}$
with a Lie bracket $[\,,\,]$: $[f,g]=-\{f,g\}$. Given $n$, $n'$ in
$N^{+}$, if $\{n,n'\}=0$, then $\{z^{n},z^{n'}\}=0$. Thus
$\mathfrak{g}$ is a $N^{+}$-graded Lie algebra skew-symmetric for
$\{\,,\,\}$ on~$N$. Let $\hat{\mathfrak{g}}$ be the pro-nilpotent Lie
algebra after completing~$\mathfrak{g}$. Let $G=\exp(\hat{\mathfrak{g}})$
be the pro-unipotent Lie group corresponding to $\hat{\mathfrak{g}}$.
The Lie group~$G$ acts on $\mathbb{C}[[z^{e_{1}},\dots,z^{e_{r}}]]$
via the Hamiltonian flow. Explicitly, an element $f$ in $\mathfrak{g}$
acts on $\mathbb{C}[[z^{e_{1}},\dots,z^{e_{r}}]]$ via the vector
field $\{\,,f\}$.

The seed ${\bf s}$ gives rises to a scattering diagram $\mathfrak{D}_{{\bf s}}$
whose set of initial walls are defined as follows:
\[
\mathfrak{D}_{{\bf s},{\rm in}}=\big\{ \big(e_{i}^{\perp},\exp\big({-}{\rm Li}_{2}\big({-}z^{e_{i}}\big)\big)\big)\,|\, i\in I\big\},
\]
where ${\rm Li}_{2}$ is the dilogarithm function ${\rm Li}_{2}(x)=\sum\limits_{k\geq1}\frac{x^{k}}{k^{2}}$.
Explicitly, $\exp\big({-}{\rm Li}_{2}\big({-}z^{e_{i}}\big)\big)$ acts as
the following automorphism:{\samepage
\[
z^{n}\mapsto z^{n}\big(1+z^{e_{i}}\big)^{\{n,e_{i}\}}.
\]
This is the coordinate free expression for the inverse of $\mathcal{X}$-cluster mutation.}

Let $\Pi$ be a group action on $I$ such that for any indices $i,j\in I$, for any $\pi_{1}$, $\pi_{2}$ in~$\Pi$,
\begin{gather}
\{e_{i},e_{j}\}=\{e_{\pi_{1}\cdot i},e_{\pi_{2}\cdot j}\}.\label{eq:-8}
\end{gather}
The group $\Pi$ acts naturally on $N$ as follows:
\[
\pi\cdot\sum_{i\in I}a_{i}e_{i}=\sum_{i\in I}a_{i}e_{\pi\cdot i}.
\]
 The condition~\eqref{eq:-8} is to guarantee that the skew-symmetric
pairing is independent of choice of representative in the orbit of~$\Pi$ acting on $N$. Let $\bar{I}$ be the set of orbits~$\Pi i$
of $I_{{\rm }}$ under the action of~$\Pi$. Consider the lattice
$\overline{N}$ with a basis $\{e_{\Pi i}\}_{\Pi i\in\overline{I}}$
indexed by $\bar{I}$. There is a natural quotient homomorphism between
lattices $q\colon N\rightarrow\overline{N}$ that sends $e_{i}$ to $e_{\Pi i}$.

Define a $\mathbb{Z}$-valued skew-symmetric bilinear form $\{\,,\,\}$
on $\overline{N}$ as follows: $\{e_{\Pi i},e_{\Pi j}\}=\{e_{i},e_{j}\}$.
By condition~\eqref{eq:-8}, the skew-symmetric form is well-defined.
Now we have the following fixed data:
\begin{itemize}\itemsep=0pt
\item The lattice $\overline{N}$ with the $\mathbb{Z}$-valued skew-symmetric bilinear form $\{\,,\,\}$.
\item The index set $\overline{I}$ parametrizing the set of orbits of~$I$ under the action of~$\Pi$.
\item For each $\Pi i\in\overline{I}$, a positive integer $d_{\Pi i}=\big|q^{-1}\{e_{\Pi i}\}\big|$.
\end{itemize}
Define the following seed $\bar{{\bf s}}=(e_{\Pi i}\,|\, \Pi i\in\bar{I})$.
Let $\overline{N}^{\circ}\subset N$ be the sublattice such that $\{d_{\Pi i}e_{\Pi i}\,|\, \allowbreak i\in\overline{I}\}$
is a basis. Unlike~\cite{GHKK}, we do not impose the condition that
the greatest common divisor of $\{d_{\Pi i}\}_{\Pi i\in\overline{I}}$
need to be~$1$. Notice that the new seed $\overline{{\bf s}}$ may
no longer be simply laced since $d_{\Pi i}$ may not be $1$. The
seed $\overline{{\bf s}}$ defines the exchange matrix $\overline{\epsilon}$
such that
\[
\overline{\epsilon}_{\Pi i\,\Pi j}=\{e_{\Pi i},e_{\Pi j}\}d_{\Pi j}.
\]
Take
\[
\overline{N}^{+}=\bigg\{\sum_{\Pi i\in\bar{I}}a_{\Pi i}e_{\Pi i}\,|\, a_{\Pi i}\geq0,\sum a_{\Pi i}>0\bigg\}.
\]
Let $\mathfrak{\bar{g}}\subset\mathbb{C}\big[\,\overline{N}\,\big]$ be the $\mathbb{C}$-vector
space with basis $z^{\bar{n}}$, $\bar{n}\in\overline{N}^{+}$. Define
a Lie bracket on $\bar{\mathfrak{g}}$ as follows:
\[
\big[z^{e_{\Pi i}},z^{e_{\Pi j}}\big]=-\{e_{\Pi i},e_{\Pi j}\}z^{e_{\Pi i}+e_{\Pi j}}.
\]
Define $\tilde{q}\colon \mathfrak{g}\rightarrow\overline{\mathfrak{g}}$
by $z^{n}\mapsto z^{\overline{n}}$ where $\bar{n}$ is the image
of $n$ under the quotient homomorphism $q\colon N\rightarrow\overline{N}$.
It is clear that $\tilde{q}$ is surjective. We check that $\tilde{q}$
is a~Lie-algebra homomorphism. Indeed,
\begin{align*}
\big[\tilde{q}\big(z^{n}\big),\tilde{q}\big(z^{n'}\big)\big] & =-\{\overline{n},\overline{n'}\}z^{\overline{n+n'}} =-\{n,n'\}z^{\overline{n+n'}}=\tilde{q}\big(\big[z^{n},z^{n'}\big]\big).
\end{align*}
For each order $k$, by our definition of $d\colon N^{+}\rightarrow\mathbb{Z}$
and $\overline{d}\colon \overline{N}^{+}\rightarrow\mathbb{Z}$ (cf.~\eqref{eq:-11},
\eqref{eq:-10}), $\tilde{q}$ restricts to a Lie algebra homomorphism
\[
\mathfrak{g}^{\leq k}=\mathfrak{g}/\mathfrak{g}^{>k}\rightarrow\overline{\mathfrak{g}}/\overline{\mathfrak{g}}^{>k}=\overline{\mathfrak{g}}^{\leq k},
\]
which induces a surjective Lie group homomorphism $G^{\leq k}\rightarrow\overline{G}^{\leq k}$.
Let $\tilde{q}^{k}\colon G\rightarrow\overline{G}^{\leq k}$ be the composition
of projection map $G\rightarrow G^{\leq k}$ with the above homomorphism.
Then we obtain the following Lie group homomorphism $\tilde{q}:=\underset{\leftarrow}{\lim}\,\tilde{q}^{k}\colon G\rightarrow\overline{G}$.
It is easy to check that both $\tilde{q}$ and the action of $\Pi$
are compatible with the quotient homomorphism $q\colon N\rightarrow\overline{N}$
(cf.~\eqref{eq:-12}).

By Theorem \ref{thm: uniqueness-incoming walls}, the seed $\overline{{\bf s}}$
gives rises to a consistent scattering diagram $\mathfrak{D}_{\overline{{\bf s}}}$
whose set of incoming walls are defined as follows:
\[
\mathfrak{D}_{\overline{{\bf s}},{\rm in}}=\big\{ \big(e_{\Pi i}^{\perp},\exp\big({-}d_{\Pi i}{\rm Li}_{2}\big({-}z^{e_{\Pi i}}\big)\big)\big)\,|\,\Pi i\in I\big\}.
\]
Let $\mathfrak{p}_{+,-}^{{\bf s}}$ be the unique group element corresponding
to $\mathfrak{D}_{{\bf s}}$ and $\mathfrak{p}_{+,-}^{{\bf \overline{s}}}$
that of $\mathfrak{D}_{\overline{{\bf s}}}$. We show that image of
$\mathfrak{p}_{+,-}^{{\bf s}}$ under the Lie group homomorphism $\tilde{q}\colon G\rightarrow\overline{G}$
is $\mathfrak{p}_{+.-}^{{\bf \overline{{\bf s}}}}$ so that results
proved in Section \ref{sec:A-quotient-construction} can be applied
to $\mathfrak{D}_{{\bf s}}$ and $\mathfrak{D}_{\overline{{\bf s}}}$.
\begin{Theorem}
\label{thm:3.2}The scattering diagram $\mathfrak{D}_{\overline{{\bf s}}}$
is equivalent to $\mathfrak{D}_{\overline{\mathfrak{p_{+,-}^{{\bf s}}}}}.$
\end{Theorem}

\begin{proof}
Let $\mathfrak{D}_{\overline{\mathfrak{p_{+,-}^{{\bf s}}}},{\rm in}}$
be the set of incoming walls of $\mathfrak{D}_{\overline{\mathfrak{p_{+,-}^{{\bf s}}}}}.$
Given $(\overline{\mathfrak{d}},g_{\overline{\mathfrak{d}}})$ in
$\mathfrak{D}_{\overline{\mathfrak{p}_{+,-}^{{\bf s}}}}$ with $\overline{\mathfrak{d}}$
contained in $\overline{n_{0}}^{\perp}$ for a primitive element in
$N^{+}$, $(\overline{\mathfrak{d}},g_{\overline{\mathfrak{d}}})$
is contained in $\mathfrak{D}_{\overline{\mathfrak{p_{+,-}^{{\bf s}}}},{\rm in}}$
if and only if $\{\overline{n_{0}},\cdot\}$ is contained in $\overline{\mathfrak{d}}$.
Let $(\mathfrak{d},g_{\mathfrak{d}})$ be the wall in $\mathfrak{D}_{{\bf s}}$
such that $q^{*}(\overline{\mathfrak{d}})=\mathfrak{d}\bigcap q^{*}\left(\overline{M}_{\mathbb{R}}\right)$
and $\overline{g_{\mathfrak{d}}}=g_{\overline{\mathfrak{d}}}$. Then
$\{\overline{n_{0}},\cdot\}$ is contained in $\overline{\mathfrak{d}}$
if and only if $q^{*}(\{\overline{n_{0}},\cdot\})$ is contained in
$q^{*}(\overline{\mathfrak{d}})$, which is equivalent to that $\{n_{0},\cdot\}$
is contained in $\mathfrak{d}\bigcap q^*\big(\overline{M}_{\mathbb{R}}\big)$. Hence
$(\mathfrak{d},g_{\mathfrak{d}})$ is contained in $\mathfrak{D}_{{\bf s},{\rm in}}$
if $(\overline{\mathfrak{d}},g_{\overline{\mathfrak{d}}})$ is contained
in $\mathfrak{D}_{\overline{\mathfrak{p_{+,-}^{{\bf s}}}},{\rm in}}$.
Therefore each wall in~$\mathfrak{D}_{\overline{\mathfrak{p}_{+,-}},{\rm in}}$
comes from a wall in~$\mathfrak{D}_{{\bf s},{\rm in}}$. Given $\big(e_{i}^{\perp},\exp\big({-}{\rm Li}_{2}\big({-}z^{e_{i}}\big)\big)\big)$ in~$\mathfrak{D}_{{\bf s},{\rm in}}$,
\[
q^{*-1}\big[e_{i}^{\perp}\bigcap q^{*}\big(\overline{M}_{\mathbb{R}}\big)\big]=e_{\Pi i}^{\perp}.
\]
So
\begin{gather*}
\mathfrak{D}_{\overline{\mathfrak{p}_{+,-}},{\rm in}} =\big\{ \big(e_{\Pi i}^{\perp}, \overline{\exp\big({-}{\rm Li}_{2}\big({-}z^{e_{i}}\big)\big)}\big)\,|\, i\in I\big\}
=\big\{\big(e_{\Pi i}^{\perp}, \exp\big({-}{\rm Li}_{2}\big({-}z^{e_{\Pi i}}\big)\big)\big)\,|\, i\in I\big\}.
\end{gather*}
Therefore up to splitting of incoming walls in $\mathfrak{D}_{\overline{{\bf s}}}$,
$\mathfrak{D}_{\overline{\mathfrak{p}_{+,-}},{\rm in}}=\mathfrak{D}_{\overline{{\bf s}},{\rm in}}$.
By Theorem~\ref{thm: uniqueness-incoming walls}, $\mathfrak{D}_{\overline{{\bf s}}}$
and~$\mathfrak{D}_{\overline{\mathfrak{p_{+,-}^{{\bf s}}}}}$ are
equivalent.
\end{proof}

\subsection[An illustrative example: $A_{3}$ folded to $B_{2}$]{An illustrative example: $\boldsymbol{A_{3}}$ folded to $\boldsymbol{B_{2}}$}

Let $I=\{1,2,3\}$. Let $N$ be a lattice of rank $3$ with a basis
$e_{1}$, $e_{2}$, $e_{3}$ and a $\mathbb{Z}$-bilinear skew-symmetric form
$\{\,,\,\}$ such that
\[
\{e_{1},e_{2}\}=\{e_{3},e_{2}\}=1,\qquad\{e_{1},e_{3}\}=0,
\]
The labelled basis $(e_{i})_{i\in I}$ defines a seed ${\bf s}$ of
type $A_{3}$. It has exchange matrix
\[
\epsilon=\begin{bmatrix}0 & 1 & 0\\
-1 & 0 & -1\\
0 & 1 & 0
\end{bmatrix}.
\]
Let $\Pi=\mathbb{Z}_{2}$ act on $I$ by interchanging the index~$1$
and~$3$. Clearly the action of $\Pi$ satisfies the conditions~\eqref{eq:-8},
so we can run our construction. The quotient lattice~$\overline{N}$
is of rank~$2$ with a~basis~$e_{\Pi1}$,~$e_{\Pi2}$. The $\mathbb{Z}$-linear
skew symmetric form on~$\overline{N}$ is defined as follows:
\[
\{e_{\Pi1},e_{\Pi2}\}=1.
\]
The sublattice $\overline{N}^{\circ}\subset\overline{N}$ has a basis
$d_{\Pi1}e_{\Pi1},e_{\Pi2}$ where $d_{\Pi1}=2$. The labelled basis
$(e_{\Pi i})_{\Pi i\in\overline{I}}$ defines a seed $\overline{{\bf s}}$
of type $B_{2}$ with exchange matrix
\[
\overline{\epsilon}=\begin{bmatrix}0 & 1\\
-2 & 0
\end{bmatrix}.
\]

\begin{figure}[t]\centering
\begin{tikzpicture}[scale=0.35)]
\begin{scope}[shift={(-5,25,0)}] \draw (0,0,0)--(5,0,0); \draw [->] (0,0,0)--(0,2.5,0); \draw [->] (0,0,0)--(-5,2.5,0); \draw [->] (0,0,0)--(-5,5,0); \draw [->] (0,0,0)--(-5,0,0); \draw (0,0,0)--(0,-2.5,0); \draw (-5,-2.5,0) node {\small ${\mathfrak{D}_{\overline{\mathbf{s}}}}$};
\draw [right hook->] (0,-4,0) -- (0,-10,0); \end{scope}
\begin{scope}[shift={(12,25,0)}] \draw (0,3.5,0) node {\small ${e_{\Pi1}^{\perp}}$}; \draw (0,0,0)--(5,0,0); \draw [->] (0,0)--(0,2.5,0);
\draw [->] (0,0)--(-5,0,0); \draw (0,0,0)--(0,-2.5,0); \draw (-4,1,0) node {\small ${e_{\Pi2}^{\perp}}$}; \draw (-4,-2.5,0) node {\small ${\mathfrak{D}_{\overline{\mathbf{s}},\mathrm{in}}}$}; \draw [right hook->] (0,-4,0) -- (0,-10,0);
\end{scope}
\begin{scope}[shift={(-5,0,0)}]
\fill [gray!50,opacity=0.4] (10,7.5,10) -- (-10,7.5,-10) -- (-10,-7.5,-10)-- (10,-7.5,10); \draw [->] (-2,-9,0) to [out=100,in=180] (1,-6,0); \draw (-3,-10,0) node {\small$\{x=z\}$};
\fill[blue!50,opacity=0.8] (0,0,0) rectangle (5,5,0); \fill[blue!50,opacity=0.8] (0,0,0) -- (0,5,0) --(0,5,5) -- (0,0,5); \fill[blue!50,opacity=0.8] (0,0,0) -- (5,0,0) -- (5,0,5) -- (0,0,5);
\fill[blue!50,opacity=0.8] (0,0,0) -- (0,5,0) -- (0,10,-5) -- (0,5,-5); \fill[blue!50,opacity=0.8](0,0,0)--(5,0,0)--(5,5,-5)--(0,5,-5);
\fill[blue!50,opacity=0.8] (0,0,0) -- (5,0,0) --(5,-5,0) --(0,-5,0); \fill[blue!50,opacity=0.8] (0,0,0) -- (0,0,5) -- (0,-5,5)--(0,-5,0);
\fill [blue!50,opacity=0.8] (0,0,0) -- (0,5,0) -- (-5,10,0)-- (-5,5,0); \fill [blue!50,opacity=0.8] (0,0,0) -- (0,0,5) -- (-5,5,5) -- (-5,5,0);
\fill [blue!50,opacity=0.6](0,0,0)--(0,-5,0)--(0,-5,-5)--(0,0,-5); \fill [blue!50,opacity=0.6](0,0,0)--(0,0,-5)--(5,0,-5)--(5,0,0);
\fill [blue!50,opacity=0.6](0,0,0)--(-5,0,0)--(-5,-5,0)--(0,-5,0); \fill [blue!50,opacity=0.6](0,0,0)--(-5,0,0)--(-5,0,5)--(0,0,5);
\fill [blue!50,opacity=0.6] (0,0,0)--(-5,0,0)--(-10,5,0)--(-5,5,0); \fill [blue!50,opacity=0.6] (0,0,0)--(-5,0,0)--(-5,0,5)--(0,0,5);
\fill [blue!50,opacity=0.6](0,0,0)--(-5,5,0)--(-5,10,-5)--(0,5,-5);
\fill [blue!50,opacity=0.6](0,0,0)--(0,0,-5)--(0,5,-10)--(0,5,-5); \fill [blue!50,opacity=0.6](0,0,0)--(0,0,-5)--(5,0,-5)--(5,0,0);
\fill [blue!50,opacity=0.4](0,0,0)--(-5,5,0)--(-10,10,-5)--(-5,5,-5); \fill [blue!50,opacity=0.4](0,0,0)--(0,5,-5)--(-5,10,-10)--(-5,5,-5);
\fill [blue!50,opacity=0.4](0,0,0)--(0,0,-5)--(-5,5,-10)--(-5,5,-5); \fill [blue!50,opacity=0.4](0,0,0)--(0,5,-5)--(-5,10,-10)--(-5,5,-5);
\fill [blue!50,opacity=0.4](0,0,0)--(-5,0,0)--(-5,0,-5)--(0,0,-5);
\fill [blue!50,opacity=0.3] (0,0,0)--(-5,0,0)--(-10,5,-5)--(-5,5,-5);
\fill [blue!50,opacity=0.3] (0,0,0)--(0,0,-5)--(-5,5,-10)--(-5,5,-5);
\fill [blue!50,opacity=0.3] (0,0,0)--(-5,5,-5)--(-10,5,-5)--(-5,0,0); \fill [blue!50,opacity=0.3] (0,0,0)--(-5,5,-5)--(-10,10,-5)--(-5,5,0);
\draw [red!65] (0,0,0)--(5,0,5); \draw [->,red!65] (0,0,0)--(0,2.5,0); \draw [->,red!65] (0,0,0)--(-5,5,-5); \draw [->,red!65] (0,0,0)--(-5,2.5,-5); \draw [->,red!65] (0,0,0)--(-5,0,-5);
\draw [red!65] (0,0,0)--(0,-2.5,0); \end{scope}
\begin{scope}[shift={(13,0,0)}] \fill [gray!50,opacity=0.4] (10,7.5,10) -- (-10,7.5,-10) -- (-10,-7.5,-10)-- (10,-7.5,10); \fill[blue!50,opacity=0.8] (-5,-5,0) rectangle (5,5,0); \fill[blue!50,opacity=0.8] (-5,0,5) -- (-5,0,-5) --(5,0,-5) -- (5,0,5); \fill[blue!50,opacity=0.8] (0,-5,5) -- (0,-5,-5) -- (0,5,-5) -- (0,5,5);
\draw (2.3,7.5,0) node {\small${e_{1}^{\perp}} $}; \draw (5.5,5.5,0) node {\small${e_{3}^{\perp}} $}; \draw (5.5,-0.2,-5.5) node {\small${e_{2}^{\perp}} $}; \draw [red!65] (0,0,0)--(5,0,5); \draw [->,red!65] (0,0,0)--(0,2.5,0);
\draw [->,red!65] (0,0,0)--(-5,0,-5);
\draw [red!65] (0,0,0)--(0,-2.5,0); \end{scope} \end{tikzpicture} \caption{Embedding of scattering diagram ${\mathfrak{D}_{\overline{\mathbf{s}}}}$ of type $B_{2}$ into ${\mathfrak{D}_{\mathbf{s}}}$ of type $A_{3}$.} \label{Folding}
\end{figure}
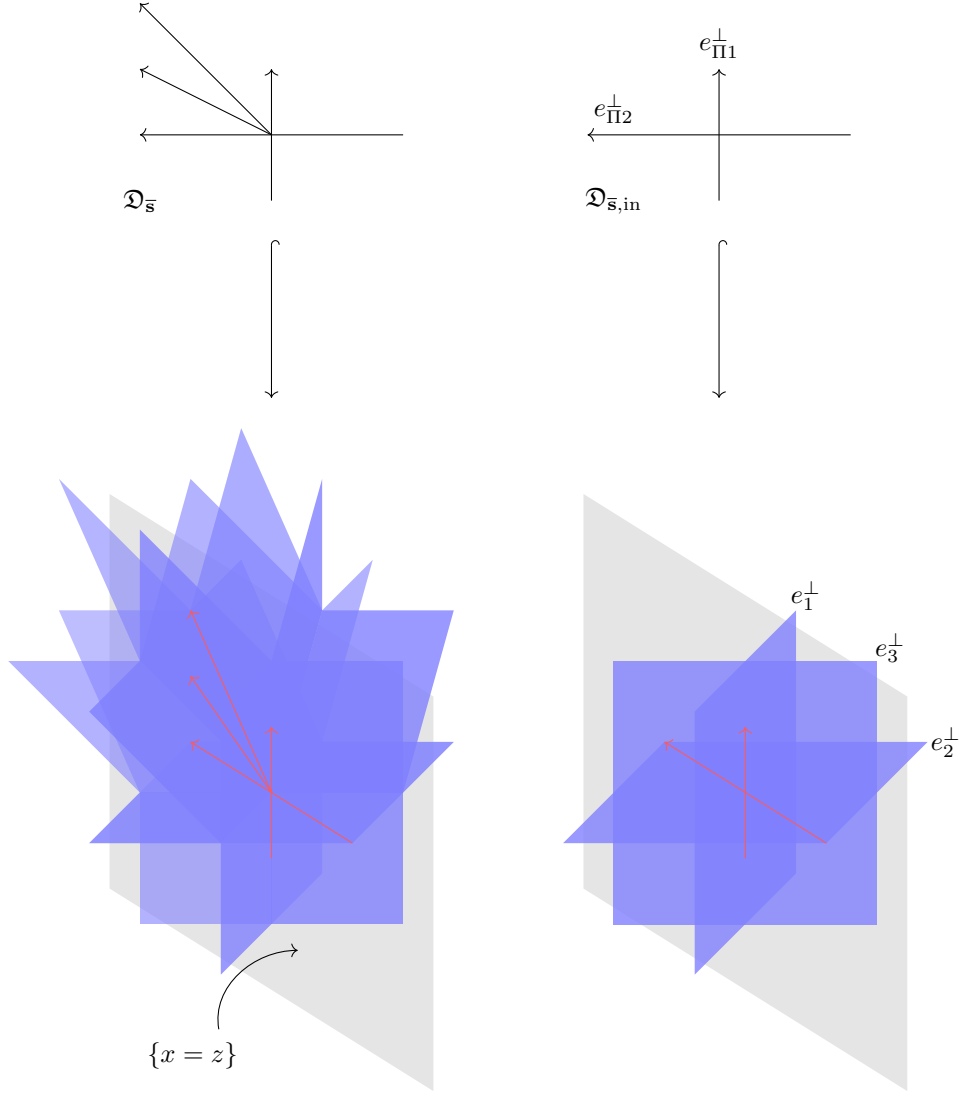

\looseness=-1 The group $\Pi$ acts on $\mathfrak{D}_{{\bf s}}$ in such a way that
on the level of the support of $\mathfrak{D}_{{\bf s}}$, its action
is reflection with respect to the plane $\{x-z=0\}\subseteq M_{\mathbb{R}}$.
The quotient map $q\colon N_{\mathbb{R}}\twoheadrightarrow\overline{N}_{\mathbb{R}}$
induces an embedding on the dual spaces $q^{*}\colon \overline{M}_{\mathbb{R}}\hookrightarrow M_{\mathbb{R}}$ that embeds $\mathfrak{D}_{\overline{{\bf s}}}$ into the plane $\{x=z\}$.
Through this embedding, we could recover $\mathfrak{D}_{\overline{{\bf s}}}$
via walls in $\mathfrak{D}_{{\bf s}}$ whose intersection with the
plane $\{x-z=0\}$ is dimension $1$. We can see this immediately
on the level of initial walls, observing that
\begin{gather*}
q^{*}(e_{\Pi1}^{\perp}) =e_{1}^{\perp}\bigcap\{x=z\}=e_{3}^{\perp}\bigcap\{x=z\}=e_{1}^{\perp}\cap e_{3}^{\perp},\\
\exp\big({-}d_{\Pi1}{\rm Li}_{2}\big({-}z^{e_{\Pi1}}\big)\big)=\overline{\exp\big({-}{\rm Li}_{2}\big({-}z^{e_{1}}\big)\big)\exp\big({-}{\rm Li}_{2}\big({-}z^{e_{3}}\big)\big)}
\end{gather*}
and
\begin{gather*}
q^{*}(e_{\Pi2}^{\perp})=e_{2}^{\perp}\bigcap\{x=z\},\qquad
\exp\big({-}{\rm Li}_{2}\big({-}z^{e_{\Pi2}}\big)\big)=\overline{\exp\big({-}{\rm Li}_{2}\big({-}z^{e_{2}}\big)\big)}.
\end{gather*}

\subsection{Application of folding to theta functions}

In this subsection, we want to apply the folding technique to canonical
bases for $\mathcal{A}_{{\rm prin}}$-cluster varieties. Let $\mathfrak{D}_{{\bf s}}^{\mathcal{A}_{{\rm prin}}}$ be the scattering diagram for the $\mathcal{A}_{{\rm prin}}$-cluster
variety associated to the initial seed ${\bf s}$ as defined in \cite[Construction~2.11]{GHKK}. Via the canonical projection map
\begin{align*}
\rho^{T}\colon \ M_{\mathbb{R}}\oplus N_{\mathbb{R}} & \rightarrow M_{\mathbb{R}},\\
(m,n) & \mapsto m
\end{align*}
walls of $\mathfrak{D}_{{\bf s}}^{\mathcal{A}_{{\rm prin}}}$ are inverse image of walls in $\mathfrak{D}_{{\bf s}}$ with the same attached Lie group elements. So to study wall-crossing in $\mathfrak{D}_{{\bf s}}^{\mathcal{A}_{{\rm prin}}}$ along a~path~$\gamma$, it suffices to study wall-crossing in $\mathfrak{D}_{{\bf s}}$ along the projection of path~$\gamma$. This is the philosophy behind this subsection.

Recall that $\overline{N}^{\circ}\subset\overline{N}$ is the sublattice
generated by $(d_{\Pi i}e_{\Pi i})_{\Pi i\in\overline{I}}$. Define
$\overline{M}^{\circ}={\rm Hom}\big(\overline{N}^{\circ},\allowbreak\mathbb{Z}\big)$.
The quotient homomorphism $q\colon N_{\mathbb{R}}\twoheadrightarrow\overline{N}_{\mathbb{R}}$
have a natural section
\begin{align*}
s\colon \ \overline{N}_{\mathbb{R}} & \hookrightarrow N_{\mathbb{R}}, \\
e_{\Pi i} & \mapsto\frac{1}{d_{\Pi i}}\sum_{i'\in\Pi i}e_{i'}.
\end{align*}
In particular, given $n$ in $N$, we have
\[
s(\overline{n})=\frac{1}{|\Pi n|}\sum_{n'\in\Pi n}n'.
\]
The section $s$ restricts to an embedding $s\colon \overline{N}^{\circ}\hookrightarrow N$
which induces an quotient homomorphism between dual lattices
\begin{align*}
s^{*}\colon \ M & \twoheadrightarrow\overline{M}^{\circ},\\
e_{i}^{*} & \mapsto\frac{e_{\Pi i}^{*}}{d_{\Pi i}}.
\end{align*}
In particular, notice that $s^{*}(\{n,\cdot\})=\{\overline{n},\cdot\}$ since for any~$i$ in~$I$,
\begin{align*}
\big\langle e_{\Pi i},s^{*}(\{n,\cdot\})\big\rangle =\big\langle s (e_{\Pi i} ),\{n,\cdot\}\big\rangle
 =\left\langle \frac{1}{|\Pi i|}\sum_{i'\in\Pi i}e_{i'}, \{ n,\cdot \} \right\rangle
 =\frac{1}{|\Pi i|}\sum_{i'\in\Pi i}\{n,\cdot e_{i'}\}
 =\{\overline{n},e_{\Pi i}\}.
\end{align*}
The quotient map of lattices
\begin{align*}
(s^{*},q)\colon \ M\oplus N & \twoheadrightarrow\overline{M}^{\circ}\oplus\overline{N},\\
(m,n) & \mapsto\left(s^{*}(m),q(n)\right)
\end{align*}
induces an embedding of algebraic torus $\phi_{(s^{*},q)}\colon T_{\overline{N}^{\circ}\oplus\overline{M}}\hookrightarrow T_{N\oplus M}$.
In the rest of this section, we often denote the image of $(m,n)$
in $M\oplus N$ under $(s^{*},q)$ as $(\overline{m},\overline{n})$
and use these images to represent elements in~$\overline{M}^{\circ}\oplus\overline{N}$.

Let $\sigma\subset M_{\mathbb{R}}\oplus N_{\mathbb{R}}$ be the cone
generated by $ \{ (\{e_{i},\cdot\},e_{i} ) \} _{i\in I}$
and $ \{ (e_{i}^{*},0 ) \} _{i\in I}$. The strictly
convex top-dimensional cone~$\sigma$ yields a monoid $P=\sigma\bigcap (M\oplus N)$
and a monoid ring~$\mathbb{C}[P]$. Complete $\mathbb{C}[P]$ with
respect to the ideal~$J$ generated by $P\setminus P^{\times}$, we
get $\widehat{\mathbb{C}[P]}$. Recall that in~\cite{GHKK}, the Lie
algebra~$\mathfrak{g}$ acts on~$\mathbb{C}[P]$ as \emph{log derivations}
as follows:
\begin{gather*}
z^{n'}\circ z^{(\{ n,\cdot\},n)} =\{ n,n'\} \cdot z^{(\{ n',\cdot\},n')}\cdot z^{(\{ n,\cdot\},n)},\\
z^{n'}\circ z^{(m,0)} =\langle n',m\rangle \cdot z^{(\{ n',\cdot\},n')}\cdot z^{(m,0)}.
\end{gather*}
Since the action of $z^{n'}$ is a derivation, it satisfies the
Leibniz's rule, i.e.,{\samepage
\[
z^{n'}\circ(f\cdot g)=\big(z^{n'}\circ f\big)\cdot g+f\cdot\big(z^{n'}\circ g\big).
\]
The action of $\mathfrak{g}$ on $\mathbb{C}[P]$ induces an action of the Lie group $G$ on $\widehat{\mathbb{C}[P]}$.}

Similarly, let $\overline{\sigma}\subset\overline{M}_{\mathbb{R}}\oplus\overline{N}_{\mathbb{R}}$
be the cone generated by $\{(\{e_{\Pi i},\cdot\},e_{\Pi i})\} _{\Pi i\in\overline{I}}$
and $\big\{\big(\frac{e_{\Pi i}^{*}}{d_{\Pi i}},0\big)\big\}_{\Pi i\in\overline{I}}$
which yields a~monoid $\overline{P}$ and a monoid ring $\mathbb{C}[\overline{P}]$.
Complete $\mathbb{C}[\overline{P}]$ with respect to the ideal $\overline{J}$
generated by $\overline{P}\setminus\overline{P}^{\times}$, we get
$\widehat{\mathbb{C}[\overline{P}]}$. Notice that $\phi_{(s^{*},q)}^{*}$
induce an surjection $\widehat{\mathbb{C}[P]}\twoheadrightarrow\widehat{\mathbb{C}[\overline{P}]}$.
The Lie algebra $\overline{\mathfrak{g}}$ acts on $\mathbb{C}[\overline{P}]$
as log derivations and we have an induced action of~$\overline{G}$
on~$\widehat{\mathbb{C}[\overline{P}]}$.

Now given a generic path $\gamma$ in $\mathfrak{D}_{{\bf s}}$, a
priori, the path-ordered product $\mathfrak{p}_{\gamma}$ only induces
an automorphism $\mathfrak{p}_{\gamma}\colon \widehat{\mathbb{C}[P]}\rightarrow\widehat{\mathbb{C}[P]}$.
However, it can happen that~$\mathfrak{p}_{\gamma}$ is rational,
i.e., it induces an automorphism of the function field~$\mathbb{C}(M\oplus N)$.
For example, let~$\mathfrak{p}_{\gamma}$ be the path starting from~$\mathcal{C}_{{\bf s}}^{+}$ and crossing only the wall
\[
\big(e_{i}^{\perp},\exp\big({-}{\rm Li}_{2}\big({-}z^{e_{i}}\big)\big)\big)\in\mathfrak{D}_{{\bf s},{\rm in}}.
\]
Then one can check that action of $\mathfrak{p}_{\gamma}$ on $\mathbb{C}(M\oplus N)$
is as follows:
\begin{gather*}
z^{(m,0)} \mapsto z^{(m,0)}\big(1+z^{(\{ e_{i},\cdot\},e_{i})}\big)^{\langle e_{i},m\rangle},\\
z^{(0,n)} \mapsto z^{(0,n)}.
\end{gather*}
This is the inverse mutation for $\mathcal{A}_{{\rm prin},{\bf s}}$.
Similarly, let $\mathfrak{p}_{\overline{\gamma}}$ be the path starting
from $\mathcal{C}_{{\bf \overline{s}}}^{+}$ and crossing only the
wall
\[
\big(e_{\Pi i}^{\perp},\exp\big({-}d_{\Pi i}{\rm Li}_{2}\big({-}z^{e_{\Pi i}}\big)\big)\big)\in\mathfrak{D}_{\overline{{\bf s}},{\rm in}}
\]
One can check that the action of $\mathfrak{p}_{\overline{\gamma}}$
on $\mathbb{C}\big(\overline{M}^{\circ}\oplus\overline{N}\big)$
agrees with the inverse mutation for $\mathcal{A}_{{\rm prin},\overline{{\bf s}}}$.

\begin{Remark}Unlike~\cite{GHKK}, we consider a single scattering diagram with
wall-crossing automorphisms in the Lie group. By letting the Lie group
acting on different completed rings, we get wall-crossings for both
$\mathcal{X}$ and $\mathcal{A}_{{\rm prin}}$ spaces.
\end{Remark}

At the end of this subsection, we will prove the following proposition:
\begin{Proposition}\label{prop:3.3}Given an element~$g$ in~$G$, denote by $\overline{g}$
its image under the quotient homomorphism $G\twoheadrightarrow\overline{G}$.
Let $G^{\Pi}$ be the subgroup of $G$ invariant under the action
of $\Pi$. Suppose~$g$ is contained in $G^{\Pi}$ and $g$ induces
an automorphism of the function field $\mathbb{C}(M\oplus N)$, then
$\overline{g}$ induces an automorphism of the functional field $\mathbb{C}\big(\overline{M}^{\circ}\oplus\overline{N}\big)$.
Moreover, we have the following commutative diagram
\[
\xymatrix{T_{\overline{N}^{\circ}\oplus\overline{M}}\ar@{^{(}->}[r]^{\phi_{(s^{*},q)}}\ar@{-->}[d]^{\varphi_{g}} & T_{N\oplus M}\ar@{-->}[d]^{\varphi_{\overline{g}}}\\
T_{\overline{N}^{\circ}\oplus\overline{M}}\ar@{^{(}->}[r]^{\phi_{(s^{*},q)}} & T_{N\oplus M},
}
\]
where $\varphi_{g}$ and $\varphi_{\overline{g}}$ are birational
automorphisms induced by $g$ and $\overline{g}$ respectively.
\end{Proposition}

To see the application of the above proposition, we need the following
lemma:
\begin{Lemma}
\label{lem:3.4}$\mathfrak{p}_{+,-}^{{\bf s}}$ is contained in $G^{\Pi}$.
\end{Lemma}

\begin{proof}
By Lemma \ref{lem:2.18}, it suffices to check that $\mathfrak{D}_{{\bf s},{\rm in}}$
is invariant under the action of $\Pi$. Indeed,
\[
\mathfrak{D}_{{\rm {\bf s},{\rm in}}}=\big\{\big(e_{i}^{\perp},\exp\big({-}{\rm Li}_{2}\big({-}z^{e_{i}}\big)\big)\big)\,|\, i\in I=I_{{\rm uf}}\big\}.
\]
Given $\pi$ in $G$,
\[
\pi\cdot\big(e_{i}^{\perp},\exp\big({-}{\rm Li}_{2}\big({-}z^{e_{i}}\big)\big)\big)=\big(e_{\pi\cdot i}^{\perp},\exp\big({-}{\rm Li}_{2}\big({-}z^{e_{\pi\cdot i}}\big)\big)\big).
\]
Hence, $\mathfrak{D}_{{\bf s},{\rm in}}$ is invariant under the action of~$\Pi$.
\end{proof}

\begin{Proposition}\label{prop:Given-a-generic}Given a generic path $\overline{\gamma}$
in $\mathfrak{D}_{\overline{{\bf s}}}$, $q^{*}(\overline{\gamma})$
can be perturbed to a generic path $\gamma$ with the same end points
in $\mathfrak{D}_{{\bf s}}$ such that $\mathfrak{p}_{\gamma}$ is
contained in $G^{\Pi}$ and $\mathfrak{p}_{\overline{\gamma}}=\overline{\mathfrak{p}_{\gamma}}$
where $\overline{\mathfrak{p_{\gamma}}}$ is the image of $\mathfrak{p}_{\gamma}$
under the quotient homomorphism $G\twoheadrightarrow\overline{G}$.
\end{Proposition}

\begin{proof}It suffices to prove the statement up to each order~$k$. It follows
from the proof of Theorem~\ref{thm:2.21} that $q^{*}(\overline{\gamma})$
can be perturbed to a generic path $\gamma$ such that $\mathfrak{p}_{\overline{\gamma}}=\overline{\mathfrak{p}_{\gamma}}$.
So it remains to show that $\mathfrak{p}_{\gamma}$ is contained in
$G^{\Pi}$. Furthermore, it suffices to prove for the case where $\overline{\gamma}$
has only one wall-crossing. Indeed, suppose at time $t_{1}$, $\overline{\gamma}$
crosses a non-trivial wall in $\mathfrak{D}_{\overline{{\bf s}}}$.
Let $x=\overline{\gamma}(t_{1})$. By finiteness of $\mathfrak{D}_{\overline{{\bf s}}}$
up to order $k$, we could assume that $x$ is a general point in
$\overline{M}_{\mathbb{R}}$ and at time $t_{1}$, $\overline{\gamma}$
crosses a single wall in $\mathfrak{D}_{\overline{{\bf s}}}$. By
Lemma~\ref{lem:2.18}, $q^{*}(x)$ is not contained in the boundary
of any wall in $\mathfrak{D}_{{\bf s}}$. Let $D=\{(\mathfrak{d}_{1},g_{\mathfrak{d}_{1}}),\dots,(\mathfrak{d}_{j},g_{\mathfrak{d}_{j}})\}$
be the collection of all nontrivial walls in $\mathfrak{D}_{{\bf s}}$
whose support contains $q^{*}(x)$. For each~$l$, let $n_{l}$ be
the primitive element in $N^{+}$ such that $n_{l}^{\perp}\supset\mathfrak{d}_{l}$. By finiteness of $\mathfrak{D}_{{\bf s}}$ up to order $k$, we
could assume that walls in $D$ are contained in distinct hyperplanes
in $M_{\mathbb{R}}$ and given any general point $x_{l}$ in $\mathfrak{d}_{l}$,
\[
g_{x_{l}} (\mathfrak{D}_{{\bf s}} )=\big(\mathfrak{p}_{+,-}^{{\bf s}}\big)_{x_{l}}^{0}=g_{\mathfrak{d}_{l}},
\]
where
\[
\mathfrak{p}_{+,-}^{{\bf s}}=\big(\mathfrak{p}_{+,-}^{{\bf s}}\big)_{x_{l}}^{+}\circ\big(\mathfrak{p}_{+,-}^{{\bf s}}\big)_{x_{l}}^{_{0}}\circ\big(\mathfrak{p}_{+,-}^{{\bf s}}\big)_{x_{l}}^{-}
\]
 is the unique factorization of $\mathfrak{p}_{+,-}^{{\bf s}}$ with
respect to $x_{l}$ (cf.\ \cite[proof of Theorem~1.17]{GHKK}).
Since $\mathfrak{p}_{+,-}^{{\bf s}}$ is contained in~$G^{\Pi}$,
by Proposition~\ref{prop:2.14}, we could assume that~$D$ is invariant
under the action of~$\Pi$ on~$\mathfrak{D}_{{\bf s}}$. By Lemma~\ref{lem:2.18}, $q^{*}(x)$ is contained in the abelian joint of~$\mathfrak{D}_{{\bf s}}$. So $g_{\mathfrak{d}_{l}}$ $(1\leq l\leq j)$
commute with each other. Therefore, $\prod_{l}g_{\mathfrak{d}_{l}}$
is contained in~$G^{\Pi}$. It follows proof of Theorem~\ref{thm:2.21}
that we could assume that $\gamma$ crosses walls in~$D$ transversally
and does not cross any non-trivial wall in~$\mathfrak{D}_{{\bf s}}\setminus D$.
Therefore,
\begin{gather*}
\mathfrak{p}_{\gamma}=\prod_{l}g_{\mathfrak{d}_{l}}\in G^{\Pi}.\tag*{\qed}
\end{gather*}\renewcommand{\qed}{}
\end{proof}

Here is an important application of Proposition~\ref{prop:3.3} together
with Lemma~\ref{lem:3.4}. Let $\mathcal{A}_{{\rm prin},{\bf s}}$
(resp.~$\mathcal{A}_{{\rm prin},\overline{{\bf s}}}$ ) be the cluster
variety of type $\mathcal{A}_{{\rm prin}}$ using ${\bf s}$ (resp.~$\overline{{\bf s}}$) as the initial seed. The seed~$\overline{{\bf s}}$ gives us an identification
\[
\mathcal{A}_{{\rm prin},\overline{{\bf s}}}^{\vee}\big(\mathbb{R}^{T}\big)\simeq\overline{M}_{\mathbb{R}}\oplus\overline{N}_{\mathbb{R}}\text{ and }\mathcal{A}_{{\rm prin},\overline{{\bf s}}}^{\vee}\big(\mathbb{Z}^{T}\big)=\overline{M}^{\circ}\oplus\overline{N}.
\]
In~\cite{GHKK}, given an integer tropical point $(\overline{m},\overline{n})$
in $\mathcal{A}_{{\rm prin},\overline{{\bf s}}}^{\vee}\big(\mathbb{Z}^{T}\big)$,
there is a \emph{theta function} $\vartheta_{(\overline{m},\overline{n})}$
associated to it. Given a generic point~$Q$ in $\mathcal{A}_{{\rm prin},\overline{{\bf s}}}^{\vee}$,
that is, a point in $\mathcal{A}_{{\rm prin},\overline{{\bf s}}}^{\vee}\big(\mathbb{R}^{T}\big)\setminus\allowbreak{\rm supp}\big(\mathfrak{D}_{\overline{{\bf s}}}^{\mathcal{A}_{{\rm prin}}}\big)$
with irrational coordinates, denote by $\vartheta_{\left(\overline{m},\overline{n}\right),Q}$
the expansion of the theta function at $Q$. The set of all integer
tropical points $(\overline{m},\overline{n})$ such that $\vartheta_{(\overline{m},\overline{n}),Q}$
is a Laurent polynomial with non-negative integer coefficients (\emph{positive
Laurent polynomial} for short) for a generic point~$Q$ in~$\mathcal{C}_{\overline{{\bf s}}}^{+}$\footnote{We denote by $\mathcal{C}_{\overline{{\bf s}}}^{\pm}$ the pull-back
of positive and negative chamber in $\mathfrak{D}_{\overline{{\bf s}}}$
via~$\rho^{T}$.} is called the \emph{theta set} of $\mathcal{A}_{{\rm prin},\overline{{\bf s}}}$,
or $\Theta(\mathcal{A}_{{\rm prin},\overline{{\bf s}}})$. As we have
mentioned at the beginning of this subsection, wall-crossing along
a path $\gamma$ in $\mathfrak{D}_{\overline{{\bf s}}}^{\mathcal{A}_{{\rm prin}}}$
is the same as wall-crossing along the projection of~$\gamma$ under
$\rho^{T}$ in $\mathfrak{D}_{\overline{{\bf s}}}$. In particular,
let~$\gamma$ be a generic path with initial point in~$\mathcal{C}_{\overline{{\bf s}}}^{+}$
and final point in~$\mathcal{C}_{\overline{{\bf s}}}^{-}$, then
\[
\mathfrak{p}_{\gamma,\mathfrak{D}_{\overline{{\bf s}}}^{\mathcal{A}_{{\rm prin}}}}=\mathfrak{p}_{+,-}^{\overline{{\bf s}}}.
\]
Recall that $\mathcal{C}_{\overline{{\bf s}}}^{+}\big(\mathbb{Z}^{T}\big)$
is always contained in $\Theta(\mathcal{A}_{{\rm prin},\overline{{\bf s}}})$.
Since $\Theta(\mathcal{A}_{{\rm prin},\overline{{\bf s}}})$ is closed
under addition after the identification $\mathcal{A}_{{\rm prin},\overline{{\bf s}}}^{\vee}\big(\mathbb{Z}^{T}\big)\simeq\overline{M}^{\circ}\oplus\overline{N}$
by Theorem 7.5 of \cite{GHKK}, we conclude that to prove $\Theta(\mathcal{A}_{{\rm prin},\overline{{\bf s}}})=\mathcal{A}_{{\rm prin},\overline{{\bf s}}}^{\vee}\big(\mathbb{Z}^{T}\big)$,
it suffices to prove that $\mathcal{C}_{{\bf \overline{{\bf s}}}}^{-}\big(\mathbb{Z}^{T}\big)$
is contained in $\Theta(\mathcal{A}_{{\rm prin},\overline{{\bf s}}})$.
Given $\left(\overline{m},0\right)$ in $\mathcal{C}_{\overline{{\bf s}}}^{-}\big(\mathbb{Z}^{T}\big)$,
$(\overline{m},0)$ is contained in $\Theta(\mathcal{A}_{{\rm prin},\overline{{\bf s}}})$
if and only if $\vartheta_{(\overline{m},0),Q}$ is a positive Laurent
polynomial for a generic point $Q$ in $\mathcal{C}_{\overline{{\bf s}}}^{+}$. Let $Q_{0}$ be a generic point in $\mathcal{C}_{\overline{{\bf s}}}^{-}$.
We have $\vartheta_{(\overline{m},0),Q_{0}}=z^{(\overline{m},0)}$.
By \cite[Theorem 3.5]{GHKK}, for $(\overline{m},0)$ in $\mathcal{C}_{\overline{{\bf s}}}^{-}\big(\mathbb{Z}^{T}\big)$,
\[
\vartheta_{(\overline{m},0),Q}=\mathfrak{p}_{-,+}^{\overline{{\bf s}}}(\vartheta_{(\overline{m},0),Q_{0}})=\mathfrak{p}_{-,+}^{\overline{{\bf s}}}\big(z^{(\overline{m},0)}\big).
\]
Now suppose $\Theta(\mathcal{A}_{{\rm prin},{\bf s}})=\mathcal{A}_{{\rm prin},{\bf s}}^{\vee}\big(\mathbb{Z}^{T}\big)$. In particular, $\mathfrak{p}_{-,+}^{{\bf s}}$ induces an automorphism
of the function field $\mathbb{C}(M\oplus N)$ and $\mathfrak{p_{-,+}^{{\bf s}}}\big(z^{(m,0)}\big)$
is a positive Laurent polynomial. Thus, by Lemma~\ref{lem:3.4}, $\mathfrak{p}_{-,+}^{{\bf s}}$
satisfies the assumption in Proposition~\ref{prop:3.3} and we have
the following commutative diagram:
\begin{gather}\begin{split}&
\xymatrix{\mathbb{C}(M\oplus N)\ar@{->>}[r]^{\phi_{(s^{*},q)}^{*}}\ar@{-->}[d]^{\mathfrak{p}_{-,+}^{{\bf s}}} & \mathbb{C}\big(\overline{M}^{\circ}\oplus\overline{N}\big)\ar@{-->}[d]^{\mathfrak{p}_{-,+}^{\overline{{\bf s}}}}\\
\mathbb{C}(M\oplus N)\ar@{->>}[r]^{\phi_{(s^{*},q)}^{*}} & \mathbb{C}\big(\overline{M}^{\circ}\oplus\overline{N}\big).
}\end{split}
\label{eq:-9}
\end{gather}
Given $(m,0)$ in $\mathcal{C}_{{\bf s}}^{-}\big(\mathbb{Z}^{T}\big)$, if
$\mathfrak{p}_{-,+}^{{\bf s}}\big(z^{(m,0)}\big)$ is a positive Laurent polynomial,
then $\mathfrak{p}_{-,+}^{\overline{{\bf s}}}\big(z^{(\overline{m},0)}\big)=\phi_{(s^{*},q)}^{*}\circ\mathfrak{p_{-,+}^{{\bf s}}}\big(z^{(m,0)}\big)$
is also a~positive Laurent polynomial and therefore $(\overline{m},0)$
is contained in~$\Theta(\mathcal{A}_{{\rm prin},\overline{{\bf s}}})$.
Since $\Theta(\mathcal{A}_{{\rm prin},\overline{{\bf s}}})$ is closed
under translation by $\overline{N}$, if $(\overline{m},0)$ is in
$\Theta(\mathcal{A}_{{\rm prin},\overline{{\bf s}}})$ for all $(\overline{m},0)$
in $\mathcal{C}_{\overline{{\bf s}}}^{+}$, then $\mathcal{C}_{\overline{{\bf s}}}^{+}$
is contained in $\Theta(\mathcal{A}_{{\rm prin},\overline{{\bf s}}}^{\vee})$.
Hence we prove the main theorem of this subsection:
\begin{Theorem}\label{thm:3.5}If $\Theta(\mathcal{A}_{{\rm prin},{\bf s}})=\mathcal{A}_{{\rm prin},{\bf s}}^{\vee}\big(\mathbb{Z}^{T}\big)$
holds, then $\Theta(\mathcal{A}_{{\rm prin},\overline{{\bf s}}})=\mathcal{A}_{{\rm prin},\overline{{\bf s}}}^{\vee}\big(\mathbb{Z}^{T}\big)$.
\end{Theorem}

The rest of this subsection is devoted to the proof of Proposition~\ref{prop:3.3}. We have $G^{\Pi}=\underleftarrow{\lim}\exp\big(\big(\mathfrak{g}^{\Pi}\big)^{\leq k}\big)$.
Notice that $\mathfrak{g}^{\Pi}$ has a set of basis element as follows:
\[
\mathfrak{g}^{\Pi}=\bigoplus_{\Pi n\,|\, n\in N^{+}}\mathbb{C}\cdot\bigg(\sum_{n'\in\Pi n}z^{n'}\bigg).
\]
Given $(m,n)$ in $M\oplus N$, the Lie groups $G$ and $\overline{G}$
acts on $z^{(m,n)}\widehat{\mathbb{C}[P]}$ and $z^{(\overline{m},\overline{n})}\widehat{\mathbb{C}[\overline{P}}]$
respectively. Moreover, $\phi_{\left(s^{*},q\right)}^{*}$ induce
an surjection $z^{(m,n)}\widehat{\mathbb{C}[P]}\twoheadrightarrow z^{(\overline{m},\overline{n})}\widehat{\mathbb{C}[\overline{P}]}$. Given any~$g$ in~$G^{\Pi}$, we want to show that
\begin{gather}
\overline{g}\cdot z^{(\overline{m},\overline{n})}=\phi_{(s^{*},q)}^{*}\big(g\cdot z^{(m,n)}\big).\label{eq:-3}
\end{gather}
It suffices to show that the equation~\eqref{eq:-3} is true for basis element $\sum\limits_{n'\in\Pi n}z^{n'}$ in~$\mathfrak{g}^{\Pi}$. Indeed,
\begin{gather*}
 \phi_{(s^{*},q)}^{*}\left[\left(\sum_{n'\in\Pi n}z^{n'}\right)\circ z^{(m,n)}\right]
= \phi_{(s^{*},q)}^{*}\left[z^{(m,n)}\cdot\left(\sum_{n'\in\Pi n}\langle n',m\rangle z^{(\{n',\cdot\},n')}\right)\right]\\
\hphantom{\phi_{(s^{*},q)}^{*}\left[\left(\sum_{n'\in\Pi n}z^{n'}\right)\circ z^{(m,n)}\right]}{} = z^{(\overline{m},\overline{n})}\cdot z^{(\{ \overline{n},\cdot\},\overline{n})}\cdot\sum_{n'\in\Pi n} \langle n',m \rangle,
\end{gather*}
and
\begin{align*}
\big(|\Pi n|\cdot z^{\overline{n}}\big)\circ\big[\phi_{(s^{*},q)}^{*}\big(z^{(m,n)}\big)\big] &
=\big(|\Pi n|\cdot z^{\overline{n}}\big)\circ z^{(\overline{m},\overline{n})}
 =|\Pi n|\cdot z^{(\{ \overline{n},\cdot\},\overline{n})}\langle \overline{n},s^{*}(m)\rangle z^{(\overline{m},\overline{n})}\\
 & =|\Pi n|\cdot z^{(\{ \overline{n},\cdot\},\overline{n})}\langle s(\overline{n}),m\rangle z^{(\overline{m},\overline{n})}\\
 & =|\Pi n|\cdot z^{(\{ \overline{n},\cdot\},\overline{n})}\left\langle \frac{1}{|\Pi n|}\sum_{n'\in\Pi n}n',m\right\rangle z^{(\overline{m},\overline{n})}\\
 & =z^{(\overline{m},\overline{n})}\cdot z^{(\{ \overline{n},\cdot\},\overline{n})}\cdot\sum_{n'\in\Pi n}\langle n',m\rangle.
\end{align*}
Now given $g$ in $G^{\Pi}$, if $g$ induces an automorphism of $\mathbb{C}(M\oplus N)$, the equation~\eqref{eq:-3} implies that~$\overline{g}$ induces an automorphism of $\mathbb{C}\big(\overline{M}^{\circ}\oplus\overline{N}\big)$. Thus we proved Proposition~\ref{prop:3.3}.

\section[Building $\mathcal{A}_{{\rm prin}}$ using all chambers in the scattering diagram]{Building $\boldsymbol{\mathcal{A}_{{\rm prin}}}$ using all chambers in the scattering diagram} \label{sec:Building--using}

Let $\mathbf{s}_{0}$ be an initial seed. In this section, we allow
seed data of ${\bf s}_{0}$ to be general, that is, there can be frozen
variables and $N^{\circ}$ can be properly contained in~$N$. The
algebraic torus $T_{N^{\circ}\oplus M}$ has character lattice $M^{\circ}\oplus N$.
Let~$\Delta_{{\bf s}_{0}}^{+}\subset\mathcal{A}_{{\rm prin},{\bf s}_{0}}^{\vee}\big(\mathbb{R}^{T}\big)$
be the cluster complex. Let $\mathfrak{T}_{{\bf s}_{0}}$ be the oriented
tree whose vertices parametrize seeds mutationally equivalent to ${\bf s}_{0}$.
Let~$\Delta_{{\bf s}_{0}}^{-}$ be the complex consisting of all negative
chambers $\{\mathcal{C}_{v}^{-}\}_{v\in\mathfrak{T}_{{\bf s}_{0}}}$
in $\mathcal{A}_{{\rm prin},{\bf s}_{0}}^{\vee}\big(\mathbb{R}^{T}\big)$.
Assume that~$\mathcal{C}_{{\bf s}_{0}}^{-}$ is not contained in $\Delta_{{\bf s}_{0}}^{+}$,
therefore $\Delta_{{\bf s}_{0}}^{+}$ and $\Delta_{{\bf s}_{0}}^{-}$
are two distinct subfans in the scattering diagram, but $\mathfrak{p}_{+,-}^{{\bf s}_{0}}$
is rational, that is, $\mathfrak{p}_{+,-}^{{\bf s}_{0}}$ acts as
an automorphism of the field of fractions $\mathbb{C}(M^{\circ}\oplus N)\rightarrow\mathbb{C}(M^{\circ}\oplus N)$.
In \cite[Section~4]{GHKK}, gluing torus charts corresponding to each
chamber in the cluster complex via birational maps induced by wall-crossings
in the scattering diagram, they built a positive space~$\mathcal{A}_{{\rm prin},{\bf s}_{0}}^{{\rm scat}}$
that is isomorphic to the $\mathcal{A}_{{\rm prin}}$ cluster variety.
In this paper, under the assumption that~$\mathfrak{p}_{+,-}^{{\bf s}_{0}}$
acts as a birational automorphism of the field of fractions~$\mathbb{C}(M^{\circ}\oplus N)$,
we will build a positive space $\tilde{\mathcal{A}}_{{\rm prin},{\bf s}_{0}}^{{\rm scat}}$
using all the chambers in $\Delta_{{\bf s}_{0}}^{+}\bigcup\Delta_{{\bf s}_{0}}^{-}$.
Observe that $\mathcal{A}_{{\rm prin},{\bf s}_{0}}^{{\rm scat}}$
is contained in $\tilde{\mathcal{A}}_{{\rm prin},{\bf s}_{0}}^{{\rm scat}}$.
Furthermore, we will show that chambers in $\Delta_{{\bf s}_{0}}^{-}$
also parametrize an atlas of torus charts that can be viewed as cluster
torus charts of a cluster variety isomorphic to $\mathcal{A}_{{\rm prin},{\bf s}_{0}}$
in the proper sense. Therefore $\tilde{\mathcal{A}}_{{\rm prin},{\bf s}_{0}}^{{\rm scat}}$
is isomorphic to a positive space $\tilde{\mathcal{A}}_{{\rm prin,{\bf s}_{0}}}$
gluing two copies of $\mathcal{A}_{{\rm prin},{\bf s}_{0}}$ together
using birational maps $\mathfrak{p}_{\gamma,\mathfrak{D}_{{\bf s}_{0}}}$\footnote{From now on, we will simply denote $\mathfrak{D}_{{\bf s}_{0}}^{\mathcal{A}_{{\rm prin}}}$
by $\mathfrak{D}_{{\bf s}_{0}}$ and hope that this abuse of notation
will not confuse readers since there are not other scattering diagrams
we refer to.} that are not cluster, where $\gamma$'s are paths connecting chambers
in $\Delta_{{\bf s}_{0}}^{+}$ and those in $\Delta_{{\bf s}_{0}}^{-}$.

\subsection[Construction of $\tilde{\mathcal{A}}_{{\rm prin},{\bf s}_{0}}^{{\rm scat}}$]{Construction of $\boldsymbol{\tilde{\mathcal{A}}_{{\rm prin},{\bf s}_{0}}^{{\rm scat}}}$}

For each chamber in $\sigma$ in $\Delta_{{\bf s}_{0}}^{+}\bigcup\Delta_{{\bf s}_{0}}^{-}$,
we attach a copy of algebraic torus $T_{N^{\circ}\oplus M,\sigma}$.
Given two chambers $\sigma$, $\sigma'$ in $\Delta_{{\bf s}_{0}}^{+}\bigcup\Delta_{{\bf s}_{0}}^{-}$,
let~$\gamma$ be a path from $\sigma'$ to $\sigma$. The path
product $\mathfrak{p}_{\gamma,\mathfrak{D}_{{\bf s}_{0}}}$ is independent
of our choice of the path $\gamma$ since $\mathfrak{D}_{{\bf s}_{0}}$
is consistent. If both $\sigma$, $\sigma'$ are contained in~$\Delta_{{\bf s}_{0}}^{+}$
or $\Delta_{{\bf s}_{0}}^{-}$, then $\mathfrak{p}_{\gamma,\mathfrak{D}_{{\bf s}_{0}}}$
acts as an automorphism of the field of fractions $\mathbb{C}(M^{\circ}\oplus N)$
and therefore induces a birational map $\mathfrak{p}_{\sigma,\sigma'}\colon T_{N^{\circ}\oplus M,\sigma}\dashrightarrow T_{N^{\circ}\oplus M,\sigma'}$.
Suppose that~$\sigma'$ is contained in~$\Delta_{{\bf s}_{0}}^{+}$
and~$\sigma$ in~$\Delta_{{\bf s}_{0}}^{-}$. Notice that~$\gamma$
can be written as composition of three paths $\gamma_{1}\circ\gamma_{2}\circ\gamma_{3}$
where~$\gamma_{1}$ is a~path connecting~$\sigma'$ to $\mathcal{C}_{{\bf s}_{0}}^{-}$, $\gamma_{2}$ a path connecting $\mathcal{C}_{{\bf s}_{0}}^{-}$ to~$\mathcal{C}_{{\bf s}_{0}}^{+}$ and~$\gamma_{3}$ a~path connecting~$\mathcal{C}_{{\bf s}_{0}}^{+}$ to~$\sigma$. Since all $\mathfrak{p}_{\gamma_{i},\mathfrak{D}_{{\bf s}_{0}}}$
act as automorphisms of the field of fractions $\mathbb{C}(M^{\circ}\oplus N)$,
we conclude that $\mathfrak{p}_{\gamma,\mathfrak{D}_{{\bf s}_{0}}}$
induces a birational map $\mathfrak{p}_{\sigma,\sigma'}\colon T_{N^{\circ}\oplus M,\sigma}\dashrightarrow T_{N^{\circ}\oplus M,\sigma'}$. Similar argument applies to the case where $\sigma'$ is contained
in $\Delta_{{\bf s}_{0}}^{-}$ and $\sigma$ in $\Delta_{{\bf s}_{0}}^{+}$.
Thus we can glue all tori $T_{N^{\circ}\oplus M,\sigma}$ together
using the birational maps induced by path products and get a space
$\tilde{\mathcal{A}}_{{\rm prin},{\bf s}_{0}}^{{\rm scat}}$. If we
only use chambers in $\Delta_{{\bf s}_{0}}^{+}$, we get the space
$\mathcal{A}_{{\rm prin},{\bf s}_{0}}^{{\rm scat}}$.

Let $\mathcal{A}_{{\rm prin},{\bf s}_{0}}=\bigcup_{v\in\mathfrak{T}_{{\bf s}_{0}}}T_{N^{\circ}\oplus M,{\bf s}_{v}}$
be the cluster variety of type $\mathcal{A}_{{\rm prin}}$ we build
using the initial seed data of ${\bf s}_{0}$, $\mathcal{A}_{{\rm prin},{\bf s}_{0}}^{\vee}=\bigcup_{v\in\mathfrak{T}_{{\bf s}_{0}}}T_{M^{\circ}\oplus N,{\bf s}_{v}}$
be its Fock--Goncharov dual. Given vertices~$v$,~$v'$ in~$\mathfrak{T}_{{\bf s}_{0}}$,
let $\mu_{v,v'}^{\vee}\colon T_{M^{\circ}\oplus N,{\bf s}_{v}}\dashrightarrow T_{M^{\circ}\oplus N,{\bf s}_{v'}}$ be the birational map induced by
\[
T_{M^{\circ}\oplus N,{\bf s}_{v}}\hookrightarrow\mathcal{A}_{{\rm prin,{\bf s}_{0}}}^{\vee}\hookleftarrow T_{M^{\circ}\oplus N,{\bf s}_{v'}}.
\]
Let $v_{0}$ be the root of $\mathfrak{T}_{{\bf s}_{0}}$. Then ${\bf s}_{v_{0}}={\bf s}_{0}$.
Let $\big(\mu_{v,v_{0}}^{\vee}\big)^{T}\colon M^{\circ}\oplus N\rightarrow M^{\circ}\oplus N$
be the tropicalization of $\mu_{v,v_{0}}^{\vee}$. It extends to a~piecewise linear map
\[
\big(\mu_{v,v_{0}}^{\vee}\big)^{T}\colon \ M_{\mathbb{R}}^{\circ}\oplus N_{\mathbb{R}}\rightarrow M_{\mathbb{R}}^{\circ}\oplus N_{\mathbb{R}}.
\]
Moreover, for any chamber $\sigma$ in $\Delta_{{\bf s}_{v}}^{+}\bigcup\Delta_{{\bf s}_{v}}^{-}$, the restriction of $\big(\mu_{v,v_{0}}^{\vee}\big)^{T}$ to $\sigma$ is
an linear isomorphism onto the corresponding chamber $\sigma_{0}:=\big(\mu_{v,v_{0}}^{\vee}\big)^{T}(\sigma)$
in $\Delta_{{\bf s}_{0}}^{+}\bigcup\Delta_{{\bf s}_{0}}^{-}$. In
the following theorem, we prove that up to isomorphisms, $\tilde{\mathcal{A}}^{\rm scat}_{{\rm prin},{\bf s}_{0}}$ is independent of the choice of the initial seed ${\bf s}_{0}$.
\begin{Theorem}\label{thm:4.1}Given any vertex $v$ in $\mathfrak{T}_{{\bf s}_{0}}$,
consider the Fock--Goncharov tropicalization $\big(\mu^{\vee}_{v,v_{0}}\big)^{T} \!\colon\!$ $M^{\circ}\oplus N\rightarrow M^{\circ}\oplus N$
of $\mu^{\vee}_{v,v_{0}}\colon T_{M^{\circ}\oplus N,{\bf s}_{v}}\dashrightarrow T_{M^{\circ}\oplus N,{\bf s}_{0}}$.
Let $T_{v,\sigma_{0}}\colon T_{N^{\circ}\oplus M,\sigma_{0}}\stackrel{\sim}{\longrightarrow}T_{N^{\circ}\oplus M,\sigma}$
be the isomorphism induced by the linear map $\big(\mu_{v,v_{0}}^{\vee}\big)^{T}|_{\sigma}$.
These isomorphisms glue together to give an isomorphism of positive
spaces $\tilde{\mathcal{A}}_{{\rm prin},{\bf s}_{0}}^{{\rm scat}}\simeq\tilde{\mathcal{A}}_{{\rm prin},{\bf s}_{v}}^{{\rm scat}}$.
\end{Theorem}

\begin{proof}It suffices to prove for the case where $v$ is adjacent to $v_{0}$
via an edge labelled with~$k$ in~$I_{{\rm uf}}$, that is, ${\bf s}_{v}=\mu_{k}({\bf s}_{0})$.
Given $\sigma_{0},\sigma_{0}'$ in $\Delta_{{\bf s}_{0}}^{+}\bigcup\Delta_{{\bf s}_{0}}^{-}$,
let $\sigma=\mu_{v,v_{0}}^{T}(\sigma_{0})$ and $\sigma'=\mu_{v,v_{0}}^{T}(\sigma_{0}')$.
The lemma amounts to the commutativity of the following diagram:
\begin{gather}\begin{split}&
\xymatrix{T_{N^{\circ}\oplus M,\sigma_{0}}\ar[r]^{T_{v,\sigma_{0}}}\ar@{-->}[d]^{\mathfrak{p}_{\sigma_{0},\sigma_{0}'}} & T_{N^{\circ}\oplus M,\sigma}\ar@{-->}[d]^{\mathfrak{p}_{\sigma,\sigma'}}\\
T_{N^{\circ}\oplus M,\sigma_{0}'}\ar[r]^{T_{v,\sigma_{0}'}} & T_{N^{\circ}\oplus M,\sigma'}.
}\end{split}\label{eq:-5}
\end{gather}
The case where both $\sigma_{0}$ and $\sigma_{0}'$ are contained
in $\Delta_{{\bf s}_{0}}^{+}$ is proved in \cite[Proposition~4.3]{GHKK}. We prove for the remaining cases.
\begin{enumerate}\itemsep=0pt
\item[(i)] Both $\sigma_{0}$ and $\sigma_{0}'$ are contained in $\Delta_{{\bf s}_{0}}^{-}$
and on the same side of $e_{k}^{\perp}$. The commutativity of~\eqref{eq:-5}
follows from the mutation invariance of the scattering diagram (cf.\ \cite[Theorem~1.24]{GHKK}).
\item[(ii)] Both $\sigma_{0}$ and $\sigma_{0}'$ are contained in $\Delta_{{\bf s}_{0}}^{-}$
and on different sides of~$e_{k}^{\perp}$. The proof is the same
as the case where both~$\sigma_{0}$ and~$\sigma_{0}'$ are contained
in $\Delta_{{\bf s}_{0}}^{+}$ and on different sides of $e_{k}^{\perp}$.
\item[(iii)] $\{\sigma_{0},\sigma_{0}'\}=\{\mathcal{C}_{{\bf s}_{0}}^{+},\mathcal{C}_{{\bf s}_{0}}^{-}\}$.
Without loss of generality, assume that $\sigma_{0}=\mathcal{C}_{{\bf s}_{0}}^{+}$
and $\sigma_{0}'=\mathcal{C}_{{\bf s}_{0}}^{-}$. Given $(m_{0},n_{0})$
in $\mathcal{C}_{{\bf s}_{0}}^{-}$, $\mathfrak{p}_{\sigma_{0},\sigma_{0}'}^{*}\big(z^{(m_{0},n_{0})}\big)$
is the expansion of theta function $\vartheta_{(m_{0},n_{0})}$ at
a~generic point $Q_{0}$ in $\sigma_{0}$. Let $(m,n)$ in $\sigma'$
be the point such that $(m_{0},n_{0})=\big(\mu_{v,v_{0}}^{\vee}\big)^{T}(m,n)$.
Then $\mathfrak{p}_{\sigma,\sigma'}^{*}\big(z^{(m,n)}\big)$ is the expansion
of theta function $\vartheta_{(m,n)}$ at a generic point $Q$ in
$\sigma$. By mutation invariance of theta functions (cf.\ \cite[Proposition~3.6]{GHKK}),
\begin{gather*}
T_{v,\sigma_{0}}^{*}(\vartheta_{(m,n),Q}) =\vartheta_{(m_{0}n_{0}),T_{v,\sigma_{0}}^{*}(Q)}.
\end{gather*}
Since $\big(\mu_{v,v_{0}}^{\vee}\big)^{T}(\sigma)=\sigma_{0}$,
$T_{v,\sigma_{0}}^{*}(Q)$ is a generic point contained in $\sigma_{0}$.
Hence $\vartheta_{(m_{0},n_{0}),T_{v,\sigma_{0}}^{*}(Q)}\allowbreak =\vartheta_{(m_{0},n_{0}),Q_{0}}$.
Hence
\begin{align*}
\mathfrak{p}_{\sigma_{0},\sigma_{0}'}^{*}\circ T_{v,\sigma_{0}'}^{*}\big(z^{(m,n)}\big) & =\mathfrak{p}_{\sigma_{0},\sigma_{0}'}^{*}\big(z^{(m_{0},n_{0})}\big) =\vartheta_{(m_{0},n_{0}),Q_{0}}
 =\vartheta_{(m_{0},n_{0}),T_{v,\sigma_{0}}^{*}(Q)}\\
 & =T_{v,\sigma_{0}}^{*}(\vartheta_{(m,n),Q}) =T_{v,\sigma_{0}}^{*}\circ\mathfrak{p}_{\sigma,\sigma'}^{*}\big(z^{(m,n)}\big).
\end{align*}

\item[(iv)] $\sigma_{0}$ is contained in $\Delta_{{\bf s}_{0}}^{+}$ and
$\sigma_{0}'$ in $\Delta_{{\bf s}_{0}}^{-}$ or vice versa. Without
loss of generality, assume that~$\sigma_{0}$ is contained in $\Delta_{{\bf s}_{0}}^{+}$
and~$\sigma_{0}'$ in $\Delta_{{\bf s}_{0}}^{-}$. We could factor
$\mathfrak{p}_{\sigma_{0},\sigma_{0}'}$ as
\[
\mathfrak{p}_{\mathcal{C}_{{\bf s}_{0}}^{-},\sigma_{0}'}\circ\mathfrak{p}_{\mathcal{C}_{{\bf s}_{0}}^{+},\mathcal{C}_{{\bf s}_{0}}^{-}}\circ\mathfrak{p}_{\sigma_{0},\mathcal{C}_{{\bf s}_{0}}^{+}}
\]
 and reduce to previous cases.\hfill\qed
\end{enumerate}\renewcommand{\qed}{}
\end{proof}

\subsection[Construction of $\tilde{\mathcal{A}}_{{\rm prin},{\bf s}_{0}}$]{Construction of $\boldsymbol{\tilde{\mathcal{A}}_{{\rm prin},{\bf s}_{0}}}$}

Given each vertex $v$ in $\mathfrak{T}_{{\bf s}_{0}}$, attach two
copies of algebraic torus $T_{N^{\circ}\oplus M,{\bf s}_{v}}^{+}$
and $T_{N^{\circ}\oplus M,{\bf s}_{v}}^{-}$. Let $\mathcal{A}_{{\rm prin},{\bf s}_{0}}^{+}=\bigcup_{v\in\mathfrak{T}_{{\bf s}_{0}}}T_{N^{\circ}\oplus M,{\bf s}_{v}}^{+}$
be the usual cluster variety of type $\mathcal{A}_{{\rm prin}}$ we
build using the initial seed ${\bf s}_{0}$, $\mathcal{A}_{{\rm prin},{\bf s}_{0}}^{\vee}=\bigcup_{v\in\mathfrak{T}_{{\bf s}_{0}}}T_{M^{\circ}\oplus N,{\bf s}_{v}}$
be its Fock--Goncharov dual. Given vertices $v$, $v'$ in $\mathfrak{T}_{{\bf s}_{0}}$,
let $\mu_{v,v'}\colon T_{N^{\circ}\oplus M,{\bf s}_{v}}^{+}\dashrightarrow T_{N^{\circ}\oplus M,{\bf s}_{v'}}^{+}$ be the birational map induced by
\[
T_{N^{\circ}\oplus M,{\bf s}_{v}}^{+}\hookrightarrow\mathcal{A}_{{\rm prin,{\bf s}_{0}}}\hookleftarrow T_{N^{\circ}\oplus M,{\bf s}_{v'}}^{+}
\]
and $\mu_{v,v'}^{\vee}\colon T_{M^{\circ}\oplus N,{\bf s}_{v}}\dashrightarrow T_{M^{\circ}\oplus N,{\bf s}_{v'}}$
that induced by
\[
T_{M^{\circ}\oplus N,{\bf s}_{v}}\hookrightarrow\mathcal{A}_{{\rm prin,{\bf s}_{0}}}^{\vee}\hookleftarrow T_{M^{\circ}\oplus N,{\bf s}_{v'}}.
\]
Let $v_{0}$ be the root of $\mathfrak{T}_{{\bf s}_{0}}$. Then ${\bf s}_{v_{0}}={\bf s}_{0}$.
Let $\big(\mu_{v_{0,}v}^{\vee}\big)^{T}\colon M^{\circ}\oplus N\rightarrow M^{\circ}\oplus N$
be the tropicalization of $\mu_{v_{0},v}^{\vee}$. It extends to a
piecewise linear map:
\[
\big(\mu_{v_{0,}v}^{\vee}\big)^{T}\colon \ M_{\mathbb{R}}^{\circ}\oplus N_{\mathbb{R}}\rightarrow M_{\mathbb{R}}^{\circ}\oplus N_{\mathbb{R}}.
\]
Moreover, its restriction to the chamber $\mathcal{C}_{v}^{\pm}$
is linear and the inverse image of $\mathcal{C}_{{\bf s}_{v}}^{\pm}$
under $\big(\mu_{v_{0,}v}^{\vee}\big)^{T}$ is $\mathcal{C}_{v}^{\pm}$ (cf.\
\cite[Construction~1.30]{GHKK}). Let
\[
\big(\mu_{v_{0},v}^{\vee}\big)^{T}|_{\mathcal{C}_{v}^{\pm}}\colon \ M^{\circ}\oplus N\rightarrow M^{\circ}\oplus N
\]
be this linear map and $\varphi_{v_{0},v}^{\pm}\colon T_{N^{\circ}\oplus M,{\bf s}_{v}}^{\pm}\rightarrow T_{N^{\circ}\oplus M,\mathcal{C}_{v}^{\pm}}$ the isomorphism such that $\big(\varphi_{v_{0},v}^{\pm}\big)^{*}=\big(\mu_{v_{0},v}^{\vee}\big)^{T}|_{\mathcal{C}_{v}^{\pm}}$.
It is proved in \cite[Theorem~4.4]{GHKK} that the isomorphisms
$\varphi_{v_{0},v}^{+}\colon T_{N^{\circ}\oplus M,{\bf s}_{v}}^{+}\rightarrow T_{N^{\circ}\oplus M,\mathcal{C}_{v}^{+}}$ induce a global isomorphism $\mathcal{A}_{{\rm prin},{\bf s}_{0}}\simeq\mathcal{A}_{{\rm prin},{\bf s}_{0}}^{{\rm scat}}$, that is, for any $v$, $v'$ in $\mathfrak{T}_{{\bf s}_{0}}$, the following diagram commutes:
\[
\xymatrix{T_{N^{\circ}\oplus M,{\bf s}_{v}}^{+}\subset\mathcal{A}_{{\rm prin},{\bf s}_{0}}\ar[r]^{\varphi_{v_{0},v}^{+}}\ar@{-->}[d]^{\mu_{{\bf s}_{v},{\bf s}_{v'}}} & T_{N^{\circ}\oplus M,\mathcal{C}_{v}^{+}}\subset\mathcal{A}_{{\rm prin},{\bf s}_{0}}^{{\rm scat}}\ar@{-->}[d]^{\mathfrak{p}_{\mathcal{C}_{v}^{+},\mathcal{C}_{v'}^{+}}}\\
T_{N^{\circ}\oplus M,{\bf s}_{v'}}^{+}\subset\mathcal{A}_{{\rm prin,{\bf s}_{0}}}\ar[r]^{\varphi_{v_{0},v'}^{+}} & T_{N^{\circ}\oplus M,\mathcal{C}_{v'}^{+}}\subset\mathcal{A}_{{\rm prin},{\bf s}_{0}}^{{\rm scat}}.
}
\]
Let $\mathcal{A}_{{\rm prin},{\bf s}_{0}}^{-}=\bigcup_{v\in\mathfrak{T}_{{\bf s}_{0}}}T_{N^{\circ}\oplus M,{\bf s}_{v}}^{-}$ be another copy of $\mathcal{A}_{{\rm prin}}$. Let $\Sigma\colon T_{N^{\circ}\oplus M}\rightarrow T_{N^{\circ}\oplus M}$ be the inversion map such that $\Sigma^{*}(z^{(m,n)})=z^{(-m,-n)}$.
Denote by $\mathcal{A}_{{\rm prin},{\bf s}_{0}}^{{\rm scat},-}$ the subspace
\[
\bigcup_{\mathcal{C}_{v}^{-}\in\Delta_{{\bf s}_{0}}^{-}}T_{N^{\circ}\oplus M,\mathcal{C}_{v}^{-}}
\]
inside $\tilde{\mathcal{A}}_{{\rm prin},{\bf s}_{0}}^{{\rm scat}}$.
\begin{Theorem}\label{thm:4.2} The compositions $\big\{ \varphi_{v_{0},v'}^{-}\circ\Sigma\,|\, v'\in\mathfrak{T}_{{\bf s}_{0}}\big\} $ glue to a global inclusion of positive spaces
\[
\varphi_{v_{0}}^{-}\colon \ \mathcal{A}_{{\rm prin},{\bf s}_{0}}^{-}\hookrightarrow\tilde{\mathcal{A}}_{{\rm prin},{\bf s}_{0}}^{{\rm scat}}
\]
such that the image of $\varphi_{v_{0}}^{-}$ in $\tilde{\mathcal{A}}_{{\rm prin},{\bf s}_{0}}^{{\rm scat}}$
is $\mathcal{A}_{{\rm prin},{\bf s}_{0}}^{{\rm scat},-}$. Furthermore,
given any $v$ in $\mathfrak{T}_{{\bf s}_{0}}$, the following diagram
commutes:
\begin{gather}\begin{split}&
\xymatrix{\mathcal{A}_{{\rm prin,{\bf s}_{0}}}^{-}\ar@{^{(}->}[r]^{\varphi_{v_{0}}^{-}}\ar[d] & \tilde{\mathcal{A}}_{{\rm prin},{\bf s}_{0}}^{{\rm scat}}\ar[d]\\
\mathcal{A}_{{\rm prin},{\bf s}_{v}}^{-}\ar@{^{(}->}[r]^{\varphi_{v}^{-}} & \tilde{\mathcal{A}}_{{\rm prin},{\bf s}_{v}}^{{\rm scat}}.
}\end{split}
\label{eq:-6}
\end{gather}
Here, the left column is the isomorphism $\mathcal{A}_{{\rm prin,{\bf s}_{0}}}^{-}\stackrel{\sim}{\rightarrow}\mathcal{A}_{{\rm prin},{\bf s}_{v}}^{-}$,
the right column is the isomorphism in Theorem{\rm ~\ref{thm:4.1}}.
\end{Theorem}

\begin{proof}We prove the second part of the theorem first. The commutativity of
diagram~\eqref{eq:-6} is equivalent to that of the following diagram
for any given vertex~$v'$ in~$\mathfrak{T}_{{\bf s}_{0}}$:
\begin{gather}\begin{split}&
\xymatrix{T_{N^{\circ}\oplus M,v'}\subset\mathcal{A}_{{\rm prin},{\bf s}_{0}}^{-}\ar[r]\ar@2{-}[d] & T_{N^{\circ}\oplus M,\mathcal{C}_{v'}^{-}}\subset\tilde{\mathcal{A}}_{{\rm prin},{\bf s}_{0}}^{{\rm scat}}\ar[d]\\
T_{N^{\circ}\oplus M,v'}\subset\mathcal{A}_{{\rm prin},{\bf s}_{v}}^{-}\ar[r] & T_{N^{\circ}\oplus M,\mathcal{C}_{v'}^{-}}\subset\tilde{\mathcal{A}}_{{\rm prin},{\bf s}_{v}}^{{\rm scat}}.
}\end{split} \label{eq:-7}
\end{gather}
The left vertical map in the diagram~\eqref{eq:-7} is given by Theorem~\ref{thm:4.1}. In terms of character lattices, it is equivalent to the following diagram:
\[
\xymatrix{M^{\circ}\oplus N\ar@2{-}[d] & M^{\circ}\oplus N\ar[l]_{\Sigma^{*}\circ\left(\mu_{v_{0},v'}^{\vee}\right)^{T}|_{\mathcal{C}_{v'\in\mathfrak{T}_{{\bf s}_{0}}}^{-}}}\\
M^{\circ}\oplus N & M^{\circ}\oplus N.\ar[l]^{\Sigma^{*}\circ\left(\mu_{v,v'}^{\vee}\right)^{T}|_{\mathcal{C}_{v'\in\mathfrak{T}_{{\bf s}_{v}}}^{-}}}\ar[u]_{\left(\mu_{v,v_{0}}^{\vee}\right)^{T}|_{\mathcal{C}_{v'\in\mathfrak{T}_{{\bf s}_{v}}}^{-}}}
}
\]
The commutativity of the above diagram is clear. Hence the commutativity
of diagram~\eqref{eq:-7} follows.

The first part of the theorem is equivalent to the following statement:
Given vertices~$v$,~$v'$ in~$\mathfrak{T}_{{\bf s}_{0}}$, the following
diagram commutes:
\begin{gather}\begin{split}&
\xymatrix{T_{N^{\circ}\oplus M,{\bf s}_{v}}^{-}\subset\mathcal{A}_{{\rm prin},{\bf s}_{0}}^{-}\ar[r]^{\varphi_{v_{0},v}^{-}\circ\Sigma}\ar@{-->}[d]^{\mu_{{\bf s}_{v},{\bf s}_{v'}}} & T_{N^{\circ}\oplus M,\mathcal{C}_{v}^{-}}\subset\tilde{\mathcal{A}}_{{\rm prin},{\bf s}_{0}}^{{\rm scat}}\ar@{-->}[d]^{\mathfrak{p}_{\mathcal{C}_{v}^{-},\mathcal{C}_{v'}^{-}}}\\
T_{N^{\circ}\oplus M,{\bf s}_{v'}}^{-}\subset\mathcal{A}_{{\rm prin,{\bf s}_{0}}}^{-}\ar[r]^{\varphi_{v_{0,}v'}^{-}\circ\Sigma} & T_{N^{\circ}\oplus M,\mathcal{C}_{v'}^{-}}\subset\tilde{\mathcal{A}}_{{\rm prin},{\bf s}_{0}}^{{\rm scat}}.
}\end{split}
\label{eq:-4}
\end{gather}
It suffices to prove for the case where there is an oriented edge
labelled by $k$ from $v$ to $v'$. Moreover, we show that it
suffices to prove for the case where $v=v_{0}$, i.e., ${\bf s}_{0}={\bf s}_{v}$.
Indeed, let $v''$ be the vertex such that there is an oriented
edge labelled by~$k'$ from~$v'$ to~$v''$. Consider the
following cube of commutative diagrams whose front vertical face is
the commutative diagram~\eqref{eq:-4} and back vertical face the commutative
diagram analogous to \eqref{eq:-4} for $v'$ and $v''$. The top
and bottom faces are instances of~\eqref{eq:-7}. The commutativity
of the left vertical face follows from inclusions of tori into $\mathcal{A}_{{\rm prin,{\bf s}_{0}}}$
and $\mathcal{A}_{{\rm prin},{\bf s}_{v}}$. The commutativity of
the right vertical face follows from Theorem~\ref{thm:4.1}. Hence
the back vertical face also commutes.

At last, let us prove the commutativity of diagram \eqref{eq:-4} for the case where $v=v_{0}$ and there is an oriented edge labelled by~$k$ from~$v$ to~$v'$. Indeed,
\begin{align*}
\Sigma^{*}\circ\mathfrak{p}_{\mathcal{C}_{v}^{-},\mathcal{C}_{v'}^{-}}^{*}(z^{(m,n)}) & =\Sigma^{*}\big(z^{(m,n)}(1+z^{(\{e_{k},\cdot\},e_{k})})^{\langle d_{k}e_{k},m\rangle}\big)\\
 & =z^{(-m,-n)}(1+z^{-(\{e_{k},\cdot\},e_{k})})^{\langle d_{k}e_{k},m\rangle},
\end{align*}
and
\begin{align*}
\mu_{{\bf s}_{v},{\bf s}_{v'}}^{*}\circ\Sigma^{*}\circ(\varphi_{v_{0,}v'}^{-})^{*}(z^{(m,n)}) & =\mu_{{\bf s}_{v},{\bf s}_{v'}}^{*}\circ\Sigma^{*}\big[z^{(m,n)+(\{e_{k},\cdot\},e_{k})\langle d_{k}e_{k},m\rangle}\big]\\
 & =\mu_{{\bf s}_{v},{\bf s}_{v'}}^{*}\big(z^{(-m,-n)-(\{e_{k},\cdot\},e_{k})\langle d_{k}e_{k},m\rangle}\big)\\
 & =z^{(-m,-n)-(\{e_{k},\cdot\},e_{k})\langle d_{k}e_{k},m\rangle} \big(1+z^{(\{e_{k},\cdot\},e_{k})}\big)^{-\langle d_{k}e_{k},-m\rangle}\\
 & =z^{(-m,-n)}\big(1+z^{-(\{e_{k},\cdot\},e_{k})}\big)^{\langle d_{k}e_{k},m\rangle}\\
 & =\Sigma^{*}\circ\mathfrak{p}_{\mathcal{C}_{v}^{-},\mathcal{C}_{v'}^{-}}^{*}(z^{(m,n)}).
\end{align*}
Hence we obtain the desired commutativity of diagram~\eqref{eq:-4}.
\end{proof}

Combine Theorem \ref{thm:4.2} together with Theorem 4.4 of \cite{GHKK},
we obtain the following corollary:
\begin{Corollary}
\label{cor:4.3}We have the isomorphism of positive spaces:
\[
\tilde{\mathcal{A}}_{{\rm prin,{\bf s}_{0}}}=\mathcal{A}_{{\rm prin{\bf s}_{0}}}^{+}\bigcup\mathcal{A}_{{\rm prin,{\bf s}_{0}}}^{-}\stackrel{\sim}{\longrightarrow}\tilde{\mathcal{A}}_{{\rm prin},{\bf s}_{0}}^{{\rm scat}}.
\]
Here the restriction of the isomorphism $\tilde{\mathcal{A}}_{{\rm prin},{\bf s}_{0}}\stackrel{\sim}{\rightarrow}\tilde{\mathcal{A}}_{{\rm prin},{\bf s}_{0}}^{{\rm scat}}$
to each of $\mathcal{A}_{{\rm prin},{\bf s}_{0}}^{\pm}$ is the inclusion
$\mathcal{A}_{{\rm prin},{\bf s}_{0}}^{\pm}\hookrightarrow\tilde{\mathcal{A}}_{{\rm prin},{\bf s}_{0}}^{{\rm scat}}$.
\end{Corollary}

\subsection[The case where $\Theta\big(\mathcal{A}_{{\rm prin},{\bf s}_{0}}\big)=\mathcal{A}_{{\rm prin},{\bf s}_{0}}^{\vee}\big(\mathbb{Z}^{T}\big)$]{The case where $\boldsymbol{\Theta\big(\mathcal{A}_{{\rm prin},{\bf s}_{0}}\big)=\mathcal{A}_{{\rm prin},{\bf s}_{0}}^{\vee}\big(\mathbb{Z}^{T}\big)}$}

Now suppose we are in the case where $\Theta\big(\mathcal{A}_{{\rm prin},{\bf s}_{0}}\big)=\mathcal{A}_{{\rm prin},{\bf s}_{0}}^{\vee}\big(\mathbb{Z}^{T}\big)$, then it follows that $\mathfrak{p_{+,-}^{{\bf s}_{0}}}$ is rational. So we can apply the previous construction and build the space $\tilde{\mathcal{A}}_{{\rm prin},{\bf s}_{0}}$. Moreover, by \cite[Theorem 7.5]{GHKK}, the vector space
\[
{\rm can}\big(\mathcal{A}_{{\rm prin},{\bf s}_{0}}\big)=\bigoplus_{q\in\mathcal{A}_{{\rm prin,{\bf s}_{0}}}^{\vee}}\mathbb{C}\cdot\vartheta_{q}
\]
 has a $\mathbb{C}[N]$-algebra structure and we have canonical inclusion
\[
{\rm can}\big(\mathcal{A}_{{\rm prin},{\bf s}_{0}}\big) \subset{\rm up}\big(\mathcal{A}_{{\rm prin},{\bf s}_{0}}\big),
\]
where ${\rm up}\big(\mathcal{A}_{{\rm prin},{\bf s}_{0}}\big)$ is the ring
of regular functions of $\mathcal{A}_{{\rm prin},{\bf s}_{0}}$.
\begin{Proposition}\label{prop:4.4}Given any $(m,n)$ in $\mathcal{A}_{{\rm prin},{\bf s}_{0}}^{\vee}\big(\mathbb{Z}^{T}\big)$, if for some generic point $Q$ in $\Delta_{{\bf s}_{0}}^{-}$, there
are finitely many broken lines with initial direction $(m,n)$ and
endpoints $Q$, then the same is true for any generic point $Q'$
in $\Delta_{{\bf s}_{0}}^{-}$. In particular, $\vartheta_{(m,n),Q'}$
is a positive Laurent polynomial.
\end{Proposition}

\begin{proof} This proposition follows from the arguments in Proposition~7.1 of~\cite{GHKK} with signs reversed.
\end{proof}
\begin{Definition} Let $\Theta^{-}\subset\mathcal{A}_{{\rm prin},{\bf s}_{0}}^{\vee}\big(\mathbb{Z}^{T}\big)$ be the collection of $(m,n)$ such that for some (or equivalently
for any by Proposition~\ref{prop:4.4}) generic point~$Q$ in $\Delta_{{\bf s}_{0}}^{-}$,
there are finitely many broken lines with initial direction $(m,n)$
and end point~$Q$.
\end{Definition}

\begin{Theorem}\label{thm:4.6}If $\mathcal{C}_{{\bf s}_{0}}^{+}$ is contained in~$\Theta^{-}$, then we have the canonical inclusion
\[
{\rm can}\big(\mathcal{A}_{{\rm prin},{\bf s}_{0}}\big)\subset{\rm up}\big(\tilde{\mathcal{A}}_{{\rm prin},{\bf s}_{0}}\big),
\]
where ${\rm up}\big(\tilde{\mathcal{A}}_{{\rm prin},{\bf s}_{0}}\big)$ is the ring of regular function of $\tilde{\mathcal{A}}_{{\rm prin},{\bf s}_{0}}$.
\end{Theorem}

\begin{proof}\looseness=-1 Since we already have the canonical inclusion ${\rm can}\big(\mathcal{A}_{{\rm prin},{\bf s}_{0}}\big)\subset{\rm up}\big(\mathcal{A}_{{\rm prin},{\bf s}_{0}}\big)$,
it suffices to show that for any $(m,n)$ in $\mathcal{A}_{{\rm prin},{\bf s}_{0}}^{\vee}\big(\mathbb{Z}^{T}\big)$,
for any generic point $Q$ in $\Delta_{{\bf s}_{0}}^{-}$, $\vartheta_{(m,n),Q}$
is a Laurent polynomial. In particular, it suffices to show that $\Theta^{-}=\mathcal{A}_{{\rm prin},{\bf s}_{0}}^{\vee}\big(\mathbb{Z}^{T}\big)$.
First we show that $\mathcal{C}_{{\bf s}_{0}}^{-}$ is contained in
$\Theta^{-}$. Indeed, using the same arguments in \cite[Proposition~3.8
and Corollary~3.9]{GHKK} with signs reversed, we can show
that given a chamber~$\sigma$ in $\Delta_{{\bf s}_{0}}^{-}$ and
a generic point in~$Q$ in ${\rm int}(\sigma)$, for initial direction
$(m,n)$ with $(m,n)$ in $\sigma\bigcap(M^{\circ}\oplus N)$, there
is only one broken line with end point~$Q$. Next, we show that $\Theta^{-}$
is closed under addition. Indeed, if~$p$,~$q$ are in~$\Theta^{-}$,
consider the multiplication $\vartheta_{p}\cdot\vartheta_{q}$. Let
$\alpha(p,q,p+q)$ be the coefficient of $\vartheta_{p+q}$ in the
expansion by theta functions of $\vartheta_{p}\cdot\vartheta_{q}$.
The pair of straight lines with initial directions $p$ and $q$ respectively
and end point very close to $p+q$ will contribute to $\alpha(p,q,p+q)$,
so we have $\alpha(p,q,p+q)\geq1$. Write
\[
\vartheta_{p}\cdot\vartheta_{q}=\sum_{r}\alpha(p,q,r)\vartheta_{r}.
\]
 The right hand side is a finite sum with non-negative coefficients.
Given a generic point~$Q$ in~$\Delta_{{\bf s}_{0}}^{-}$, we have
\[
\vartheta_{p,Q}\cdot\vartheta_{q,Q}=\sum_{r}\alpha(p,q,r)\vartheta_{r,Q}.
\]
 Since each $\vartheta_{r,Q}$ is a sum of Laurent series with positive
coefficients and each of $\vartheta_{q,Q}$ and $\vartheta_{p,Q}$
is a positive Laurent polynomial, we conclude that each $\vartheta_{r,Q}$
is also a positive Laurent polynomial. Hence for each $r$ such that
$\alpha(p,q,r)\neq0$, $r$ is contained in $\Theta^{-}$. In particular,
$p+q$ is contained in $\Theta^{-}$. Since we have both $\mathcal{C}_{{\bf s}_{0}}^{+}$
and $\mathcal{C}_{{\bf s}_{0}}^{-}$ contained in $\Theta^{-}$ and
$\Theta^{-}$ is closed under addition, we conclude that $\Theta^{-}=\mathcal{A}_{{\rm prin},{\bf s}_{0}}^{\vee}\big(\mathbb{Z}^{T}\big)$.
\end{proof}

\section[The full Fock--Goncharov conjecture for $\mathcal{A}_{{\rm prin,\mathbb{T}_{1}^{2}}}$,
$\mathcal{X}_{\mathbb{T}_{1}^{2}}$ and $\mathcal{A}_{\mathbb{T}_{1}^{2}}$]{The full Fock--Goncharov conjecture for $\boldsymbol{\mathcal{A}_{{\rm prin,\mathbb{T}_{1}^{2}}}}$,
$\boldsymbol{\mathcal{X}_{\mathbb{T}_{1}^{2}}}$ and $\boldsymbol{\mathcal{A}_{\mathbb{T}_{1}^{2}}}$}\label{sec:Full-Fock--Goncharov-conjecture}

\subsection[The full Fock--Goncharov conjecture for $\mathcal{A}_{{\rm prin}\mathbb{T}_{1}^{2}}$
and $\mathcal{X}_{\mathbb{T}_{1}^{2}}$]{The full Fock--Goncharov conjecture for $\boldsymbol{\mathcal{A}_{{\rm prin}\mathbb{T}_{1}^{2}}}$
and $\boldsymbol{\mathcal{X}_{\mathbb{T}_{1}^{2}}}$}

In this section, we will study in detail cluster varieties associated
to once punctured torus $\mathbb{T}_{1}^{2}$. To apply techniques
we have developed so far, we will construct the seed associated to
the ideal triangulation of $\mathbb{T}_{1}^{2}$ from that associated
to the ideal triangulation of the four-punctured sphere~$\mathbb{S}_{4}^{2}$.
Consider the following ideal triangulation of~$\mathbb{S}_{4}^{2}$:
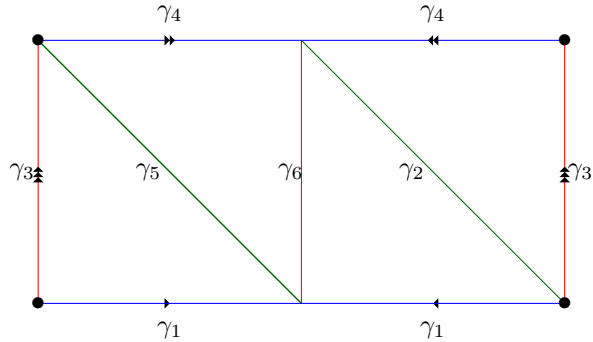
\begin{figure}[h!]\centering \begin{tikzpicture}[scale=0.7] \draw [red](0,0) -- (0,5); \draw [red](5,0) -- (5,5); \draw [red](-5,0) -- (-5,5); \fill (-5.1,2.3) -- (-4.9,2.3) -- (-5,2.4);
\fill (-5.1,2.4) -- (-4.9,2.4) -- (-5,2.5); \fill (-5.1,2.5) -- (-4.9,2.5) -- (-5,2.6); \draw (-5.3,2.5) node {\small$\gamma_{3}$}; \fill (5.1,2.3) -- (4.9,2.3) -- (5,2.4); \fill (5.1,2.4) -- (4.9,2.4) -- (5,2.5); \fill (5.1,2.5) -- (4.9,2.5) -- (5,2.6); \draw (5.3,2.5) node {\small$\gamma_{3}$}; \draw (-2.9,2.5) node {\small$\gamma_{5}$}; \draw (2.1,2.5) node {\small$\gamma_{2}$}; \draw (-0.2,2.5) node {\small$\gamma_{6}$};
\draw [blue](-5,5) -- (5,5); \fill (-2.6,-0.1) -- (-2.6,0.1) -- (-2.5,0); \fill (2.6,-0.1) -- (2.6,0.1) -- (2.5,0); \draw (-2.5,-0.5) node {\small$\gamma_{1}$}; \draw (2.5,-0.5) node {\small$\gamma_{1}$};
\fill (-2.6,4.9) -- (-2.6,5.1) -- (-2.5,5); \fill (-2.5,4.9) -- (-2.5,5.1) -- (-2.4,5); \fill (2.6,4.9) -- (2.6,5.1) -- (2.5,5); \fill (2.5,4.9) -- (2.5,5.1) -- (2.4,5); \draw (-2.5,5.5) node {\small$\gamma_{4}$}; \draw (2.5,5.5) node {\small$\gamma_{4}$};
\draw [blue](-5,0) -- (5,0);
\draw [darkgreen] (0,0)--(-5,5); \draw [darkgreen] (0,0)--(-5,5); \draw [darkgreen] (0,5)--(5,0);
\node at (-5,5) {\textbullet}; \node at (5,5) {\textbullet}; \node at (5,-0) {\textbullet}; \node at (-5,0) {\textbullet};
\end{tikzpicture}
\caption{An ideal triangulation of the 4-punctured sphere $\mathbb{S}_{4}^{2}$.}\label{Ideal Trian} \end{figure}

This ideal triangulation yields the following exchange matrix $\epsilon$:
\[
\begin{bmatrix}0 & 1 & -1 & 0 & 1 & -1\\
-1 & 0 & 1 & -1 & 0 & 1\\
1 & -1 & 0 & 1 & -1 & 0\\
0 & 1 & -1 & 0 & 1 & -1\\
-1 & 0 & 1 & -1 & 0 & 1\\
1 & -1 & 0 & 1 & -1 & 0
\end{bmatrix}.
\]
Let $N$ be a lattice of rank $6$ with a basis $(e_{i})_{i\in I}$, $I=\{1,\dots,6\}$. The exchange matrix $\epsilon$ equips~$N$ with the following skew-symmetric $\mathbb{Z}$-bilinear form
\begin{align*}
\{\,,\,\}\colon \ N\times N & \rightarrow\mathbb{Z},\\
\{e_{i},e_{j}\} & =\epsilon_{ij}.
\end{align*}
 Let ${\bf s}=(e_{i})_{i\in I}$ be our initial seed. Let $(f_{i})_{i\in I}$
be the basis for the dual lattice $M$ with $f_{i}=e_{i}^{*}$.

Let $\Pi=\left\langle \sigma_{1}\right\rangle \times\left\langle \sigma_{2}\right\rangle \times\left\langle \sigma_{3}\right\rangle $
be the subgroup of $S_{6}$ where $\sigma_{i}$ is the involution
that interchange the indices $i$ and $i+3$ for each $i=1,2,3$.
The group $\Pi$ acts on the index set $I$ for~${\bf s}$ via permutation
of indices. Denote by $\overline{i}$ the orbit of $i$ under the
action of~$\Pi$. Applying the quotient construction in Section~\ref{sec:Apply-quotient-constructions},
we obtained the following data:
\begin{itemize}\itemsep=0pt
\item An index set $\overline{I}=\{\overline{i}\}_{i=1,2,3}$.
\item A lattice $\overline{N}$ of rank $3$ with a basis $(e_{\overline{i}})_{\overline{i}\in\overline{I}}$
whose dual lattice $\overline{M}$ has a basis $(e_{\overline{i}}^{*})_{\overline{i}\in\overline{I}}$.
\item A skew-symmetric bilinear form $\{\,,\,\}\colon \overline{N}\times\overline{N}\rightarrow\mathbb{Z}$
on $\overline{N}$ such that $\{e_{\overline{i}},e_{\overline{j}}\}=\{e_{i},e_{j}\}$.
\item For each $\overline{i}\in\overline{I}$, a positive integer $d_{\overline{i}}=2$.
\item A sublattice $\overline{N}^{\circ}$ of $\overline{N}$ with a basis
$(d_{\overline{i}}e_{\overline{i}})_{\overline{i}\in\overline{I}}$
whose dual lattice $\overline{M}^{\circ}$ has a basis $(f_{\overline{i}})_{\overline{i}\in\overline{I}}$
where $f_{\overline{i}}=\frac{1}{d_{\overline{i}}}e_{\overline{i}}^{*}$.
\end{itemize}
The new seed $\overline{{\bf s}}$ is the labelled basis $(e_{\overline{i}})_{\overline{i}\in\overline{I}}$.
The exchange matrix $\overline{\epsilon}$ for $\overline{{\bf s}}$ is as follows:
\[
\begin{bmatrix}0 & 2 & -2\\
-2 & 0 & 2\\
2 & -2 & 0
\end{bmatrix}.
\]
It is the exchange matrix arising from the following ideal triangulation of $\mathbb{T}_{1}^{2}$:

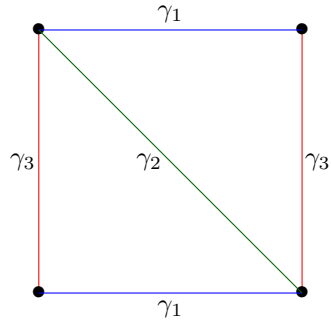
\begin{center} \begin{tikzpicture}[scale=0.7] \draw [red](0,0) -- (0,5); \draw [red](5,0) -- (5,5);

\draw (5.3,2.5) node {\small$\gamma_{3}$};
\draw (2.1,2.5) node {\small$\gamma_{2}$}; \draw (-0.3,2.5) node {\small$\gamma_{3}$};
\node at (0,5) {\textbullet}; \node at (5,5) {\textbullet}; \node at (5,0) {\textbullet}; \node at (0,0) {\textbullet}; \draw [blue](0,5) -- (5,5);

\draw (2.5,-0.3) node {\small$\gamma_{1}$};
\draw (2.5,5.3) node {\small$\gamma_{1}$};
\draw [blue](0,0) -- (5,0);
\draw [darkgreen] (0,5)--(5,0);
\end{tikzpicture}
 \end{center}

The quotient map $q\colon N_{\mathbb{R}}\twoheadrightarrow\overline{N}_{\mathbb{R}}$
is defined as follows:
\begin{align*}
q\colon \ N_{\mathbb{R}} & \twoheadrightarrow\overline{N}_{\mathbb{R}},\\
e_{i} & \mapsto e_{\overline{i}},\\
e_{i+3} & \mapsto e_{\overline{i}}.
\end{align*}
The map $q$ has the following natural section
\begin{align*}
s\colon \ \overline{N}_{\mathbb{R}} & \rightarrow N_{\mathbb{R}},\\
e_{\overline{i}} & \mapsto\frac{1}{d_{\overline{i}}}(e_{i}+e_{i+3}).
\end{align*}
Its restriction $s\colon \overline{N}^{\circ}\rightarrow N$, $d_{\overline{i}}e_{\overline{i}}\mapsto e_{i}+e_{i+3}$
induces a map $s^{*}$ on the dual lattices
\begin{align*}
s^{*}\colon \ M & \rightarrow\overline{M}^{\circ},\\
f_{i} & \mapsto f_{\overline{i}},\\
f_{i+3} & \mapsto f_{\overline{i}},\qquad i=1,2,3.
\end{align*}
Our goal is to show that the full Fock--Goncharov conjecture holds
for $\mathcal{A}_{{\rm prin},\mathbb{T}_{1}^{2}}$, that is,
\[
\Theta(\mathcal{A}_{{\rm prin},\mathbb{T}_{1}^{2}})=\mathcal{A}_{{\rm prin},\mathbb{T}_{1}^{2}}^{\vee}\big(\mathbb{Z}^{T}\big)
\qquad \text{and} \qquad
{\rm can}(\mathcal{A}_{{\rm prin},\mathbb{T}_{1}^{2}})=\Gamma(\mathcal{A}_{{\rm prin},\mathbb{T}_{1}^{2}},\mathcal{O}_{\mathcal{A}_{{\rm prin},\mathbb{T}_{1}^{2}}}),
\]
where ${\rm can}\big(\mathcal{A}_{{\rm prin},\mathbb{T}_{1}^{2}}\big)$ is
the $\mathbb{C}$-algebra with a vector space basis parametrized by
$\Theta\big(\mathcal{A}_{{\rm prin},\mathbb{T}_{1}^{2}}\big)$.

Our first step is to show that $\Theta(\mathcal{A}_{{\rm prin},\mathbb{T}_{1}^{2}})=\mathcal{A}_{{\rm prin},\mathbb{T}_{1}^{2}}^{\vee}\big(\mathbb{Z}^{T}\big)$.
Let $Q$ be a generic point in the interior of $\mathcal{C}_{\bar{{\bf s}}}^{+}$. It is known that the path-product $\mathfrak{p}_{-,+}^{\overline{{\bf s}}}$
is a composition of infinitely many wall-crossing automorphisms, so
it is hard to do the computation directly. Instead, we apply the commutative
diagram~\eqref{eq:-9} to our case and obtain the following equality
\[
\mathfrak{p}_{-,+}^{\overline{{\bf s}}}\big(\big(z^{(-f_{\overline{i}},0)}\big)\big)=\phi_{(s^{*},q)}^{*}\circ\mathfrak{p}_{-,+}^{{\bf s}}\big(z^{(-f_{i},0)}\big).
\]
Let $T_{N\oplus M}$ be the algebraic torus whose character lattice
is $M\oplus N$. Recall the inversion automorphism of $T_{N\oplus M}$
\begin{align*}
\Sigma\colon \ T_{N\oplus M} & \rightarrow T_{N\oplus M},\\
z^{(m,n)} & \mapsto z^{(-m,-n)}.
\end{align*}
Then $\mathfrak{p}_{-,+}^{{\bf s}}\circ\Sigma^{*}$ is the\emph{ Donaldson--Thomas
transformation }of $\mathcal{A}_{{\rm prin},\mathbb{S}_{4}^{2}}$
with respect to the seed~${\bf s}$, as defined in \cite{GS16} (see
also \cite{Ke13,KS08,KS13,Na10}). By Theorem 1.3 of \cite{GS16},
DT-transformation of $\mathcal{A}_{{\rm prin},\mathbb{S}_{4}^{2}}$
can be written as a composition of finitely many cluster mutations.
In particular, $\mathfrak{p}_{-,+}^{{\bf s}}\circ\Sigma^{*}(z^{(f_{i},0)})=\mathfrak{p}_{-,+}^{{\bf s}}(z^{(-f_{i},0)})$
is a positive Laurent polynomial. Hence after passing through the
quotient map $\phi_{\left(s^{*},q\right)}^{*}$, we conclude that
$\mathfrak{p}_{-,+}^{\overline{{\bf s}}}\big(\big(z^{(-f_{\overline{i}},0)}\big)\big)$
is also a positive Laurent polynomial. Therefore, $(-f_{\overline{i}},0)$
is contained in $\text{\ensuremath{\Theta}}(\mathcal{A}_{{\rm prin},\mathbb{T}_{1}^{2}})$
for each $\overline{i}$. Since $\text{\ensuremath{\Theta}}\big(\mathcal{A}_{{\rm prin},\mathbb{T}_{1}^{2}}\big)$
is saturated, closed under addition and translation by $\overline{N}$
after the identification $\mathcal{A}_{{\rm prin},\mathbb{T}_{1}^{2}}^{\vee}\big(\mathbb{Z}^{T}\big)\simeq\overline{M}^{\circ}\oplus\overline{N}$,
we conclude that $\mathcal{C}_{\overline{{\bf s}}}^{-}\big(\mathbb{Z}^{T}\big)$
is contained in $\Theta\big(\mathcal{A}_{{\rm prin},\mathbb{T}_{1}^{2}}\big)$.
Since $\mathcal{C}_{\overline{{\bf s}}}^{+}\big(\mathbb{Z}^{T}\big)$ is also
contained in $\Theta\big(\mathcal{A}_{{\rm prin},\mathbb{T}_{1}^{2}}^{{\bf }}\big)$
and $\Theta\big(\mathcal{A}_{{\rm prin},\mathbb{T}_{1}^{2}}\big)$ is closed
under addition we conclude that $\mathcal{A}_{{\rm prin},\mathbb{T}_{1}^{2}}^{\vee}\big(\mathbb{Z}^{T}\big)=\Theta\big(\mathcal{A}_{{\rm prin},\mathbb{T}_{1}^{2}}\big)$.

By \cite[Theorem 7.5]{GHKK}, ${\rm mid}(\mathcal{A}_{{\rm prin}})$
is always contained in $\Gamma\big(\mathcal{A}_{{\rm prin},}\mathcal{O}_{\mathcal{A}_{{\rm prin}}}\big)$.
Since $\mathcal{A}_{{\rm prin},\mathbb{T}_{1}^{2}}^{\vee}\big(\mathbb{Z}^{T}\big)=\Theta\big(\mathcal{A}_{{\rm prin},\mathbb{T}_{1}^{2}}\big)$
as we just showed, in this case, we have ${\rm can}\big(\mathcal{A}_{{\rm prin},\mathbb{T}_{1}^{2}}\big)$
equal to ${\rm mid}\big(\mathcal{A}_{{\rm prin},\mathbb{T}_{1}^{2}}\big)$
and therefore contained in $\Gamma\big(\mathcal{A}_{{\rm prin},\mathbb{T}_{1}^{2}},\mathcal{O}_{\mathcal{A}_{{\rm prin},\mathbb{T}_{1}^{2}}}\big)$.
To prove that ${\rm can}\big(\mathcal{A}_{{\rm prin},\mathbb{T}_{1}^{2}}\big)$
is equal to \linebreak $\Gamma\big(\mathcal{A}_{{\rm prin},\mathbb{T}_{1}^{2}},\mathcal{O}_{\mathcal{A}_{{\rm prin},\mathbb{T}_{1}^{2}}}\big)$,
we need to do some more explicit computations. For our choice of initial
seed~${\bf s}$, the following mutation sequence, to be read from
left to right, is a maximal green sequence\footnote{Daping Wen showed this explicit maximal green sequence to me.}
\[
[1,3,2,4,6,5,1,6,4,3,2,5].
\]
The DT-transformation $\mathfrak{p}_{-,+}^{{\bf s}}\circ\Sigma^{*}$
can be computed via compositions of cluster mutations along the above
sequence of directions. For $i=1,2,3$,
\begin{align*}
\mathfrak{p}_{-,+}^{\overline{{\bf s}}}\big(\big(z^{(-f_{\overline{i}},0)}\big)\big) & =\phi_{(s^{*},q)}^{*}\circ\mathfrak{p}_{-,+}^{{\bf s}}\circ\Sigma^{*}\big(z^{(f_{i},0)}\big)\\
 & =\phi_{(s^{*},q)}^{*}\big[z^{(-f_{i,}0)}\cdot F\big(z^{(\{e_{1},\cdot\},e_{1})},z^{(\{e_{2},\cdot\},e_{2})}, \dots,z^{(\{e_{6},\cdot\},e_{6})}\big)\big]\\
 & =z^{(-f_{\overline{i}},0)}\cdot\overline{F}\big(z^{(\{e_{\overline{i}},\cdot\}, e_{\overline{i}})},z^{(\{e_{\overline{i+1}},\cdot\},e_{\overline{i+1}})}, z^{(\{e_{\overline{i+2}},\cdot\},e_{\overline{i+2}})}\big).
\end{align*}
Here,
\begin{gather*}
F(x_{1},x_{2},\dots,x_{6}) =1+x_{1}+x_{1}x_{3}+x_{1}x_{6}+x_{1}x_{3}x_{6}+x_{1}x_{3}x_{5}x_{6}\\
\hphantom{F(x_{1},x_{2},\dots,x_{6}) =}{} +x_{1}x_{2}x_{3}x_{6}+x_{1}x_{2}x_{3}x_{5}x_{6}+x_{1}^{2}x_{2}x_{3}x_{5}x_{6}
\end{gather*}
and
\[
\overline{F}(w_{1},w_{2},w_{3})=1+w_{1}+2w_{1}w_{3}+w_{1}w_{3}^{2}+2w_{1}w_{2}w_{3}^{2}+w_{1}w_{2}^{2}w_{3}^{2}+w_{1}^{2}w_{2}^{2}w_{3}^{2}.
\]
With an argument very similar to the case of the Kronecker $2$ quiver,
we can directly enumerate broken lines with initial direction $(f_{\overline{i}}-f_{\overline{i+1}},0)$
and ending at a generic point~$Q$ in $\mathcal{C}_{\bar{{\bf s}}}^{+}$.
The theta function $\vartheta_{(f_{\overline{i}}-f_{\overline{i+1}},0)}$
has expansion at~$Q$ as follows:
\[
\vartheta_{(f_{\overline{i}}-f_{\overline{i+1}},0),Q} =z^{(f_{\overline{i}}-f_{\overline{i+1}},0)} \big(1+z^{(\{e_{\overline{i+1}},\cdot\},e_{\overline{i+1}})} +z^{(\{e_{\overline{i}}+e_{\overline{i+1}},\cdot\},e_{\overline{i}}+e_{\overline{i+1}})}\big).
\]
Indeed, given a broken line $\gamma$ with initial direction $(f_{\overline{i}}-f_{\overline{i+1}},0)$,
first observe that in order for~$\gamma$ to end at~$Q$, the first bending of $\gamma$ can only happen at the wall
\[
\mathfrak{d}_{1}=\big(\mathbb{R}_{\geq0}f_{\overline{i}} +\mathbb{R}_{\geq0}f_{\overline{i+2}}+\overline{N}_{\mathbb{R}}, 1+z^{(\{e_{\overline{i+1}},\cdot\},e_{\overline{i+1}} )}\big)
\]
 or the wall
\[
\mathfrak{d}_{2}=\big(\mathbb{R}_{\geq0}\big({-}f_{\overline{i}}+2f_{\overline{i+2}}\big) +\mathbb{R}_{\geq0}f_{\overline{i+2}}+\overline{N}_{\mathbb{R}}, 1+z^{(\{e_{\overline{i+1}},\cdot\},e_{\overline{i+1}})}\big).
\]
Both $\mathfrak{d}_{1}$ and $\mathfrak{d}_{2}$ are contained in
the hyperplane $e_{\overline{i+1}}^{\perp}$. If $\gamma$ has first
bending anywhere else, it will shoot back out and never reach $\mathcal{C}_{\overline{{\bf s}}}^{+}$.
Next, observe that by alternatively mutating in the direction $\overline{i}$
and $\overline{i+1}$, we obtain a copy $\mathfrak{D}$ of the scattering
diagram of the Kronecker $2$ quiver inside $\mathfrak{D}_{\overline{{\bf s}}}$
that contain $\mathfrak{d}_{1}$ and $\mathfrak{d}_{2}$. Thus, the
enumeration of all broken lines with initial direction $\big(\big(f_{\overline{i}}-f_{\overline{i+1}}\big),0\big)$
and end point $Q$ happens inside $\mathfrak{D}\subset\mathfrak{D}_{\overline{{\bf s}}}$
and can be reduced to the case of Example~3.10 in~\cite{GHKK}.

\begin{figure}\centering
 \begin{tikzpicture}[scale=2.2] \draw (0,0)--(1,0); \draw (0,1) -- (0,0);
\draw (1,0) -- (2,0); \draw [blue] (2,0) -- (0.5,0.866);
\draw (2,0) -- (3,0); \draw [blue] (3,0) -- (0.5,0.866);
\draw (3,0) -- (4,0); \draw [blue] (4,0) -- (0.5,0.866);
\draw [dotted] (4,0) -- (4.3,0); \draw [blue] (0,1.732) -- (1,0); \draw [blue] (0,0) -- (1,1.732);
\draw (0,0) -- (0,1.732);
\draw (0.3,0.1) node {\tiny$(0,1,0)$}; \draw (0.3,0.9) node {\tiny$(0,0,1)$}; \draw (0.3,1.65) node {\tiny$(-1,0,2)$}; \draw (1.2,1.65) node {\tiny$(0,-1,2)$}; \draw (2.2,1.65) node {\tiny$(1,-2,2)$}; \draw (3.2,1.65) node {\tiny$(2,-3,2)$}; \draw (1.2,0.1) node {\tiny$(1,0,0)$}; \draw (2.2,0.1) node {\tiny$(2,-1,0)$}; \draw (3.2,0.1) node {\tiny$(3,-2,0)$};
\draw (0,1.732) --(1,1.732); \draw (1,1.732) -- (2,1.732); \draw [blue] (2,1.732)-- (0.5,0.866); \draw (2,1.732) -- (3,1.732); \draw [blue] (3,1.732)-- (0.5,0.866); \draw (3,1.732) -- (4,1.732); \draw [blue] (4,1.732)-- (0.5,0.866); \draw [dotted] (4,1.732) -- (4.3,1.732);
\draw (0,1.732) -- (1.5,2.598) -- (1,1.732);
\draw (1,0) -- (1.5, -0.866) -- (2,0); \draw (0,0) -- (-0.5,-0.866) -- (-1,0) -- (0,0); \draw (1,0) -- (-0.5,-0.866); \draw (1,0) -- (1,-1.732) -- (1.5, -0.866);
\draw (-1,0) -- (-1.5, 0.866) -- (0,0);
\end{tikzpicture} \caption{The projection of $\mathfrak{D}_{\overline{\mathbf{s}}}$ on to the plane $x+y+z=1$ in $\overline{M}_{\mathbb{R}}$. If we only mutate in the direction of $\overline{1}$ and $\overline{2}$, we obtain a copy of the scattering diagram of the Kronecker 2 quiver as shaded by blue.}
 \end{figure}
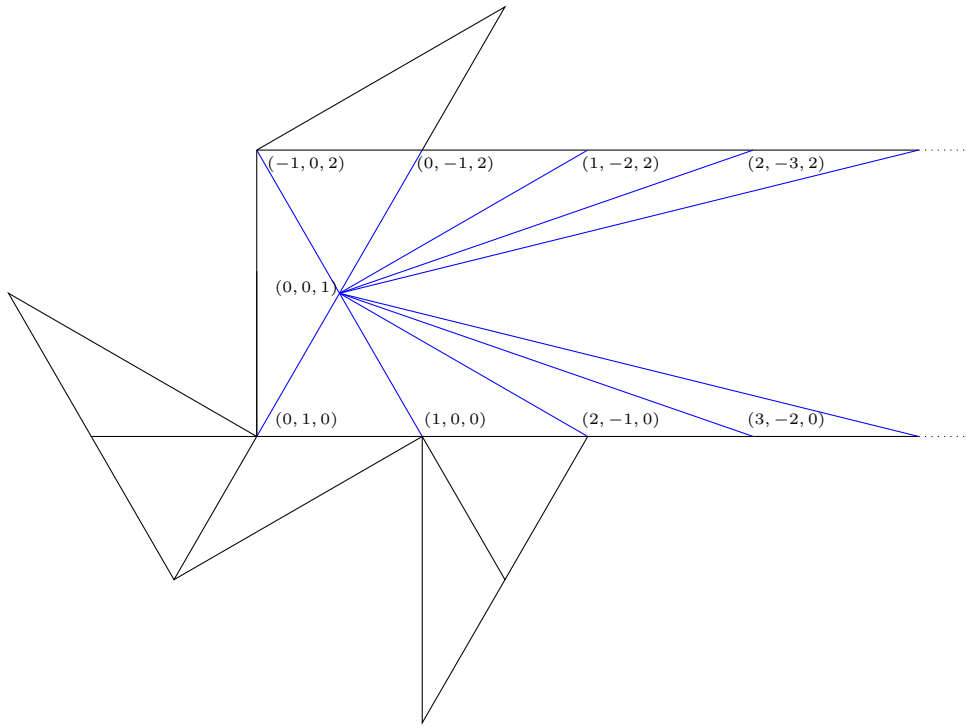

For any $Q$ in $\mathfrak{D}_{\overline{{\bf s}}}\setminus{\rm supp}(\mathfrak{D}_{\overline{{\bf s}}})$,\footnote{Again, same as in Section~\ref{sec:Building--using}, we will simply denote $\mathfrak{D}_{\overline{{\bf s}}}^{\mathcal{A}_{{\rm prin}}}$
by $\mathfrak{D}_{\overline{{\bf s}}}$. For justification of this
abuse of notation, see the discussion at the end of Section~\ref{sec:Apply-quotient-constructions}.} for any $\overline{n}$ in $\overline{N}$, $\vartheta_{(0,\overline{n}),Q}=z^{(0,\overline{n})}$.
Hence we will replace $\vartheta_{(0,\overline{n})}$ with $z^{(0,\overline{n})}$
for any~$\overline{n}$ in~$\overline{N}$. The upper cluster algebra
$\Gamma\big(\mathcal{A}_{{\rm prin},\mathbb{T}_{1}^{2}},\mathcal{O}_{\mathcal{A}_{{\rm prin},\mathbb{T}_{1}^{2}}}\big)$ is computed explicitly in~\cite{MM13} with a given set of generators.
By changing variables in \cite[Proposition~7.1.1]{MM13}\footnote{Because of different conventions for $\mathcal{X}$-cluster variables
in \cite{FZ07} and \cite{GHKK}, the $B$-matrix used in \cite{MM13}
is the transpose of the exchange matrix we have used.} -- replacing their $x_{i}$ with our $z^{(f_{\overline{i}},0)}$ and their $f_{i}$ with our $z^{(0,e_{\overline{i}})}$, we can easily
see that the generators of $\Gamma\big(\mathcal{A}_{{\rm prin},\mathbb{T}_{1}^{2}},\mathcal{O}_{\mathcal{A}_{{\rm prin},\mathbb{T}_{1}^{2}}}\big)$
given in \cite{MM13} are all theta functions, that is, $\big\{\vartheta_{\pm(f_{\overline{i}},0)},z^{\pm(0,e_{\overline{i}})}, \vartheta_{(f_{\overline{i}}-f_{\overline{i+1}},0)}\big\}_{i=1,2,3}$
is a set of generators for $\Gamma\big(\mathcal{A}_{{\rm prin},\mathbb{T}_{1}^{2}},\mathcal{O}_{\mathcal{A}_{{\rm prin},\mathbb{T}_{1}^{2}}}\big)$.
Thus we conclude that
\[
{\rm can}\big(\mathcal{A}_{{\rm prin},\mathbb{T}_{1}^{2}}\big)=\Gamma\big(\mathcal{A}_{{\rm prin},\mathbb{T}_{1}^{2}},\mathcal{O}_{\mathcal{A}_{{\rm prin},\mathbb{T}_{1}^{2}}}\big).
\]
Hence we prove the first main theorem of this section:
\begin{Theorem}\label{thm:5.1}The full Fock--Goncharov conjecture holds for $\mathcal{A}_{{\rm prin},\mathbb{T}_{1}^{2}}$, that is,
\[
\Theta\big(\mathcal{A}_{{\rm prin},\mathbb{T}_{1}^{2}}\big)=\mathcal{A}_{{\rm prin},\mathbb{T}_{1}^{2}}^{\vee}\big(\mathbb{Z}^{T}\big),\qquad{\rm can}\big(\mathcal{A}_{{\rm prin},\mathbb{T}_{1}^{2}}\big)=\Gamma\big(\mathcal{A}_{{\rm prin},\mathbb{T}_{1}^{2}},\mathcal{O}_{\mathcal{A}_{{\rm prin},\mathbb{T}_{1}^{2}}}\big).
\]
\end{Theorem}

The corollary below follows immediately:
\begin{Corollary}\label{cor:5.2}$\Theta\big(\mathcal{X}_{\mathbb{T}_{1}^{2}}\big) =\mathcal{X}_{\mathbb{T}_{1}^{2}}^{\vee}\big(\mathbb{Z}^{T}\big)$
and ${\rm can}\big(\mathcal{X}_{\mathbb{T}_{1}^{2}}\big) =\Gamma\big(\mathcal{X}_{\mathbb{T}_{1}^{2}},\mathcal{O}_{\mathcal{X}_{\mathbb{T}_{1}^{2}}}\big)$.
\end{Corollary}

\subsection{A cluster variety with two non-equivalent cluster structures}

Given a lattice $L$, let $T_{L}=\mathbb{G}_{m}\otimes L$ be the
algebraic torus whose cocharacter lattice is $L$. Given an initial
seed ${\bf s}_{0}$, let $\mathfrak{T}_{{\bf s}_{0}}$ be the oriented
tree whose root has ${\bf s}_{0}$ attached to it and whose vertices
parametrize seeds mutationally equivalent to ${\bf s}_{0}$. There
is an labelled edge $w\stackrel{k}{\rightarrow}w'$ in $\mathfrak{T}_{{\bf s}_{0}}$
if and only if ${\bf s}_{w}=\mu_{k}({\bf s}_{w'})$. We write the
cluster variety associated to ${\bf s}_{0}$ of type $\mathcal{A}$
or $\mathcal{X}$ or $\mathcal{A}_{{\rm prin}}$ as $V'=\bigcup_{w\in\mathfrak{T}_{{\bf s}_{0}}}T_{L,{\bf s}_{w}}$ where the lattice $L$ depends on the type of $V'$.
\begin{Definition}\label{def:5.3}Given a log Calabi--Yau variety~$V$, if there exists
a cluster variety $V'=\bigcup_{w\in\mathfrak{T}_{{\bf s}_{0}}}T_{L,{\bf s}_{w}}$
of type $\mathcal{A}$, $\mathcal{X}$ or $\mathcal{A}_{{\rm prin}}$ together with a birational morphism
\[
\iota\colon \ V'\rightarrow V
\]
such that for any pair of adjacent seeds ${\bf s}_{w}\stackrel{\mu_{k}}{\rightarrow}{\bf s}_{w'}$
in $\mathfrak{T}_{{\bf s}_{0}}$, the restriction of $\iota$ to $T_{L,{\bf s}_{w}}\bigcup_{\mu_{k}} \allowbreak T_{L,{\bf s}_{w'}}$ is an open embedding into~$V$, then we say that the cluster torus charts $T_{L,{\bf s}_{w}}$ together with their embeddings $\iota_{{\bf s}_{w}}\colon T_{L,{\bf s}_{w}}\hookrightarrow V$ into $V$ provide \emph{an atlas of cluster torus charts }for $V$.
If $\iota$ is an isomorphism outside strata of codimension at least~$2$ in the domain and range, then we say that the atlas of cluster torus charts $\{(T_{L,{\bf s}_{w}},\iota_{{\bf s}_{w}})\}_{w\in\mathfrak{T}_{{\bf s}_{0}}}$ provides a \emph{cluster structure} for~$V$. In this case, the codimension~$2$ condition guarantees that~$V$ and the cluster variety~$V'$ have the same ring of regular functions.
\end{Definition}

\begin{Remark}In the above definition, by calling $V$ a variety, we are assuming that~$V$ is an integral, separated scheme over~$\mathbb{C}$, though possibly not of finite type. As shown in~\cite{GHK15}, $\mathcal{A}$~and~$\mathcal{A}_{{\rm prin}}$ spaces are always separated while
$\mathcal{X}$ spaces are in general not (cf.\ \cite[Theorem~3.14 and Remark~4.2]{GHK15}). In this light, calling the $\mathcal{X}$ spaces
\emph{$\mathcal{X}$ cluster varieties} can be very misleading, though
the terminology has already been accepted in the cluster literature.
In the above definition, if $V'$ is of type $\mathcal{A}$ or
$\mathcal{A}_{{\rm prin}}$, we could just assume that $\iota$ is
an open embedding. If $V'$ is of type $\mathcal{X}$, by requiring
only the restriction of $\iota$ to any pair of adjacent cluster tori
to be an open embedding into $V$, we allow certain loci in $V'$
to get identified when mapped into $V$ and therefore get around the
issue of non-separateness of~$V'$.
\end{Remark}
\begin{Example}Let $V'$ be a cluster variety of type $\mathcal{A}$ or $\mathcal{A}_{{\rm prin}}$. Then by \cite[Theorem~3.14]{GHK15}, the canonical map
\[
\iota\colon \ V'\rightarrow{\rm Spec}({\rm up}(V'))
\]
is an open embedding where ${\rm up}(V')$ is the ring of regular
functions on $V'$. Since $V'$ is a $S_{2}$-scheme, we know
that the complement of the image of $V'$ in ${\rm Spec}({\rm up}(V'))$
has codimension at least $2$. Therefore the atlas of cluster torus
charts of $V'$ provide a cluster structure for~${\rm Spec}({\rm up}(V'))$.
\end{Example}

\begin{Definition}\label{def:5.4}Given a log Calabi--Yau variety $V$ and two atlases
of cluster torus charts
\[
\mathcal{T}_{1}=\{ (T_{L,{\bf s}_{w}},\iota_{{\bf s}_{w}})\} _{w\in\mathfrak{T}_{{\bf s}_{0}}},\qquad\mathcal{T}_{2}=\{(T_{L',{\bf s}_{v}},\iota_{{\bf s}_{v}})\}_{v\in\mathfrak{T}_{{\bf s}_{0}'}}
\]
for $V$, we say $\mathcal{T}_{1}$ and $\mathcal{T}_{2}$ are \emph{non-equivalent atlases} if for any $w$ in $\mathfrak{T}_{{\bf s}_{0}}$, there is no $v$ in $\mathfrak{T}_{{\bf s}_{0}'}$ such that the embeddings $\iota_{{\bf s}_{w}}\colon T_{L,{\bf s}_{w}}\hookrightarrow V$ and $\iota_{{\bf s}_{v}}\colon T_{L,{\bf s}_{v}}\hookrightarrow V$ have the same image in~$V$. Given two non-equivalent atlases of cluster
torus charts for~$V$, we say they provide \emph{non-equivalent cluster
structures} for~$V$ if both atlases cover~$V$ up to codimension~$2$.
\end{Definition}

\begin{Definition}\label{def:5.5}Given a variety $V$ and an atlas of cluster torus
charts $\mathcal{T}=\{(T_{L,{\bf s}_{w}},\iota_{{\bf s}_{w}})\}_{w\in\mathfrak{T}_{{\bf s}_{0}}}$
for $V$, a \emph{global monomial} for $\mathcal{T}$ is a regular function on $V$ that restricts to a character on some torus $T_{L,{\bf s}_{w}}$ in~$\mathcal{T}$.
\end{Definition}

Given our initial seed $\overline{{\bf s}}$ for $\mathcal{A}_{\mathbb{T}_{1}^{2}}$
as constructed in the beginning of this section, let $\mathfrak{T}_{\overline{{\bf s}}}$
be the infinite directed tree whose vertices parametrize seeds mutationally
equivalent to~$\overline{{\bf s}}$. Since~$\mathfrak{p}_{+,-}^{\overline{{\bf s}}}$
is rational as we have shown, we can run the construction in Section~\ref{sec:Building--using} and build a variety $\tilde{\mathcal{A}}_{{\rm prin},\mathbb{T}_{1}^{2}}:=\tilde{\mathcal{A}}_{{\rm prin},\overline{{\bf s}}}$ that is isomorphic to the space $\tilde{\mathcal{A}}_{{\rm prin},\overline{{\bf s}}}^{{\rm scat}}$
we build using all chambers in~$\Delta_{{\bf \overline{{\bf s}}}}^{+}$
and~$\Delta_{\overline{{\bf s}}}^{-}$ in the scattering diagram.
Let $\mathcal{A}_{{\rm prin},\mathbb{T}_{1}^{2}}^{\pm}\subset\tilde{\mathcal{A}}_{{\rm prin},\mathbb{T}_{1}^{2}}$ be the copy of~$\mathcal{A}_{{\rm prin},\mathbb{T}_{1}^{2}}$ corresponding
to $\Delta_{\overline{{\bf s}}}^{\pm}$ respectively.
\begin{Proposition}\label{prop:5.6}We have
\[
{\rm up}\big(\tilde{\mathcal{A}}_{{\rm prin},\mathbb{T}_{1}^{2}}\big)={\rm can}\big(\mathcal{A}_{{\rm prin},\mathbb{T}_{1}^{2}}\big),
\]
 where ${\rm up}\big(\tilde{\mathcal{A}}_{{\rm prin},\mathbb{T}_{1}^{2}}\big)$ is the ring of regular functions of $\tilde{\mathcal{A}}_{{\rm prin},\mathbb{T}_{1}^{2}}$.
\end{Proposition}

\begin{proof}Since
\[
{\rm up}\big(\tilde{\mathcal{A}}_{{\rm prin},\mathbb{T}_{1}^{2}}\big)\subset{\rm up}\big(\mathcal{A}_{{\rm prin},\mathbb{T}_{1}^{2}}\big)={\rm can}\big(\mathcal{A}_{{\rm prin},\mathbb{T}_{1}^{2}}\big),
\]
to prove the equality, we only need to show that
\[
{\rm can}\big(\mathcal{A}_{{\rm prin},\mathbb{T}_{1}^{2}}\big)\subset{\rm up}\big(\tilde{\mathcal{A}}_{{\rm prin},\mathbb{T}_{1}^{2}}\big).
\]
By Theorem~\ref{thm:4.6}, it suffices to show that $\mathcal{C}_{{\bf \overline{s}}}^{+}$ is contained in $\Theta^{-}\big(\mathcal{A}_{{\rm prin},\mathbb{T}_{1}^{2}}\big)$. Indeed, given a generic point~$Q$ in $\mathcal{C}_{\overline{{\bf s}}}^{-}$, by direct computations we obtain that
\begin{gather*}
\vartheta_{(f_{\overline{i}},0),Q} =\mathfrak{p}_{+,-}^{\overline{{\bf s}}}\big(z^{(f_{\overline{i}},0)}\big)
 =z^{(f_{\overline{i}},0)}\overline{G}(z^{(\{e_{\overline{i}},\cdot\}, e_{\overline{i}})},z^{(\{e_{\overline{i+1}},\cdot\},e_{\overline{i+1}})}, z^{(\{e_{\overline{i+2}},\cdot\},e_{\overline{i+2}})}),
\end{gather*}
where
\[
G(w_{1},w_{2},w_{3})=1+w_{1}+2w_{1}w_{2}+w_{1}w_{2}^{2}+2w_{1}w_{2}^{2}w_{3} +w_{1}w_{2}^{2}w_{3}^{2}+w_{1}^{2}w_{2}^{2}w_{3}^{2}.
\]
Since $\Theta^{-}\big(\mathcal{A}_{{\rm prin},\mathbb{T}_{1}^{2}}\big)$ is
closed under addition, as we have shown in the proof of Theorem~\ref{thm:4.6},
and translation by $\overline{N}$, we conclude that $\mathcal{C}_{{\bf \overline{s}}}^{+}$
is contained in $\Theta^{-}\big(\mathcal{A}_{{\rm prin},\mathbb{T}_{1}^{2}}\big)$.
\end{proof}

It follows immediately from Proposition~\ref{prop:5.6} that ${\rm up}\big(\tilde{\mathcal{A}}_{{\rm prin},\mathbb{T}_{1}^{2}}\big)={\rm up}\big(\mathcal{A}_{{\rm prin},\mathbb{T}_{1}^{2}}\big)$.
Hence, each of~$\mathcal{A}_{{\rm prin},\mathbb{T}_{1}^{2}}^{\pm}$
has complement of codimension $2$ inside $\tilde{\mathcal{A}}_{{\rm prin},\mathbb{T}_{1}^{2}}$
and cluster tori for each of $\mathcal{A}_{{\rm prin},\mathbb{T}_{1}^{2}}^{\pm}$
provide an atlas of cluster torus charts for $\tilde{\mathcal{A}}_{{\rm prin},\mathbb{T}_{1}^{2}}$.
We denote these two atlases by $\mathcal{T}^{+}$ and $\mathcal{T}^{-}$
respectively.
\begin{Theorem}\label{thm:5.7}The atlases of cluster torus charts $\mathcal{T}^{+}$
and $\mathcal{T}^{-}$ give $\tilde{\mathcal{A}}_{{\rm prin},\mathbb{T}_{1}^{2}}$
two non-equivalent cluster structures in the sense of Definition~{\rm \ref{def:5.3}},
corresponding to $\Delta_{\overline{{\bf s}}}^{+}$ and $\Delta_{\overline{{\bf s}}}^{-}$
respectively.
\end{Theorem}

\begin{proof}Write
\[
\mathcal{T}^{\pm}=\big\{ T_{\overline{N}^{\circ}\oplus\overline{M},\overline{{\bf s}}_{v}}^{\pm},\iota_{v}^{\pm}\big\} _{v\in\mathfrak{T}_{\overline{{\bf s}}}}.
\]
Given any $v$ in $\mathfrak{T}_{\overline{{\bf s}}}$, assume there
exists $w$ in $\mathfrak{T}_{\overline{{\bf s}}}$ such that the embeddings
\[
\iota_{v}^{+}\colon \ T_{\overline{N}^{\circ}\oplus\overline{M},\overline{{\bf s}}_{v}}^{+}\hookrightarrow\tilde{\mathcal{A}}_{{\rm prin},\mathbb{T}_{1}^{2}}
\qquad \text{and} \qquad
\iota_{w}^{-}\colon \ T_{\overline{N}^{\circ}\oplus\overline{M},\overline{{\bf s}}_{w}}^{-}\hookrightarrow\tilde{\mathcal{A}}_{{\rm prin},\mathbb{T}_{1}^{2}}
\]
 have the same image. Characters on $T_{\overline{N}^{\circ}\oplus\overline{M},\overline{{\bf s}}_{v}}^{+}$
that are global monomials for $\mathcal{T}^{+}$ (cf.\ Definition~\ref{def:5.5})
correspond to theta functions parametrized by $\mathcal{C}_{v}^{+}\big(\mathbb{Z}^{T}\big)$
. Characters of $T_{\overline{N}^{\circ}\oplus\overline{M},\overline{{\bf s}}_{w}}^{-}$
that are regular functions on $\tilde{\mathcal{A}}_{{\rm prin},\mathbb{T}_{1}^{2}}$
correspond to theta functions parametrized by $\mathcal{C}_{w}^{-}\big(\mathbb{Z}^{T}\big)$.
Since $\mathcal{C}_{w}^{-}\big(\mathbb{Z}^{T}\big)\bigcap\mathcal{C}_{v}^{+}\big(\mathbb{Z}^{T}\big)=0\oplus\overline{N}$,
no characters other than $z^{(0,\overline{n})}$ for $\overline{n}\in\overline{N}$
on $T_{\overline{N}^{\circ}\oplus\overline{M},\overline{{\bf s}}_{w}}^{-}$
can be global monomials for $\mathcal{T}^{+}$, which is in contradiction
with that the embeddings $\iota_{v}^{+}$ and $\iota_{w}^{-}$ have
the same image. Hence $\mathcal{T}^{+}$ and $\mathcal{T}^{-}$ give
$\tilde{\mathcal{A}}_{{\rm prin},\mathbb{T}_{1}^{2}}$ two non-equivalent
cluster structures.
\end{proof}

\subsection[The full Fock--Goncharov conjecture for $\mathcal{A}_{\mathbb{T}_{1}^{2}}$]{The full Fock--Goncharov conjecture for $\boldsymbol{\mathcal{A}_{\mathbb{T}_{1}^{2}}}$}

Since $\Theta\big(\mathcal{A}_{{\rm prin},\mathbb{T}_{1}^{2}}\big)=\mathcal{A}_{{\rm prin},\mathbb{T}_{1}^{2}}^{\vee}$, it follows that $\Theta\big(\mathcal{A}_{\mathbb{T}_{1}^{2}}\big)=\mathcal{A}_{\mathbb{T}_{1}^{2}}^{\vee}$.
However, we will see that $\text{{\rm can}}\big(\mathcal{A}_{\mathbb{T}_{1}^{2}}\big)$
is strictly contained in $\Gamma\big(\mathcal{A}_{\mathbb{T}_{1}^{2}},\mathcal{O}_{\mathbb{T}_{1}^{2}}\big)$.
Hence the full Fock--Goncharov conjecture does not hold for $\mathcal{A}_{\mathbb{T}_{1}^{2}}$.
Since
\[
\big\{\vartheta_{(\pm f_{\overline{i}},0)},\vartheta_{(f_{\overline{i}-\overline{i+1}},0)},z^{(0,\pm e_{\overline{i}})}\big\}_{i=1,2,3}
\]
 is a set of generators of ${\rm {\rm can}}\big(\mathcal{A}_{{\rm prin},\mathbb{T}_{1}^{2}}\big)$,
$\big\{\vartheta_{\pm f_{\overline{i}}},\vartheta_{f_{\overline{i}-\overline{i+1}}}\big\}_{i=1,2,3}$
is a set of generators of ${\rm can}\big(\mathcal{A}_{\mathbb{T}_{1}^{2}}\big)$.
Denote $z^{f_{\overline{i}}}$ by $A_{i}$. Then with respect to the
initial seed $\overline{{\bf s}}$, ${\rm can}\big(\mathcal{A}_{\mathbb{T}_{1}^{2}}\big)$
has the following set of generators
\begin{gather*}
\vartheta_{f_{\overline{i}}}=A_{i},\qquad\vartheta_{f_{\overline{i}-\overline{i+1}}}
=\frac{A_{1}^{2}+A_{2}^{2}+A_{3}^{2}}{A_{i+1}A_{i+2}} =A_{i}\left(\frac{A_{1}^{2}+A_{2}^{2}+A_{3}^{2}}{A_{1}A_{2}A_{3}}\right),
\\
\vartheta_{-f_{\overline{i}}}= \frac{\big(A_{1}^{2}+A_{2}^{2}+A_{3}^{2}\big)^{2}}{A_{i}A_{i+1}^{2}A_{i+2}^{2}}= A_{i}\left(\frac{A_{1}^{2}+A_{2}^{2}+A_{3}^{2}}{A_{1}A_{2}A_{3}}\right)^{2},\qquad i=1,2,3\mod3.
\end{gather*}
By \cite[Proposition~6.2.2]{MM13}, $\Gamma\big(\mathcal{A}_{\mathbb{T}_{1}^{2}},\mathcal{O}_{\mathbb{T}_{1}^{2}}\big)$
has the following set of generators:
\[
\left\{ A_{1},A_{2},A_{3},\frac{A_{1}^{2}+A_{2}^{2}+A_{3}^{2}}{A_{1}A_{2}A_{3}}\right\}.
\]
Hence ${\rm can}\big(\mathcal{A}_{\mathbb{T}_{1}^{2}}\big)$ is strictly contained
in $\Gamma\big(\mathcal{A}_{\mathbb{T}_{1}^{2}},\mathcal{O}_{\mathbb{T}_{1}^{2}}\big)$.

We can run the construction in Section~\ref{sec:Building--using}
similarly for $\mathcal{A}_{\mathbb{T}_{1}^{2}}$ and obtain a new
variety \mbox{$\tilde{\mathcal{A}}_{\mathbb{T}_{1}^{2}}\!\supset\!\mathcal{A}_{\mathbb{T}_{1}^{2}}$}.
Moreover, given each oriented edge $v\stackrel{k}{\rightarrow}v'$ in
$\mathfrak{T}_{\overline{{\bf s}}}$, the commutative diagram
\[
\xymatrix{T_{\overline{N}^{\circ}\oplus\overline{M},\overline{{\bf s}}_{v}}^{\pm}\ar@{^{(}->}[dr]\ar@{-->}[rr]^{\mu_{k}} & & T_{\overline{N}^{\circ}\oplus\overline{M},\overline{{\bf s}}_{v'}}^{\pm}\ar@{_{(}->}[dl]\\
 & \widetilde{\mathcal{A}}_{{\rm prin},\mathbb{T}_{1}^{2}}\ar[d]^{\tilde{\pi}}\\
 & T_{\overline{M}}
}
\]
restrict to
\[
\xymatrix{T_{\overline{N}^{\circ},\overline{{\bf s}}_{v}}^{\pm}\ar@{^{(}->}[dr]\ar@{-->}[rr]^{\mu_{k}} & & T_{\overline{N}^{\circ},\overline{{\bf s}}_{v'}}^{\pm}\ar@{_{(}->}[dl]\\
 & \widetilde{\mathcal{A}}_{\mathbb{T}_{1}^{2}}\ar[d]^{\tilde{\pi}}\\
 & e.
}
\]
Given each vertex $v$, we can compute the gluing map $T_{\overline{N}^{\circ},\overline{{\bf s}}_{v}}^{+}\dashrightarrow T_{\overline{N}^{\circ},\overline{{\bf s}}_{v}}^{-}$
very explicitly. For the initial seed $\overline{{\bf s}}$, the gluing
map $\alpha_{\overline{{\bf s}}}\colon T_{\overline{N}^{\circ},\overline{{\bf s}}}^{+}\dashrightarrow T_{\overline{N}^{\circ},\overline{{\bf s}}}^{-}$
is given by
\[
\alpha_{\overline{{\bf s}}}^{*}(A_{i})=A_{i}\left(\frac{A_{1}^{2}+A_{2}^{2}+A_{3}^{2}}{A_{1}A_{2}A_{3}}\right)^{2}.
\]
For any seed $\overline{{\bf s}}_{v}=(e_{\overline{i}}')_{\overline{i}\in\overline{I}}$
for $v$ in $\mathfrak{T}_{\overline{{\bf s}}}$, denote $z^{\frac{e_{\overline{i}}^{'*}}{2}}$
by $A_{i}^{\overline{{\bf s}}_{v}}$. Let~$\eta$ be the rational
function such that $\eta(x_{1},x_{2},x_{3})=\frac{x_{1}^{2}+x_{2}^{2}+x_{3}^{2}}{x_{1}x_{2}x_{3}}$.
We know that $\eta(A_{1},A_{2},A_{3})$ is invariant under cluster
mutations, that is, for any seed $\overline{{\bf s}}'$, we have
\[
\eta\big(A_{1}^{\overline{{\bf s}}_{v}},A_{2}^{\overline{{\bf s}}_{v}},A_{3}^{\overline{{\bf s}}_{v}}\big)=\eta(A_{1},A_{2},A_{3})=\eta.
\]
For each $v$ in $\mathfrak{T}_{\overline{{\bf s}}}$, define the
following birational map $\alpha_{\overline{{\bf s}}_{v}}\colon T_{\overline{N}^{\circ},{\bf \overline{{\bf s}}}}\dashrightarrow T_{\overline{N}^{\circ},\overline{{\bf s}}}$
such that $\alpha_{\overline{{\bf s}}_{v}}^{*}\big(A_{i}^{\overline{{\bf s}}_{v}}\big)=A_{i}^{\overline{{\bf s}}_{v}}\cdot\eta^{2}$.
We show that the birational maps $\alpha_{\overline{{\bf s}}_{v}}$
give the glue maps $T_{\overline{N}^{\circ},\overline{{\bf s}}_{v}}^{+}\dashrightarrow T_{\overline{N}^{\circ},\overline{{\bf s}}_{v}}^{-}$.
It suffices to show that the following diagram commutes for any oriented
edge $v\stackrel{k}{\rightarrow}v'$ in $\mathfrak{T}_{\overline{{\bf s}}}$:
\[
\xymatrix{T_{\overline{N}^{\circ},{\bf \overline{s}}_{v}}\ar@{-->}[r]^{\mu_{k}}\ar@{-->}[d]^{\alpha_{\overline{{\bf s}}_{v}}} & T_{\overline{N}^{\circ},{\bf \overline{s}}_{v'}}\ar@{-->}[d]^{\alpha_{{\bf \overline{s}}_{v'}}}\\
T_{\overline{N}^{\circ},\overline{\mathbf{s}}_{v}}\ar@{-->}[r]^{\mu_{k}} & T_{\overline{N}^{\circ},{\bf \overline{s}}_{v'}}.
}
\]
Indeed, for $i=k$,
\begin{gather*}
\alpha_{{\bf \overline{s}}_{v}}^{*}\circ\mu_{k}^{*}\big(A_{i}^{{\bf \overline{s}}_{v'}}\big) =\alpha_{{\bf \overline{s}}_{v}}^{*}\left[\frac{\big(A_{k+1}^{{\bf \overline{s}}_{v'}}\big)^{2}+\big(A_{k+2}^{{\bf \overline{s}}_{v'}}\big)^{2}}{A_{k}^{{\bf \overline{s}}_{v'}}}\right]
 =\frac{\big(A_{k+1}^{\overline{{\bf s}}_{v}}\cdot\eta^{2}\big)^{2}+\big(A_{k+2}^{{\bf \overline{s}}_{v}}\cdot\eta^{2}\big)^{2}}{A_{k}^{{\bf \overline{s}}_{v}}\cdot\eta^{2}} \\
\hphantom{\alpha_{{\bf \overline{s}}_{v}}^{*}\circ\mu_{k}^{*}\big(A_{i}^{{\bf \overline{s}}_{v'}}\big)}{}
 =\frac{\big(A_{k+1}^{\overline{{\bf s}}_{v}}\big)^{2}+\big(A_{k+2}^{{\bf \overline{s}}_{v}}\big)^{2}}{A_{k}^{{\bf \overline{s}}_{v}}}\cdot\eta^{2}
 =\mu_{k}^{*}\big(A_{k}^{{\bf \overline{s}}_{v'}}\big)\mu_{k}^{*}\big(\eta^{2}\big)
 =\mu_{k}^{*}\big(A_{k}^{{\bf \overline{s}}_{v'}}\cdot\eta^{2}\big)\\
\hphantom{\alpha_{{\bf \overline{s}}_{v}}^{*}\circ\mu_{k}^{*}\big(A_{i}^{{\bf \overline{s}}_{v'}}\big)}{}
 =\mu_{k}^{*}\circ\alpha_{{\bf \overline{s}}_{v'}}^{*}\big(A_{k}^{{\bf \overline{s}}_{v'}}\big).
\end{gather*}
For $i\neq k$,
\begin{gather*}
\alpha_{{\bf \overline{s}}_{v}}^{*}\circ\mu_{k}^{*}\big(A_{i}^{{\bf \overline{s}}_{v'}}\big) =\alpha_{{\bf \overline{s}}_{v}}^{*}\big(A_{i}^{{\bf \overline{s}}_{v}}\big)
 =A_{i}^{{\bf \overline{s}}_{v}}\cdot\eta^{2}
 =\mu_{k}^{*}\big(A_{i}^{{\bf \overline{s}}_{v'}}\big)\cdot\mu_{k}^{*}\big(\eta^{2}\big)
 =\mu_{k}^{*}\circ\alpha_{{\bf \overline{s}}_{v'}}^{*}\big(A_{i}^{{\bf \overline{s}}_{v'}}\big).
\end{gather*}
Hence the diagram commutes.

Having already written down explicitly the gluing maps for $\tilde{\mathcal{A}}_{\mathbb{T}_{1}^{2}}$,
let us show that
\[
\Gamma\big(\tilde{\mathcal{A}}_{\mathbb{T}_{1}^{2}},\mathcal{O}_{\tilde{\mathcal{A}}_{\mathbb{T}_{1}^{2}}}\big) ={\rm can}\big(\mathcal{A}_{\mathbb{T}_{1}^{2}}\big).
\]
Indeed, since we already have the containment ${\rm can}\big(\mathcal{A}_{\mathbb{T}_{1}^{2}}\big)\subset\Gamma \big(\tilde{\mathcal{A}}_{\mathbb{T}_{1}^{2}},\mathcal{O}_{\tilde{\mathcal{A}}_{\mathbb{T}_{1}^{2}}}\big)$,
we only need to show the reverse containment. Consider $\alpha_{\overline{{\bf s}}}^{-1}\colon T_{\overline{N},\overline{{\bf s}}}^{-}\dashrightarrow T_{\overline{N},\overline{{\bf s}}}^{+}$.
Since $\big(\alpha_{\overline{{\bf s}}}^{-1}\big)^{*}(\eta)=\eta^{-1}$,
$\eta=\frac{A_{1}^{2}+A_{2}^{2}+A_{3}^{2}}{A_{1}A_{2}A_{3}}$ is not
contained in $\Gamma\big(\tilde{\mathcal{A}}_{\mathbb{T}_{1}^{2}},\mathcal{O}_{\tilde{\mathcal{A}}_{\mathbb{T}_{1}^{2}}}\big)$.
Recall that ${\rm can}\big(\mathcal{A}_{\mathbb{T}_{1}^{2}}\big)$ has the
following set of generators
\begin{gather}
\big\{A_{i},A_{i}\eta,A_{i}\eta^{2}\big\}_{i=1,2,3}.\label{eq:-13}
\end{gather}
Suppose that $A_{1}^{a_{1}}A_{2}^{a_{2}}A_{3}^{a_{3}}\eta^{b}$ is
a monomial in $\Gamma\big(\tilde{\mathcal{A}}_{\mathbb{T}_{1}^{2}},\mathcal{O}_{\tilde{\mathcal{A}}_{\mathbb{T}_{1}^{2}}}\big)$.
Since
\[
\big(\alpha_{\overline{{\bf s}}}^{-1}\big)^{*} (A_{i} )=A_{i}\eta^{2},\qquad \big(\alpha_{\overline{{\bf s}}}^{-1}\big)^{*} (\eta)=\eta^{-1},
\]
in order for $\big(\alpha_{\overline{{\bf s}}}^{-1}\big)^{*}\big(A_{1}^{a_{1}}A_{2}^{a_{2}}A_{3}^{a_{3}}\eta^{b}\big)$
to be regular on $\tilde{\mathcal{A}}_{\mathbb{T}_{1}^{2}}$, $a_{1}$, $a_{2}$, $a_{3}$
and $b$ must satisfy the following relation
\[
2(a_{1}+a_{2}+a_{3})\geq b\geq0.
\]
It is straightforward to see that $A_{1}^{a_{1}}A_{2}^{a_{2}}A_{3}^{a_{3}}\eta^{b}$
could be written as products of generators in~\eqref{eq:-13}. Therefore
$\Gamma\big(\tilde{\mathcal{A}}_{\mathbb{T}_{1}^{2}},\mathcal{O}_{\tilde{\mathcal{A}}_{\mathbb{T}_{1}^{2}}}\big)$
is contained in ${\rm can}\big(\mathcal{A}_{\mathbb{T}_{1}^{2}}\big)$.

Since $\tilde{\mathcal{A}}_{\mathbb{T}_{1}^{2}}$ is the space where
Fock--Goncharov duality conjecture holds not $\mathcal{A}_{\mathbb{T}_{1}^{2}}$,
it justifies our motto that to build the correct mirror dual cluster
variety, we should also include torus charts coming from negative
chambers in the scattering diagram.

\subsection*{Acknowledgements}

Linhui Shen first suggested to introduce folding into scattering diagrams.
I cannot thank him more. I am grateful to my advisor Sean Keel to
give suggestions for simplifying the quotient construction of scattering
diagrams. Besides I benefit from inspiring discussions with Andy Neitzke
and Daping Wen and email correspondence with Travis Mandel and Greg
Muller. I~also want to thank Andy Neitzke for carefully proofreading
the draft of this paper and thank anonymous referees for their numerous
suggestions for improvement.

\pdfbookmark[1]{References}{ref}
\LastPageEnding

\end{document}